\documentclass[11pt]{amsart}
\usepackage{hyperref}

\usepackage{enumerate}

\usepackage[utf8]{inputenc}
\usepackage[T1]{fontenc}

\usepackage[mathscr]{eucal}

\usepackage[a4paper, twoside=false, vmargin={2cm,3cm}, includehead]{geometry}

\newtheorem{lemma}{Lemma}[section]
\newtheorem{theorem}[lemma]{Theorem}
\newtheorem*{theorem*}{Theorem}
\newtheorem{corollary}[lemma]{Corollary}

\newtheorem{proposition}[lemma]{Proposition}
\newtheorem*{proposition*}{Proposition}

\newtheorem{problem}{Problem}

\newtheorem{claim}{Claim}

\theoremstyle{definition}
\newtheorem*{claim*}{Claim}
\newtheorem*{convention}{Convention}
\newtheorem*{notation}{Notation}
\newtheorem*{definition}{Definition}
\newtheorem*{equivdefinition}{Equivalent Definition}

\newtheorem*{examples}{Examples}
\newtheorem*{remark}{Remark}
\newtheorem*{remarks}{Remarks}

\newcommand{\fp}{\mathfrak{p}}

\newcommand{\C}{{\mathbb C}}
\newcommand{\E}{{\mathbb E}}

\newcommand{\N}{{\mathbb N}}
\renewcommand{\P}{{\mathbb P}}
\newcommand{\Q}{{\mathbb Q}}
\newcommand{\R}{{\mathbb R}}
\newcommand{\T}{{\mathbb T}}
\newcommand{\Z}{{\mathbb Z}}

\newcommand{\CA}{{\mathcal A}}
\newcommand{\CC}{{\mathcal C}}
\newcommand{\CF}{{\mathcal F}}
\newcommand{\CI}{{\mathcal I}}
\newcommand{\CM}{{\mathcal M}}
\newcommand{\CN}{{\mathcal N}}
\newcommand{\CP}{{\mathcal P}}

\newcommand{\CX}{{\mathcal X}}

\newcommand{\bg}{{\mathbf{g}}}
\newcommand{\bh}{{\mathbf{h}}}
\newcommand{\bk}{{\mathbf{k}}}
\newcommand{\bt}{{\mathbf{t}}}
\newcommand{\bu}{{\mathbf{u}}}
\newcommand{\bv}{{\mathbf{v}}}
\newcommand{\bx}{{\mathbf{x}}}

\newcommand{\balpha}{{\boldsymbol{\alpha}}}
\newcommand{\bepsilon}{{\boldsymbol{\epsilon}}}
\newcommand{\bgamma}{{\boldsymbol{\gamma}}}
\newcommand{\bdelta}{{\boldsymbol{\delta}}}
\newcommand{\bzero}{{\boldsymbol{0}}}
\newcommand{\btheta}{{\boldsymbol{\theta}}}
\newcommand{\beeta}{{\boldsymbol{\eta}}}
\newcommand{\bbeta}{{\boldsymbol{\beta}}}

\newcommand{\ve}{\varepsilon}
\newcommand{\wh}{\widehat}
\newcommand{\wt}{\widetilde}
\newcommand{\e}{\mathrm{e}}
\newcommand{\one}{\mathbf{1}}

\newcommand{\ZN}{\Z_{\widetilde N}}
\newcommand{\tN}{{\widetilde N}}
\newcommand{\CCM}{{\CM_1^c}}
\newcommand{\tZN}{\Z_\tN}
\newcommand{\norm}[1]{\lVert #1 \rVert}
\newcommand{\lip}{{\text{\rm Lip}}}
\newcommand{\inv}{^{-1}}
\newcommand{\st}{{\text{\rm st}}}
\newcommand{\un}{{\text{\rm un}}}
\newcommand{\er}{{\text{\rm er}}}

\DeclareMathOperator{\spec}{Spec}

\DeclareMathOperator{\Span}{Span}
\newcommand{\dmult}{d_\textrm{mult}}
\DeclareMathOperator{\reel}{Re}
\renewcommand{\Re}{\reel}
\DeclareMathOperator{\Imag}{Im}
\renewcommand{\Im}{\Imag}
\DeclareMathOperator{\poly}{poly}

\numberwithin{equation}{section}

\begin{document}

\title[Higher order Fourier analysis of multiplicative functions]{Higher order Fourier analysis of multiplicative functions and applications}

\author{Nikos Frantzikinakis}
\address[Nikos Frantzikinakis]{University of Crete, Department of mathematics, Voutes University Campus, Heraklion 71003, Greece} \email{frantzikinakis@gmail.com}
\author{Bernard Host}
\address[Bernard Host]{
Universit\'e Paris-Est Marne-la-Vall\'ee, Laboratoire d'analyse et
de math\'ematiques appliqu\'ees, UMR CNRS 8050, 5 Bd Descartes,
77454 Marne la Vall\'ee Cedex, France }
\email{bernard.host@univ-mlv.fr}

\begin{abstract}
We prove  a structure theorem for multiplicative
functions which
states that an arbitrary  multiplicative function  of
 modulus at most $1$ can be decomposed
into two terms, one that is approximately periodic and another that has
 small  Gowers uniformity norm of an arbitrary degree. The proof uses
  tools from higher order Fourier analysis and  finitary ergodic theory, and  some soft number theoretic input
 that comes in the form of an orthogonality criterion of K\'atai.
We use variants of this structure theorem  to derive applications of
number theoretic  and combinatorial
flavor: $(i)$  we give simple necessary and sufficient conditions
for  the Gowers norms (over $\N$) of a
bounded multiplicative function to be zero, $(ii)$
generalizing a classical result of Daboussi we prove asymptotic orthogonality of
multiplicative functions to ``irrational'' nilsequences,
 $(iii)$    we prove that for certain polynomials in two variables all ``aperiodic'' multiplicative functions
satisfy Chowla's zero mean conjecture,
$(iv)$ we give the first partition regularity results for homogeneous quadratic equations in three variables,
showing for example that on every partition of the integers into finitely many cells there exist distinct $x,y$ belonging to the same
cell and $\lambda\in \N$ such that $16x^2+9y^2=\lambda^2$ and the same holds for the equation $x^2-xy+y^2=\lambda^2$.
\end{abstract}

\thanks{The  first author was partially supported by
 Marie Curie IRG  248008. The second author was partially supported by Centro de Modelamiento Matem\'atico, Universitad de Chile.}



\date{\today}

\maketitle

\setcounter{tocdepth}{1}
\tableofcontents

\section{Introduction}
 A function $f\colon \N\to \C$ is called \emph{multiplicative} if
 $$
 f(mn)=f(m)f(n) \ \text{ whenever } \  (m,n)=1.
 $$
If multiplicativity holds for every $m,n\in \N$ we call $f$   \emph{completely multiplicative}.  We  denote by
 $\CM$ the set  of multiplicative functions of modulus at most $1$.

 The asymptotic behavior of averages of multiplicative functions
 is a central topic in analytic number theory that has been studied extensively.
In this article  we are interested  in studying the asymptotic behavior
of   averages  of the following form
\begin{equation}\label{E:Correlations}
\frac{1}{N^2}\sum_{1\leq m,n\leq N} \prod_{i=1}^{s}f (L_i(m,n)),
\end{equation}
where $f\in \CM$ is arbitrary and  $L_i(m,n)$, $i=1,\ldots,s$,  are linear forms with integer coefficients.
We are mainly motivated by  applications, perhaps the most surprising one being that there is a link between the aforementioned problem and partition regularity problems of non-linear  homogeneous equations in three variables; our methods enable us to address  some previously
intractable problems.

Two typical questions  we want to answer, stated somewhat imprecisely,  are as follows:
\begin{enumerate}
\item
\label{it:problem1}
 Can we impose ``soft'' conditions on  $f\in\CM$ implying that  the   averages \eqref{E:Correlations} converge to $0$ as $N\to+\infty$?
 \item
\label{it:problem2}
Is it always possible to replace $f\in \CM$ with a
 ``structured'' component $f_{\st}$,
such that  the averages
  \eqref{E:Correlations} remain unchanged, modulo a small error,  for all large $N$?
\end{enumerate}

The answer to both  questions is positive  in a very strong sense,
the necessary condition of question~\eqref{it:problem1} turns out to be extremely simple, we call it ``aperiodicity'',
and the structured component $f_{\st}$ that works for~\eqref{it:problem2} can be taken to be approximately periodic with approximate period independent of $f$ and $N$.

 For $s\leq 3$  both questions can be answered
by combining simple Fourier analysis tools on cyclic groups and
 a quantitative version of a classical result  of Daboussi~\cite{D74, DD74, DD82} which gives  information on the Fourier transform of a multiplicative function.
The key point is that for $s\leq 3$ the norm of the averages \eqref{E:Correlations}
can be controlled by the maximum of the Fourier coefficients of $f$
and the previous result  can be used to give satisfactory necessary and sufficient conditions so that this maximum converges to $0$ as $N\to+\infty$.

 For $s\geq 4$  it is impossible to control the norm of the averages \eqref{E:Correlations}
by the maximum of the Fourier coefficients of $f$  and
  classical Fourier analytic tools do not seem to facilitate  the study of these more complicated averages.
 To overcome this obstacle,  we  supplement our toolbox with some
  deep
 results from ``higher order Fourier analysis''; in particular,   the
 inverse theorem for the Gowers uniformity norms \cite{GTZ12c} and  the quantitative
 factorization of polynomial sequences  on nilmanifolds \cite{GT12a} play a prominent role. In an argument
that spans a substantial part of this article, these tools  are combined
with  an orthogonality criterion for multiplicative functions, and a delicate  equidistribution result on nilmanifolds, in order to prove
  a
 structure theorem for multiplicative functions. This structure theorem   is going to do
 the heavy lifting in answering questions \eqref{it:problem1}, \eqref{it:problem2}, and  in  subsequent applications;
 we mention here a variant
 sacrificing  efficiency for ease of understanding
    (more efficient variants  and the definition of the $U^s$-norms appear  in  Sections~\ref{subsec:UkIntro} and \ref{sec:decompUs}).
\begin{theorem}[Structure theorem for multiplicative functions I]
\label{T:DecompositionSimple}
Let $s\geq 2$ and  $\ve>0$.
There exist positive integers
  $Q:=Q(s,\ve)$ and  $R:=R(s,\ve)$,
  such that for every sufficiently large  $N\in \N$, depending on $s$ and $\ve$ only,   every
   $f\in \CM$ admits the decomposition
$$
f(n)=f_\st(n)+f_\un(n), \quad \text{for } \ n=1,\ldots, N,
$$
where  $f_\st$ and $ f_\un$ depend on $N$,  $|f_\st|\leq 1$, and
\begin{enumerate}
\item\label{E:AlmostPeriodic}
$\displaystyle |f_\st(n+Q)- f_\st(n)|\leq \frac{R}{N}$, \ \
for \ $n=1,\ldots,  N-Q$;
\item
$\norm{ f_\un}_{U^{s}(\Z_N)}\leq  \ve$.
\end{enumerate}
\end{theorem}

 A distinctive feature of Theorem~\ref{T:DecompositionSimple} is that it  applies to arbitrary
 bounded multiplicative functions. For this reason our  argument differs significantly from arguments in
 \cite{GT10b,GT12b, GTZ12c, La14a, La14b, Mat12a, Mat15, Mat12b, Mat13, Mat14}, where pseudorandomness properties of the M\"obius and other arithmetical functions are exploited. For instance, the lack of  effective estimates that could be used to treat the ``minor arc'' part of our argument renders  the method of \cite{GT12b} inapplicable and necessitates the introduction of new tools. These new  ideas   span Sections~\ref{S:U^2}, \ref{S:minorarcs1}, \ref{S:minorarcs2}, \ref{sec:decompUs}
  and are properly explained in the course of this article.
Another important feature  of Theorem~\ref{T:DecompositionSimple} is that the structured component $ f_\st$ is always approximately periodic  and that its
  approximate period   is independent of $ f $ and $N$. In fact we show that $ f_\st$
    is a convolution product of $ f $ with a kernel on $\Z_\tN$ ($\tN\geq N$ is a prime)
      that does not depend on $ f $ and the cardinality of its spectrum depends only on
   on $s$ and $\ve$.
    All these
   properties turn out to be  very crucial for subsequent applications.
Note that  for arbitrary bounded sequences, decomposition results with
similar flavor have been proved in \cite{G10, GW11, GT10,
  Sz12, T06}, but in order to  work in this generality one is forced to use
a structured component that  does  not  satisfy the strong
rigidity condition in~\eqref{E:AlmostPeriodic};  the best that can be said is that it is   an $(s-1)$-step  nilsequence
of bounded complexity. This property is much weaker than \eqref{E:AlmostPeriodic} even when $s=2$ and insufficient for our applications.   Similar comments apply
 for analogous decomposition results for infinite sequences \cite{HK09} that were motivated by structural results in ergodic theory \cite{HK05}.

  Despite its clean and succinct form, Theorem~\ref{T:DecompositionSimple} turns out to be    difficult to prove.  The main ideas are sketched in  Sections~\ref{subsec:decomposition}, \ref{SS:ideasdisc}, \ref{subsec:prelim}; furthermore,  Proposition~\ref{prop:baby} provides a  toy model
of the much more complicated general case.

  We remark that although explicit use of ergodic theory is not made anywhere in the proof of Theorem~\ref{T:DecompositionSimple} and its variants,
ergodic structural results and dynamical properties of sequences on nilmanifolds have guided some of our arguments.

Next, we give some representative examples  of the
 applications that we are going to derive from variants of Theorem~\ref{T:DecompositionSimple}. Again, we sacrifice generality   for ease of understanding; the precise statements of the more general   results
 appear in the next section.

\subsection*{Partition regularity of quadratic equations}
Since the theorems of Schur and van der Waerden, numerous partition
regularity results have been  proved for linear equations, but
progress has been scarce for non-linear ones, the hardest case being
equations in three variables.  We prove partition regularity for certain equations
involving
 quadratic forms in three variables. For example, we show in Corollary~\ref{Corol1} that
  for
every  partition of $\N$ into finitely many cells, there exist distinct $x,y$ belonging to the same cell and $\lambda\in \N$
such that $
16x^2+9y^2=\lambda^2.$
Similar results hold for the equation $x^2-xy+y^2=\lambda^2$ and in much greater generality  (see Theorems~\ref{th:partition-regular1} and \ref{th:partition-regular3}). We actually prove
stronger density statements from which the previous partition regularity results follow.

\subsection*{Uniformity of multiplicative functions}
In Theorem~\ref{th:aperiod_uniform} we show   that for $s\geq 2$, for every $ f \in \CM$ we have
$$
\lim_{N\to+\infty} \norm{ f }_{U^s(\Z_N)}=0
\text{ if and only if }
 \lim_{N\to+\infty} \frac{1}{N}\sum_{n=1}^N f (an+b)=0 \ \text{ for every } a,b\in \N.
$$
Furthermore,   using a result of  Hal\'asz (see Theorem~\ref{T:Halasz}), it is easy to
 recast the second condition as a  simple statement that is easy to verify or refute for explicit multiplicative functions (see Property~\eqref{it:ChiSeries} of Proposition~\ref{prop:equiv-aperiodic}).

\subsection*{A generalization  of a result of Daboussi}

A classical result of Daboussi~\cite{D74, DD74, DD82} states that \begin{equation}\label{E:Daboussi}
\lim_{N\to+\infty} \sup_{ f \in \CM}\Big|\frac{1}{N}\sum_{n=1}^N f (n) \ \! e^{2\pi in\alpha}\Big|=0 \ \text{ for every }\  \alpha\in \R\setminus \Q.
\end{equation}
K\'atai~\cite{K86} showed that the same thing holds  if $e^{2\pi in\alpha}$ is replaced by $e^{2\pi i p(n)}$ where  $p(n)=\alpha_1n+\cdots+\alpha_dn^d$ has at least one coefficient irrational. In Theorem~\ref{T:DaboussiGeneral}
 we generalize this even further to cover sequences induced by totally equidistributed polynomial sequences on nilmanifolds. One such  example is the sequence  $e^{2\pi i [n\sqrt{2}]n\sqrt{3}}$.

\subsection*{ A variant of Chowla's conjecture} A classical conjecture of Chowla \cite{Ch65} states that if
$\lambda$ is the Liouville function and $P\in \Z[x,y]$
is a homogeneous polynomial  such that $P\neq cQ^2$ for every  $c\in \Z$, $Q\in\Z[x,y]$,  then
\begin{equation}\label{E:Chowla}
 \lim_{N\to +\infty} \frac{1}{N^2}\sum_{1\leq m,n\leq N}\lambda(P(m,n))=0.
\end{equation}
  This was established by  Landau when  $\deg(P)=2$  \cite{La1918} (see also \cite{Hel03}), by Helfgott when
$\deg(P)=3$ \cite{Hel06, Hel13}, and   when $P$ is a product of pairwise independent linear forms
by Green, Tao, and Ziegler \cite{GT10b, GT12a, GT12b, GTZ12c}.  The conjecture is also closely related to the problem of representing primes by irreducible polynomials, for relevant work see \cite{FI98, He01,HeM02, HeM04}.
In Theorem~\ref{th:chowla} we show that if
$$
P(m,n) :=(m^2+n^2)^r \prod_{i=1}^s L_i(m,n),
 $$ where $r\geq 0$, $s\in\N$, and $L_i$ are pairwise independent linear forms with integer coefficients, and if
 $ f \in \CM$ is completely multiplicative and  aperiodic, meaning, it
   averages to zero on every infinite arithmetic progression, then
 $$
  \lim_{N\to +\infty} \frac{1}{N^2}\sum_{1\leq m,n\leq N} f (P(m,n))=0.
  $$
As a consequence, for aperiodic multiplicative functions $ f $ and pairwise independent linear forms,  the averages \eqref{E:Correlations} converge to $0$ as $N\to +\infty$. Note, that even in the case where $r=0$  our result is new, as
it applies to arbitrary aperiodic multiplicative functions of modulus at most  $1$, not just the M\"obius or the Liouville.

\medskip

In the next section we give a more precise formulation of our main results and also define
some of the concepts used throughout the article.

\subsection*{Acknowledgements} We would like to thank Wenbo Sun for pointing out a mistake  in an earlier  version of this article and  Nikos Tzanakis for his help with the material in Section~\ref{SS:Higher}.  We would also like to thank the referee for helpful suggestions and for pointing us to Chebotarev's theorem that helped to strengthen Theorem~\ref{th:chowla}.

\section{Precise statements of the main results}
\subsection{Structure theorem for multiplicative functions}
\label{subsec:UkIntro}
 Roughly speaking, our  main structure theorem asserts that an arbitrary multiplicative
function of modulus at most $1$ can be split  into two components, one that is approximately periodic,  and
another that behaves randomly enough to have a negligible
contribution for the averages we are interested in handling. For our purposes,
randomness
is  measured by the Gowers uniformity norms.
 Before proceeding to
the precise statement of the structure theorem,  we start with some
discussion regarding the Gowers uniformity norms and the
uniformity properties (or lack thereof) of multiplicative functions.
\subsubsection{Gowers uniformity norms}\label{SS:Gowers}
For  $N\in\N$ we let $\Z_N:=\Z/N\Z$ and $[N]:=\{1,\dots,N\}$. These sets are often identified in the obvious way, but arithmetic operations performed on them behave differently. If $x,y$ are considered as elements of $[N]$, expressions like
$x+y$, $x-y$, $\ldots$,   are computed in $\Z$.
If $x,y$  are considered as elements of $\Z_N$,  $x+y$, $x-y$, \dots, are computed modulo $N$ and are elements of $\Z_N$.

If $A$ is a finite set and $a\colon A\to \C$ is a function, we write
$$
\E_{x\in A}a(x):=\frac 1{|A|}\sum_{x\in A}a(x).
$$
The same notation is used for a function of several variables.
We recall the definition of the    $U^s$-Gowers uniformity
norms from \cite{G01}.
\begin{definition}[Gowers norms on a cyclic group~\cite{G01}]
Let $N\in \N$  and $a\colon \Z_N\to \C$. For $s\in \N$ the \emph{Gowers $U^s(\Z_N)$-norm} $\norm a_{U^s(\Z_N)}$ of $a$ is defined inductively as follows:   For every $t\in\Z_N$ we write $a_t(n):=a(n+t)$. We let
$$
\norm a_{U^1(\Z_N)}:=|\E_{n\in\Z_N}a(n)|
=\Bigl(\E_{x,t\in\Z_N} a(x)\, \overline a(x+t)\Bigr)^{1/2},
$$
and for every $s\geq 1$ we let
\begin{equation}
\label{eq:def-gowers}
\norm a_{U^{s+1}(\Z_N)}:=\Bigl(\E_{t\in\Z_N}\norm{a\cdot \overline a_t}_{U^s(\Z_N)}^{2^s}\Bigr)^{1/2^{s+1}}.
\end{equation}
\end{definition}
For example,
$$
\norm a_{U^2(\Z_N)}^4=\E_{x,t_1,t_2\in\Z_N} a(x)\, \overline a(x+t_1)\, \overline a(x+t_2)\, a(x+t_1+t_2)
$$
and a similar  closed formula can be given for the $U^s(\Z_N)$-norms for $s\geq 3$. It can be shown that $\norm\cdot_{U^s(\Z_N)}$ is a norm for $s\geq 2$ and
for  every $s\in \N$ we have
\begin{equation}\label{E:UkIncreases}
\norm a_{U^{s+1}(\Z_N)}\geq\norm a_{U^s(\Z_N)}.
\end{equation}
In an informal way, having a small $U^s$-norm is interpreted as a
property of $U^s$-uniformity, and we say that a function or sequence of functions
is
$U^s$-uniform if
 the corresponding  uniformity norms converge to $0$ as $N\to+ \infty$.
By \eqref{E:UkIncreases}
we get that  $U^{s+1}$-uniformity implies $U^s$-uniformity.

Recall that the \emph{Fourier  transform} of a function $a$ on $\Z_N$ is given by
$$
\wh a(\xi):=\E_{n\in\Z_N}a(n)\, \e\bigl(-n\frac\xi N\bigr)\ \text{ for }\ \xi\in\Z_N,
$$
where, as is standard, $\e(x):=\exp(2\pi ix)$.  A direct computation
gives the following identity that links the $U^2$-norm of a function
$a$ on $\Z_N$ with its Fourier coefficients:
\begin{equation}
\label{eq:U2Fourier}
\norm a_{U^2(\Z_N)}=\norm{\wh a}_{\ell^{4}(\Z_N)}:=\Bigl(\sum_{\xi\in\Z_N}\bigl|\wh a(\xi)\bigr|^{4}\Bigr)^{1/4}.
\end{equation}
It follows that, if $|a|\leq 1$, then
\begin{equation}
\label{eq:U2Fourier2}
 \norm a_{U^2(\Z_N)}
\leq\sup_{\xi\in\Z_N}|\wh a(\xi)|^{1/2}
\leq  \norm a_{U^2(\Z_N)}^{1/2}.
\end{equation}
We would like to stress though that similar formulas
and estimates do not exist for higher order Gowers norms; a
function bounded by $1$  may have small  Fourier coefficients, but large
$U^s(\Z_N)$-norm for $s\geq 3$. In fact, eliminating all possible obstructions to
$U^s(\Z_N)$-uniformity necessitates the study of correlations with all
 polynomial phases $\e(P(n))$, where $P\in \R[x]$ has degree $s-1$,
  and also the larger class of $(s-1)$-step nilsequences
 of bounded complexity (see Theorem~\ref{th:inverse}).

For the purposes of this article it will be convenient to also define Gowers norms on  an interval $[N]$
(this was also done in \cite{GT10b}).
For $N^*\geq N$ we often identify the interval $[N^*]$ with $\Z_{N^*}$ in which case we  consider $[N]$ as a subset of $\Z_{N^*}$.

\begin{definition}[Gowers norms on an interval~\cite{GT10b}]
\label{def:Us-interv}
 Let $s\geq 2$, $N\in \N$, and $a\colon [N]\to \C$ be a function. By Lemma~\ref{cl:gowersMN} in Appendix~A,
the quantity
$$
\norm a_{U^s[N]}:=\frac 1{\norm{\one_{[N]}}_{U^s(\Z_{N^*})}}\; \norm{\one_{[N]}\cdot a}_{U^s(\Z_{N^*})}
$$
is independent of $N^*$ provided that $N^*>2N$. It is called \emph{the $U^s[N]$-norm of $a$}.
\end{definition}

In complete analogy with the  $U^s(\Z_N)$-norms
we have that $\norm \cdot_{U^s[N]}$ is a norm that
 increases with $s$, in the sense that for every $s\geq 2$ there exists a constant $c:=c(s)$ such that $\norm{\cdot}_{U^{s+1}[N]}\geq c\, \norm{\cdot}_ {U^s[N]}$.
In Appendix~A we derive various relations between the $U^s(\Z_N)$
and the $U^s[N]$ norms for $s\geq 2$. In particular, in
Lemma~\ref{lem:NormsUs} we show that if $|a|\leq 1$, then
$\norm a_{U^s(\Z_N)}$  is small if and only if  $\norm a_{U^s[N]}$
is small.  Thus,
the reader should think of the  $U^s[N]$ and $U^s(\Z_N)$ norms as equivalent measures of randomness;  which one we  use
is a matter of convenience and depends on the particular   problem at hand.

\subsubsection{Multiplicative functions}
Some examples of multiplicative functions of modulus at most $1$ are the M\"obius and the Liouville function,
the function $n\mapsto n^{it}$ for $t\in \R$, and  Dirichlet characters, that is,  periodic completely  multiplicative functions that are not identically zero.
Throughout, we denote by $ \chi_q$ a Dirichlet character of least period $q$. Then
$\chi_q(n)=0$ whenever $(n,q)>1$ and $\chi_q(n)$ is a $\phi(q)$-root of unity if $(n,q)=1$, where $\phi$ is Euler's totient function.

It follows from results  in~\cite{GT12b, GTZ12c} that the M\"obius and the Liouville function
are $U^s$-uniform for every $s\in \N$.
 The next examples illustrates some simple but very
important obstructions to uniformity for general bounded multiplicative functions.
\begin{examples}[Obstructions to uniformity]
$(i)$ One easily sees that $\E_{n\in [N]}n^{it}\sim c_N:=\frac{N^{it}}{1+it}$, hence for $t\neq 0$  the range of this average
is contained densely  in the circle with center at zero and radius $1/\sqrt{1+t^2}$.
Therefore,
there is no constant $c$, independent of $N$, so that the function $n^{it}-c$ averages to $0$ on $\N$.
On the other hand,  the average of $n^{it}-c_N$ on the interval $[N]$  converges to $0$ as $N\to+ \infty$,
and in fact it can be seen\footnote{One can see this by using Theorem~\ref{T:DecompositionI} and adjusting the argument used to prove Theorem~\ref{th:aperiod_uniform}.} that $(n^{it}-c_N)_{n\in [N]}$, $N\in \N$,  is $U^s$-uniform for every $s\geq 2$.

$(ii)$
 A non-principal
  Dirichlet character
$ \chi_ q$
has average $0$ on every interval with length a multiple
 of $q$,
hence $\E_{n\in [N]} \chi_ q(n)\to 0$ as $N\to +\infty$. However, $ \chi_ q$ is not $U^2$-uniform  because it is
 periodic.

$(iii)$
Let $ f $ be the completely multiplicative function  defined by
$ f (2):=-1$ and $ f (p):=1$ for
every prime $p\neq 2$.
Equivalently,
 $ f (2^m(2k+1))=(-1)^m$ for
all $k,m\geq 0$.
Then $\E_{n\in [N]} f (n)= 1/3+o(1)$ and this
non-zero mean value already gives an obstruction to $U^2$-uniformity. But
this is not the only obstruction. We have $\E_{ n\in [N]}(-1)^n( f (n)-1/3)=-2/3+o(1)$,
  which   implies that  $ f -1/3$ is not $U^2$-uniform. In fact, it is not possible to subtract from $ f $   a
 periodic component $ f_ \st$ that is independent of $N$  so that
  $ f - f_ \st$ becomes $U^2$-uniform.
 But this problem is alleviated if we allow $ f_ \st$ to  depend on $N$.
\end{examples}
The first and third  examples  illustrate that the structured component we need to subtract from
 a multiplicative function so that the difference has small $U^s(\Z_N)$-norm may vary a lot with $N$.
 This is one of the reasons why we cannot obtain an infinite
variant of the  structural result of  Theorem~\ref{T:DecompositionSimple}.
The last two examples illustrate that normalized
multiplicative functions can have significant correlation with
  periodic phases;   thus this is an
  obstruction to
$U^2$-uniformity that we should  take   into account.
However, it  is a non-trivial fact that plays a central role in this
article, that correlation with periodic phases is, in a sense to
be made precise later, the only obstruction
 not only to $U^2$-uniformity but also to $U^s$-uniformity of multiplicative functions in $\CM$  for all $s\geq 2$.

\subsubsection{The main  structure theorem}
\label{subsec:decomposition}
The structural result of Theorem~\ref{T:DecompositionSimple} suffices for some applications and
a more informative variant is given  in Theorem~\ref{T:DecompositionI}. But both results are not well suited  for the combinatorial applications given in Section~\ref{SS:preg}. The reason is  that in such problems we seek to obtain positive lower bounds for certain   averages of multiplicative functions,
and the error introduced by the uniform component typically
subsumes the positive lower bound coming from  the structured component
 as this depends on $\ve$ (via $Q$ and $R$)  in a rather  inexplicit way.
In order to overcome this obstacle, we  would like to know that the uniformity norm of the uniform component
can be chosen to be smaller than any predetermined positive function of $Q$ and $R$. This can be achieved if we
 introduce an additional term that has small $L^1[N]$-norm. We give the precise form of  such a  structural result
  after we set up some notation.

As it is often easier to work on a cyclic group rather  than  an
interval of integers (this makes Fourier analysis tools more
readily available) we introduce some notation to help us avoid
 roundabout issues.
\begin{notation}
 Throughout, we assume that an integer $\ell\geq 2$ is given. This parameter adds some flexibility needed
  in the applications   of our main structural results;  its precise value  will depend on the particular application
we have in mind.
 We consider $\ell$  as fixed and the dependence on $\ell$ is always left implicit.
For every $N\in \N$, we denote by  $\wt N$  any prime such that $N\leq \tN\leq \ell N$.  By Bertrand's
postulate, such a prime always exists.
In some cases we  specify the  value of $\tN$ and its precise dependence on $N$ depends on the application we have in mind.
\end{notation}

 For every multiplicative
function $ f \in \CM$ and every $N\in\N$, we denote by
 $ f_ N$ the function on $\Z_\tN$, or on $[\tN]$, defined by
\begin{equation}
\label{eq:def-chiN}
 f_ N(n):=\begin{cases}  f (n)& \text{if }n\in[N];\\
0&\text{otherwise.}
\end{cases}
\end{equation}
Each time the domain of $ f_ N$ will be  clear from the context.
Working with the truncated function $ f \cdot \one_{[N]}$, rather
than the function $ f $, is a technical maneuver and the reader
will not lose  much by ignoring the cutoff. We should stress  that
for the purposes of the structure theorem, $U^s$-norms are going to be defined and Fourier
analysis is going to happen on the group $\Z_\tN$ and not on the
group $\Z_N$.

\begin{definition}
By a \emph{kernel} on $\ZN$ we mean a non-negative function with
average $1$.
\end{definition}

In the next statement we assume that the set $\CM$
 is endowed with the topology of pointwise
convergence and thus is a compact metric space.

\begin{theorem}[Structure theorem for multiplicative functions II]\footnote{When $s=3$, W.~Sun \cite{S15} recently proved a similar result for multiplicative functions defined on the Gaussian integers.}
\label{T:DecompositionII}
Let $s\geq 2$, $\ve>0$,
$\nu$ be a probability measure on the compact set $\CM$,
 and  $F\colon\N \times\N \times \R^+\to\R^+$ be arbitrary.
Then
 there exist positive integers
  $Q$ and  $R$ that are  bounded by a constant which depends only on
  $s$, $\ve$,  $F$,
  such that the following holds:  For every
sufficiently large  $N\in \N$, which depends only on $s$, $\ve$, $F$, and for
  every $ f \in\CM$, the function $ f_ N$ admits the
 decomposition
$$
 f_ N(n)= f_ {N,\st}(n)+ f_ {N,\un}(n)+ f_ {N,\er}(n) \quad \text{ for
every }\  n\in\tZN,
$$
where  $ f_ {N,\st}$, $ f_ {N,\un}$,
   $ f_ {N,\er}$ satisfy the following properties:
\begin{enumerate}
\item
\label{it:decomU31}
$ f_ {N,\st}= f_ N*\psi_{N,1}$ and  $ f_ {N,\st}+ f_ {N,\er}= f_ N*\psi_{N,2}$,
where  $\psi_{N,1}$, $\psi_{N,2}$ are kernels on $\Z_\tN$ that do not depend on $ f $, and
  the convolution product is  defined in $\tZN$;

\item\label{it:decompU32}
$\displaystyle | f_ {N,\st}(n+Q)- f_ {N,\st}(n)|\leq \frac R\tN$
for every $n\in\tZN$,  where  $n+Q$ is taken $\!\!\! \mod \tN$;
\item\label{it:weakU32b}
If $\xi \in \Z_\tN$ satisfies $\widehat{ f }_{N,\st}(\xi)\neq 0$, then $\displaystyle \big|\frac{\xi}{\tN}-\frac{p}{Q}\big|\leq \frac{R}{\tN}$ for some   $p\in \{0,\ldots Q-1\}$;
\item
\label{it:decomU33}
$\displaystyle \norm{ f_ {N,\un}}_{U^s(\tZN)}\leq\frac 1{F(Q, R,\ve)}$;
\item
\label{it:decomU34}
 $\displaystyle  \E_{n\in \tZN} \int_\CM | f_ {N,\er}(n)|\,d\nu( f )\leq\ve$.
\end{enumerate}
\end{theorem}
\begin{remarks}
 (1) The result is of interest even when $\nu$ is supported on a single multiplicative function, that is, when $f\in \CM$ is fixed and
 Property~\eqref{it:decomU34} is replaced with the estimate
$\displaystyle\E_{n\in \tZN}| f_ {N,\er}(n)|\leq\ve$.
  The
stronger version stated is needed for the combinatorial applications.

\smallskip

(2)  As remarked in the introduction, various  decomposition results with
similar flavor have been proved  for arbitrary bounded sequences
but working in this generality necessitates the
use of structured components that are much less rigid.
  An additional important feature of our result
  is  that the structured component  is defined by a
convolution product with a kernel that is independent of $ f\in \CM$. All
these properties
play an important role in the derivation of some of our applications.

\smallskip

(3)
In Section~\ref{S:Recurrence} we  use Theorem~\ref{T:DecompositionII}  for the function $F(x,y,z):=c
x^2y^2/z^4$ where $c$ is a constant that depends on $\ell$ only.
Restricting the statement to this function though does not simplify
our proof.

\smallskip

(4)
It is a consequence of Property~\eqref{it:decomU31} that for fixed $F,N,  \ve, \nu$, the
maps $ f \mapsto  f_ \st,  f \mapsto f_ \un,  f \mapsto
 f_ \er$ are continuous, and
 $| f_ \st|\leq 1$,
 $| f_ \un|\leq 2$,  $| f_ \er|\leq 2$.

\smallskip

(5)  We do not know if a uniform version of the result holds, meaning, with
Property~\eqref{it:decomU34} replaced with
$\displaystyle
\sup_{ f \in\CM} \E_{n\in \tZN}| f_ {N,\er}(n)|\leq\ve$.
\end{remarks}

The bulk of the work  in the proof of Theorem~\ref{T:DecompositionII} goes in the proof of
Theorem~\ref{T:DecompositionI} below which is a more informative version of
Theorem~\ref{T:DecompositionSimple} given in the introduction.
Two ideas that play a prominent role in the proof,
roughly speaking, are:
\begin{enumerate}[(a)]
\item
   A multiplicative function that has $U^2$-norm bounded away
from zero correlates with a linear  phase that has frequency close
to a rational with small denominator.
\smallskip

\item
  A multiplicative function that has $U^s$-norm bounded away
from zero necessarily has $U^2$-norm bounded away from zero.
\end{enumerate}

The proof of (a) uses classical  Fourier analysis tools and  is
given in Section~\ref{S:U^2}. The key number theoretic input is the
orthogonality criterion of K\'atai stated in Lemma~\ref{lem:katai}.
\smallskip

  The proof of (b) is much harder and is done in
several steps using higher order Fourier analysis machinery. In
Sections~\ref{S:minorarcs1} and \ref{S:minorarcs2}
 we study the correlation of multiplicative functions with
totally equidistributed (minor arc) nil-sequences. This is the heart of the matter and the technically more demanding part in the proof of  the
 structure theorem. The argument used by Green and Tao in \cite{GT12b}  to prove similar estimates for  the M\"obius function uses special features of the M\"obius
 and  is inadequate for our purposes.
To overcome this serious obstacle, we combine the orthogonality criterion of  K\'atai
  with a rather  delicate asymptotic orthogonality result of polynomial nilsequences in order to
  establish a key  discorrelation estimate (Theorem~\ref{th:discorrelation}).
  This estimate is  then used
 in Section~\ref{sec:decompUs}, in conjunction with   the $U^s$-inverse theorem  (Theorem~\ref{th:inverse}) and
a factorization theorem
  (Theorem~\ref{th:FactoGT}) of Green and Tao, to
conclude the proof of
Theorem~\ref{T:DecompositionI}.
We defer the reader to
Sections~\ref{SS:ideasdisc} and \ref{subsec:prelim} for a more detailed sketch of the proof
strategy of Property (b).

  Upon proving
Theorem~\ref{T:DecompositionI}, the proof of
Theorem~\ref{T:DecompositionII}  consists of a Fourier
analysis energy increment argument, and  avoids the use of finitary
ergodic theory  and  the Hahn-Banach theorem, tools that
are typically used  for other decomposition results (see \cite{G10,
GW11,GW11b, GT10, T06}). This hands-on approach enables us to
transfer all the information obtained in Theorem~\ref{T:DecompositionI}
which is important for our applications.

\begin{problem}
Can Theorems~\ref{T:DecompositionII} and \ref{T:DecompositionI} be extended to multiplicative functions defined on quadratic number fields? More general number fields?
\end{problem}

\subsection{ A generalization  of a result of Daboussi}

The classic result of  Daboussi~\cite{D74, DD74, DD82}  was recorded in the introduction (see \eqref{E:Daboussi}).
We prove the following generalization (all notions used below are defined in Sections~\ref{sec:Nil} and \ref{sec:equidistribution}):
\begin{theorem}[Daboussi for nilsequences]\label{T:DaboussiGeneral}
Let $X:=G/\Gamma$ be a nilmanifold and  $(g(n))_{n\in\N}$ be a polynomial sequence in $G$  such that
$(g(n)\cdot e_X)_{n\in\N}$ is totally equidistributed in $X$. Then for every  $\Phi\in C(X)$ with $\int_X \Phi\, dm_X=0$  we have
\begin{equation}\label{E:DaboussiGeneral'}
\lim_{N\to+\infty} \sup_{ f \in \CM}\Big|\frac{1}{N}\sum_{n=1}^N f (n) \ \! \Phi(g(n)\cdot e_X)\Big|=0.
\end{equation}
\end{theorem}
\begin{remarks}
(1) Theorem~\ref{T:DaboussiGeneral} follows from the stronger finitary  statement  in Theorem~\ref{th:discorrelation}.

(2) If $P$ is a polynomial with at least one non-constant coefficient irrational, then $(P(n))_{n\in \N}$ is totally equidistributed on the circle. Hence, for such a polynomial, $\e(P(n))$ can take the place of  $\Phi(g(n)\cdot e_X)$
in \eqref{E:DaboussiGeneral'},   recovering a result of K\'atai~\cite{K86}.

(3) An easy approximation argument allows to extend the eligible functions $F$ to all Riemann
integrable functions on $X$.
We can use this enhancement to show that   any sequence of the form $\e(2\pi i [n\alpha]n\beta)$ with $1, \alpha, \beta$ rationally independent over $\Q$ can take the place of $\Phi(g(n)\cdot e_X)$
in \eqref{E:DaboussiGeneral'}.
\end{remarks}

\subsection{Aperiodic multiplicative functions}\label{SS:aperiodic}
It is  known that the M\"obius and the Liouville functions have zero average on every infinite arithmetic progression. In this subsection we work with the following vastly more general class of multiplicative functions:
\begin{definition}
We say that a multiplicative function $ f \colon \N\to \C$ is \emph{aperiodic} if it has zero  average   on
every infinite arithmetic progression, that is,
$$
\lim_{N\to + \infty} \E_{n\in [N]}  f (an+b)=0, \quad \text{for every } a,b\in \N.
$$
\end{definition}

In order to give easy to check conditions that imply aperiodicity, we will use
a celebrated result of    Hal\'asz~\cite{Hal68}.
To facilitate  exposition we first define the distance between two multiplicative functions (see for example   \cite{GranSound07,GS15}).
\begin{definition}
If $f,g\in \CM$ we let $\mathbb{D}\colon \CM\times \CM\to [0,\infty]$ be given by
$$
\mathbb{D}(f,g)^2=\sum_{p\in \P} \frac{1}{p}\,\bigl(1-\Re\bigl(f(p)\,  \overline{g}(p)\bigr)\bigr)
$$
where $\P$ denotes the set of prime numbers.
\end{definition}
\begin{remark}
Note that if $|f|=|g|=1$, then
$
\mathbb{D}(f,g)^2=\sum_{p\in \P} \frac{1}{2p}\, |f(p)- g(p)|^2.
$
\end{remark}

\begin{theorem}[Hal\'asz~\cite{Hal68}]
\label{T:Halasz}
A multiplicative function  $f\in \CM$ has mean value zero if and only if
for every $t\in \R$  we either have $\mathbb{D}(f,n^{it})=\infty$
or
$
f(2^m)= -2^{imt}$ for all $m\in \N$.
\end{theorem}
We record several  conditions equivalent to aperiodicity; the last one is the easiest  to verify for explicit multiplicative functions.
\begin{proposition}
\label{prop:equiv-aperiodic}
For a  multiplicative function $ f \in \CM$ the following are equivalent:
\begin{enumerate}
\item
\label{it:aperiodic}
 $ f $ is aperiodic.
\item
\label{it:chi-npq}
For every $p,q\in \N$,we have $\displaystyle \lim_{N\to +\infty} \E_{n\in [N]}  f (n)\, \e(np/q)=0$.
\item
\label{it:chi-chi}
For every Dirichlet character $ \chi_ q$ we have $\displaystyle\lim_{N\to +\infty} \E_{n\in [N]}  f (n) \,
\chi_q(n)=0$.
\item
\label{it:ChiSeries}
 For every $t\in \R$ and Dirichlet character $ \chi_ q$ we either have
 $\mathbb{D}(f,\chi_ q(n)n^{it})=\infty$ or
$f (2^m)  \, \chi_ {q}(2^m) =-2^{-imt}$ for every $m\in \N.$
\end{enumerate}
\end{proposition}
Assuming Theorem~\ref{T:Halasz},
the proof of the equivalences is simple (this was already observed in \cite{D74, DD74, DD82}); we give it for the convenience of the reader
  in Section~\ref{subsec:proof-aperiodic}.
Lending terminology from \cite{GranSound07, GS15},
 condition~\eqref{it:ChiSeries} states that  a multiplicative function is aperiodic unless it ``pretends'' to be $ \chi_ q(n) n^{it}$ for some Dirichlet character $ \chi_ q$ and some $t\in \R$.
It follows easily from \eqref{it:ChiSeries} that if $ f \in \CM$ has real values and satisfies
$\sum_{p\in \P\cap (d\Z+1)}\frac{1- f (p) }{p}=+\infty$
for every $d\in \N$, then $ f $ is aperiodic.
In particular, this is satisfied by the  M\"obius and the Liouville functions.
Sharper results can be obtained using a theorem of R.~Hall~\cite{Hall95} and the argument in \cite[Corollary~2]{GranSound07}.
For instance, it can be shown that if  $f(p)$  takes values in a finite subset of the unit disc  and $f(p)\neq 1$ for all
$p\in \P$,
    then $f$ is aperiodic.

\subsubsection{Uniformity of aperiodic functions} \label{SS:aperunif} We give explicit necessary and sufficient conditions
for a multiplicative function
 $ f \in \CM$ to be $U^s$-uniform, that is,  have $U^s[N]$-norm converging to zero as $N\to +\infty.$

Aperiodicity is easily shown to be a necessary  condition for $U^2$-uniformity. For general bounded sequences it is far from sufficient   though. For instance,  the sequences $(\e(n\alpha))_{n\in\N}$  and  $(\e(n^2\alpha))_{n\in\N}$, where $\alpha$ is irrational, are aperiodic, but the first is not $U^2$-uniform and
the second is $U^2$-uniform but  not $U^3$-uniform. It is a rather surprising (and non-trivial) fact that for the general multiplicative function in $\CM$ aperiodicity suffices  for $U^s$-uniformity for every $s\geq 2$.
This is not hard to show  for $s=2$ by combining well known results about multiplicative functions,
but  for $s\geq 3$  it  is much harder to do so,  and we need to use  essentially the full force of  Theorem~\ref{T:DecompositionSimple}.
\begin{theorem}[$U^s$-uniformity of aperiodic multiplicative functions]
\label{th:aperiod_uniform}
If a multiplicative function $ f \in \CM$ is aperiodic, then
$\displaystyle \lim_{N\to+\infty}\norm f_ {U^s[N]}=0$ for every $s\geq 2$.
\end{theorem}

\begin{remarks}
(1) Our proof gives the following finitary inverse theorem:
For given  $s\in \N$ and $\ve>0$,  there exists $\delta:=\delta(s,\ve)>0$ and $Q:=Q(s,\ve)\in \N$,
such that if $ f \in \CM$ satisfies $\limsup_{N\to+\infty}\norm f_ {U^s[N]}\geq \ve$, then there
exists  $a,b\in \N$ with $1\leq a,b\leq Q$
such that
$\limsup_{N\to+\infty}\big|\E_{n\in [N]}  f (an+b) \big|\geq \delta.
$ Note that $\delta$ and $Q$ do not depend on $ f $.

(2)  Theorem~\ref{th:aperiod_uniform} implies that if $ f $ is an
aperiodic multiplicative function, then $ f $ does not correlate with any polynomial
phase function $\e(p(n))$, $p\in \R[t]$. More generally, it implies that  $\lim_{N\to+\infty}\E_{n\in [N]}
f(n)\, \phi(n)=0$
for every nilsequence $(\phi(n))_{n\in\N}$. For the M\"obius and the Liouville function  this result was  obtained by Green and Tao in \cite{GT12b}.
\end{remarks}

\subsubsection{Chowla's conjecture for aperiodic multiplicative functions}
\label{SS:chowla}

We provide a class of  homogeneous polynomials $P\in \Z[m,n]$ such that
$\E_{m,n\in[N]}f(P(m,n))\to 0$ for every aperiodic completely multiplicative $f\in \CM$.

Here and below, $d$ is a positive integer, and by  $\sqrt{-d}$ we mean $i\sqrt{d}$.
It is well known that the ring of integers of $\Q(\sqrt {-d})$ is equal to $\Z[\tau_d]$ where
$$
\tau_d:=\frac 12(1+\sqrt{-d})\  \text{ if }\ d= 3 \bmod 4\ \text{ and }\ \tau_d:=\sqrt{-d}\ \text{ otherwise.}
$$
The norm of $z\in\Z[\tau_d]$ is $\CN(z)=|z|^2$.
We write $Q_d(m,n)$ for the corresponding fundamental quadratic form, that is,
$$Q_d(m,n):=\CN(m+n\tau_d)=|m+n\tau_d|^2.$$
 Explicitly, we have
$$
Q_d(m,n)= \begin{cases}
\displaystyle
m^2+mn+\frac{d+1}4 n^2 & \text{ if }d=3 \bmod 4;\\
m^2+dn^2&\text{ otherwise.}
\end{cases}
$$
We say that a quadratic form $Q(m,n)$ with integer coefficients is \emph{equivalent} to the form $Q_d$ if it is obtained from $Q_d$ by a change of variables given by a $2\times 2$ matrix with integer entries and determinant equal to $\pm 1$.
\begin{convention}
Every multiplicative function $f\in\CM$ is extended to an even multiplicative function on $\Z$, by putting $f(0)=0$ and $f(-n)=f(n)$ for every $n\in\N$. We denote this extension by $f$ as well.
\end{convention}
We prove  the following:
\begin{theorem}
[Some cases of Chowla's conjecture for aperiodic  functions]
\label{th:chowla}
 Let $f\in \CM$ be an aperiodic  multiplicative function, $d\in\N$, $Q$ be a quadratic form, equivalent to a quadratic form $Q_d$ defined above, and let $r\geq 0$ and $s\geq 1$ be integers.
Let $L_j(m,n)$, $j=1,\ldots, s$,  be linear forms with integer coefficients
 and suppose  that either $s=1$ or $s>1$ and the linear forms $L_1,L_j$ are linearly independent for $j=2,\ldots,  s$.
 Then
\begin{equation}
\label{E:PQr}
\lim_{N\to+\infty}\E_{m,n\in[N]}
f\bigl(Q(m,n)^r\bigr)\,\prod_{j=1}^s f\bigl(L_j(m,n)\bigr)=0.
\end{equation}
\end{theorem}
\begin{remarks}
(1) The same statement holds with
$f\bigl(Q(m,n)\bigr)^r$ in place of $f\bigl(Q(m,n)^r\bigr)$.

(2)  A more general result  is given in  Theorem~\ref{th:chowla2}.

(3)
The result fails  when there are no linear factors, for instance, it fails  for averages of  the form $\E_{m,n\in[N]}f(m^2+n^2)$, as it is easy to construct aperiodic multiplicative functions such that $f(m^2+n^2)=1$ for all $m,n\in \N$;
let $f$ be $1$ for integers that are  sums of two squares and $0$ if they are not. One can also  construct
 examples of aperiodic completely multiplicative functions $f$ with values $\pm 1$ such that
 $f(m^2+n^2)=1$ for all $m,n\in \N$.

 (4)  We restrict to positive definite quadratic forms because we can handle them using results for  imaginary quadratic fields;  real quadratic fields have infinitely many units and this causes problems in our  proof.
\end{remarks}
 When $r=0$,  Theorem~\ref{th:chowla}  follows by  combining
 Theorem~\ref{th:aperiod_uniform} with the estimates of Lemma~\ref{P:semest}. For  $r\geq 1$ the main observation is
 that since $Q_d(m,n)=\CN(m+n\tau_d)$  is completely multiplicative, the map $m+n\tau_d\mapsto f\big(Q_d(m,n)^r\big)$ is multiplicative in $\Z[\tau_d]$, in the sense defined in Section~\ref{subsec:Katai-Ztau}. Then using
 a variant of the  orthogonality criterion of Kat\'ai
 for the ring $\Z[\tau_d]$ (see Proposition~\ref{prop:KataiZd}) we can show that the
average on the left hand side of~\eqref{E:PQr}  converges to $0$
if some other average that involves  products of $2s$ linear forms converges to $0$. With a bit of effort we  show that the linear independence assumption of the linear forms is preserved, thus reducing the problem to the case  $r=0$ that we already know how to deal with using Theorem~\ref{T:DecompositionSimple}.

Perhaps the previous argument can be adjusted to treat the case of any irreducible quadratic polynomial $Q$.
  On the other hand, when $Q$  has  two or more irreducible quadratic factors or has irreducible factors of degree greater than two, we loose the basic multiplicativity property mentioned before and we do not see how to proceed.  It could be the case though that  Theorem~\ref{th:chowla} continues to hold, at least in the case that the function $f$ is completely multiplicative, for averages of the form
 $$
\E_{m,n\in[N]}
f\bigl(P(m,n)\bigr)
$$
 under the much weaker assumption that $P$ is
 any  homogeneous polynomial  such that  some linear form appears in the factorization of $P$ with degree exactly one.
\begin{problem}
Can Theorem~\ref{th:chowla} be extended to the case where $Q$ is an arbitrary irreducible quadratic  or a product of such? What about the case where  $Q$ is an arbitrary homogeneous polynomial without linear factors?
\end{problem}

\subsection{Partition regularity results}\label{SS:preg}
An important question in Ramsey theory is to determine which
algebraic equations, or systems of equations, are partition regular
over the natural numbers. Here, we restrict our attention to
polynomials in three variables, in which case  partition regularity of
 the equation  $p(x,y,z)=0$  amounts to  saying that,
 for any partition of  $\N$ into finitely many
cells, some cell contains \emph{distinct} $x,y,z$ that satisfy the
equation.

 The case where the polynomial $p$ is linear
was completely solved by  Rado~\cite{R33}; for $a,b,c\in
\N$  the equation $ax+by=cz$ is partition regular if and only if
either $a$, $b$, or $a+b$ is equal to $c$.
The situation is much less clear   for second or higher degree
equations and only scattered results are known.
 Partition regularity is known when the equation satisfies
a shift invariance property, as is the case for  the equations
$z-x=(y-x)^2$ (see \cite{BL96} or \cite{W00}) and  $x-y=z^2$
\cite{Be96}, and  in other instances it can be deduced from  related
linear statements  as is the case for  the equation $xy=z^2$
(consider the induced partition for the powers of $2$). But
such fortunate occurrences are rather rare.

\subsubsection{Quadratic equations} \label{SS:quadratic} A notorious old question of Erd\"{o}s and Graham~\cite{EG80} is
whether the equation $x^2+y^2=z^2$ is partition regular. As Graham
remarks in \cite{G08} ``There is actually very little data (in
either direction) to know which way to guess''. More generally, one
may ask for which $a,b,c\in \N$ is the equation
\begin{equation}\label{E:StrongPartitionRegular}
ax^2+by^2=cz^2
\end{equation}
 partition regular. A necessary condition is that at least  one of
$a$, $b$,  $a+b$ is equal to $c$, but currently there are no $a,b,c\in
\N$ for which partition regularity of 
\eqref{E:StrongPartitionRegular} is known.

We  study here the  partition regularity of equation
\eqref{E:StrongPartitionRegular}, and other quadratic equations,
under the relaxed condition that the  variable $z$ is allowed to
vary freely in $\N$; henceforth we use the letter $\lambda$ to designate the special role
of this variable.
\begin{definition}
The equation $p(x,y,\lambda)=0$ is \emph{partition
regular }  if for every partition of $\N$
into finitely many cells,
one of the cells  contains \emph{distinct} $x,y$ that satisfy the equation
for some $\lambda\in \Z$.
\end{definition}
A classical result of Furstenberg-S\'ark\"ozy  \cite{Fu77, Sa78}  is
that the equation $x-y=\lambda^2$ is partition regular.
 Other examples of translation invariant equations can be given using the polynomial van der Waerden theorem
 of Bergelson and Leibman~\cite{BL96},
 but not much is known in the non-translation invariant case.
A result of Khalfalah and Szemer\'edi \cite{KS06}  is that the
equation $x+y=\lambda^2$ is partition regular.\footnote{V.~Bergelson and J.~Moreira~\cite{BM15a,BM15b} recently proved
partition regularity in $\mathbb{Q}$ for patterns of the form $\{x+y,xy\}$, or equivalently, for the  equation $\lambda x -y=\lambda^2$. Partition regularity in  $\Z$ remains open.}
Again, the
situation is much less clear when one considers non-linear
polynomials in $x$ and $y$, as is the case for the equation
 $a x^2+by^2=\lambda^2$ where $a,b\in \N$.
We give the  first
positive results in this direction. For example, we   show that the
equations
$$
16x^2+9y^2=\lambda^2    \quad \text{ and } \quad x^2+y^2- xy=\lambda^2
$$
are partition regular (note that $16x^2+9y^2=z^2$ is not partition
regular).
In fact we prove a  more general result for  homogeneous quadratic
forms in three variables.
\begin{theorem}[Partition regularity of quadratic equations]\footnote{W.~Sun \cite{S15} recently proved a similar partition regularity result  on the Gaussian integers which covers the equation $x^2-y^2=\lambda^2$, where $x,y,\lambda\in \Z[i]$.}\label{th:partition-regular1}
Let $p$ be the quadratic form
\begin{equation}
\label{E:Q2} p(x,y,z)=ax^2+by^2+cz^2+dxy+exz+fyz,
\end{equation}
where  $a, b, c$ are non-zero  and $d,e,f$ are arbitrary integers.
Suppose that all three forms $p(x,0,z)$, $p(0,y,z)$, $p(x,x,z)$ have
non-zero square discriminants. Then the equation $ p(x,y,\lambda)=0 $ is
partition regular.
\end{theorem}
The last hypothesis means that
the three integers
$$
\Delta_1:=e^2-4ac,  \quad \Delta_2:=f^2-4bc,  \quad \Delta_3:=(e+f)^2-4c(a+b+d)
$$
are non-zero squares.
As a special case, 
 we get the following result:
\begin{corollary}\label{Corol1}
Let  $a,b,c,$ and $a+b$ be non-zero squares. Then the equation $
ax^2+by^2=c\lambda^2 $ is partition regular. Moreover, if   $a,b,$
and $a+b+c$ are  non-zero squares, then the equation $
ax^2+by^2+cxy=\lambda^2 $ is partition regular.
\end{corollary}

A partition $\mathscr P_1,\ldots, \mathscr P_r$ of the squares induces
a partition $\wt{\mathscr P}_1,\ldots, \wt{\mathscr P}_r$ of $\N$ by
the following rule: $x\in \wt{\mathscr P}_i$ if and only if $x^2\in
\mathscr P_i$.  Applying the first part of  Corollary~\ref{Corol1} for the  induced
partition we deduce partition regularity  results for   the set of  squares:
\begin{corollary}
 Let  $a, b,$ and $a+b$ be non-zero squares.
Then for every  partition of the squares
 into finitely many cells there exist distinct $x$ and $y$ belonging to the same cell
 such that $ax+by$ is a square.
\end{corollary}
Furthermore, we get the following result:
\begin{corollary}\label{C:parregsquares}
 Let  $a, b,$ and $a+b$ be non-zero squares, with $a,b$ coprime.
Then for every  partition of the squares
 into finitely many cells there exist $m,n \in \N$ such that the integers
 $bm^2, n^2-am^2$ belong to the same cell.
\end{corollary}
Indeed, as in the previous corollary, we deduce from Theorem~\ref{th:partition-regular1} that
 there exist squares $x,y$ in the same cell and $\lambda\in \N$ such that
$ax+by=b\lambda^2$. Since $(a,b)=1$  we have
 $x=bm^2$ for some $m\in \N$ and $y=\lambda^2-am^2$. The asserted conclusion  holds for  $n:=\lambda$.

\smallskip

 Although combinatorial  tools,  Fourier analysis tools, and the circle method have been used
successfully to prove partition regularity of equations that enjoy
some linearity features (also for non-linear equations with at least
  five variables \cite{BP15, Hen15a, Hen15b, Ke14, Ke15, Sm09}), we have not found such tools adequate for the fully
non-linear setup we are interested in.
    Instead, our main tool is going to be the structural result of Theorem~\ref{T:DecompositionII}.
    We give a summary of our  proof strategy   in Sections~\ref{SS:link} and \ref{SS:assuming}.

\subsubsection{Parametric reformulation}\label{SS:parametric}
In order to prove Theorem~\ref{th:partition-regular1} we exploit some
special features of  the solution sets of the equations involved given in parametric form.

\begin{definition}
We say that the integers $\ell_0,\ldots,\ell_4$
are \emph{admissible} if
$\ell_0$ is positive,
 $\ell_1\neq \ell_2$, $\ell_3\neq \ell_4$, and $\{\ell_1,\ell_2\}\neq \{\ell_3,\ell_4\}$.
\end{definition}
The following result  is proved in
Appendix~\ref{SS:AppNumberTheory}:
\begin{proposition}[Parametric form of solutions]
\label{prop:linearfactors}
Let the quadratic form $p$ satisfy the hypothesis of
Theorem~\ref{th:partition-regular1}.
Then there exist
admissible integers $\ell_0,\ldots,\ell_4$,
 such that for every $k,m,n\in \Z$,  the
integers $x:=k\ell_0 (m+\ell_1n)(m+\ell_2n)$ and $y:=
k\ell_0(m+\ell_3n)(m+\ell_4n)$ satisfy the equation $p(x,y,\lambda)=0$ for
some $\lambda\in \Z$.
\end{proposition}
For example,  the equation  $16x^2+9y^2=\lambda^2$ is satisfied by the
integers $x:=km(m+3n)$, $y:=k(m+n)(m-3n)$, $\lambda:=k(5m^2+9n^2+6mn)$,
 and  the equation  $x^2+y^2-
xy=\lambda^2$ is satisfied by the integers $x:=km(m+ 2n)$, $y:=
k(m-n)(m+n)$, $\lambda:=k( m^2+n^2+mn)$.

The  key properties of the patterns involved in Proposition~\ref{prop:linearfactors} are: $(a)$
they are  dilation invariant, which follows from homogeneity, $(b)$
they ``factor linearly''
which follows from our assumption that the discriminants $\Delta_1,\Delta_2$ are
squares,
 and $(c)$ the coefficient of $m$ in all forms can be taken to be $1$ which follows from our assumption that the discriminant $\Delta_3$ is a square.

Using Proposition~\ref{prop:linearfactors}, we see that
Theorem~\ref{th:partition-regular1}  is a consequence of the
following result:
\begin{theorem}[Parametric reformulation of Theorem~\ref{th:partition-regular1}]
\label{th:partition-regular2}
Let $\ell_0,\ldots,\ell_4\in \Z$ be  admissible.
Then for every partition of $\N$ into
finitely many cells, there exist $k,m,n\in\Z$ such that the integers
$k\ell_0(m+\ell_1n)(m+\ell_2n)$ and $k\ell_0(m+\ell_3n)(m+\ell_4n)$ are positive,
distinct, and belong to  the same cell.
\end{theorem}
In fact, in Theorem~\ref{th:density-regular} we prove something stronger, that  any set of integers with positive multiplicative density (a notion defined in  Section~\ref{SS:dildens}) contains the
aforementioned configurations.




\subsubsection{More general patterns and higher degree equations}\label{SS:Higher}
Theorem~\ref{th:partition-regular2} is proved using the structural result of  Theorem~\ref{T:DecompositionII} for $s=3$;  using this structural result for general $s\in \N$
 we can prove, without essential changes in our argument, the following strengthening:

\begin{theorem}
\label{th:partition-regular3}
Let $s\geq 2$. For $i=1,2$ and $j=1,\ldots, s$, let $L_{i,j}(m,n)$ be  linear forms with integer coefficients. Suppose that for $i=1,2$, the linear forms $L_{i,j}$, $j=1,\ldots,s$, are  pairwise independent and that
 the product of the coefficients of $m$ in the forms $L_{1,j}$ and
in the forms $L_{2,j}$ are equal and non-zero.
Then for every partition of $\N$ into
finitely many cells, there exist $k, m,n\in\Z$ such that the integers
$
k\prod_{j=1}^s L_{1,j}(m,n)$  and  $k\prod_{j=1}^sL_{2,j}(m,n)$  are positive, distinct,  and belong to  the same cell.
\end{theorem}

%
Theorem~\ref{th:partition-regular3} can be used  to show that several  homogeneous equations in three variables of degree  greater than two are partition regular.
Unfortunately, we have no general criterion like  the one given in Theorem~\ref{th:partition-regular1} and Corollary~\ref{Corol1}.\footnote{Note that a celebrated  result of Faltings~\cite{Fa83} implies that for $d\geq 4$ the equation $ax^d+by^d=cz^d$,  has
finitely many coprime solutions. This implies that such equations cannot be partition regular.} We record here one  example of degree three that we found with the help of N.~Tzanakis and some computer software\footnote{\texttt{http://www.wolframalpha.com}.} (examples of higher degree equations in three variables can also be found). Let
 $$
p(x,y,z):=2x^3-2x^2y+17x^2z-4xyz+44xz^2-y^2z+36z^3.
$$
The equation $p(x,y,\lambda)=0$ is satisfied for
$$
x:=km(2m+n)(m-n),\quad  y:=k(m+n) (2m-n)(m+2n), \quad  \lambda:=km^2n
$$
for every $k,m,n\in \Z$. It follows from
Theorem~\ref{th:partition-regular3} that the equation $p(x,y,\lambda)=0$ is partition regular.

When one considers four  or more variables there is even more flexibility. For example   N.~Tzanakis brought to our attention the following
identity of G\'erardin:
$$
(m^2-n^2)^4 + (2mn+m^2)^4 + (2mn+n^2)^4 = 2(m^2+mn+n^2)^4.
$$
Using Theorem~\ref{th:partition-regular2}, we deduce that for every partition of $\N$ into
finitely many cells, there exist distinct $x,y$ belonging to the same cell and $\lambda,\mu\in \N$ such  that $x^4+y^4=2\lambda ^4-\mu^4$.
As in Corollary~\ref{C:parregsquares}, we deduce that
 for every  partition of the fourth powers
 into finitely many cells, there exist $m,n,r \in \N$ such that the integers
 $m^4$ and  $2n^4-m^4-r^4$ belong to the same cell.

\subsubsection{From partition regularity to multiplicative functions}\label{SS:link}
Much like the translation invariant case, where partition regularity
results can be deduced from corresponding density statements with
respect to a translation invariant density, we deduce
Theorem~\ref{th:partition-regular2}
  from the density regularity result of Theorem~\ref{th:density-regular} that
involves a dilation invariant density (a notion defined in
Section~\ref{SS:dildens}).


In Section~\ref{subsec:Bochner} we use a well known   integral representation  result of Bochner
 that characterizes positive definite sequences on the group $\Q^+$ in order to recast
the density regularity statement as     a
 positivity property  for an integral of averages of products of multiplicative functions
 (see Theorem~\ref{th:ergo2b}).
 It is  this positivity property that we seek to prove,
and the heavy-lifting is done by the structural result of  Theorem~\ref{T:DecompositionII}
for $s=3$. The proof of the analytic statement of
Theorem~\ref{th:ergo2b} is completed in Section~\ref{SS:assuming} and
the reader will  find there
 a detailed sketch
of the proof strategy for this crucial step. We remark that although we do not make explicit use of ergodic theory
anywhere in this argument, ideas from the ergodic theoretic proof, given by Furstenberg~\cite{Fu77},  of  S\'ark\"ozy's theorem~\cite{Sa78} have guided the last part of our argument.

\subsubsection{Further directions}
Theorem~\ref{th:partition-regular1} implies that the equation
\begin{equation}\label{eq:basic1}
ax^2+by^2=c\lambda^2
\end{equation}
 is partition regular
 provided that all three integers
  $ac, bc$, $(a+b)c$,  are non-zero squares.
Two interesting
cases, not covered by the previous result,  are the following:
\begin{problem}\label{prob1}
Are the equations $x^2+y^2=\lambda^2$ and $x^2+y^2=2\lambda^2$ partition
regular?\footnote{Note that the equation $x^2+y^2=3\lambda^2$ does not have solutions in $\N$.
Furthermore, the equation $x^2+y^2=5\lambda^2$  has solutions in $\N$ but it is not
partition regular. Indeed, if we partition the integers in $6$ cells
according to whether their first non-zero digit in the $7$-adic
expansion is $1,2,\ldots, 6$, it turns out  that for every
$\lambda\in \N$ the equation has no solution on any single partition
cell.}
\end{problem}
 Let us explain why we cannot yet handle these equations using the methods
 of this article. The equation
$x^2+y^2=2z^2$ has the following solutions: $x:=k(m^2-n^2+2mn)$,
$y:=k(m^2-n^2-2mn)$, $z:=k(m^2+n^2)$, where $k,m,n\in \Z$. The values
of $x$ and $y$ do not factor in linear terms and uniformity estimates
analogous to the ones stated in Lemma~\ref{L:UnifromityEstimates2} fail.
The equation $x^2+y^2=z^2$ has the following solutions:
$x:=k(m^2-n^2)$, $y:=2kmn$, $z:=k(m^2+n^2)$. In this case, it is
possible to  establish the needed uniformity estimates but we are
not able to carry out the argument of Section~\ref{SS:assuming}
 in order to prove the relevant positivity property (see footnote~\ref{foot3} in Section~\ref{SS:assuming}
 below for more details).

 A set $E\subset\N$ has \emph{positive
\emph{(additive)} upper density} if $\limsup_{N\to+\infty}|E\cap [N]|/N>0$.
It turns out that the equations of Corollary~\ref{Corol1}
 have non-trivial solutions on
every
 infinite arithmetic progression, making the following statement  plausible:
\begin{problem}
Does every set $E\subset \N$ with positive upper density contain distinct
$x,y\in \N$ that satisfy  the equation $16x^2+9y^2=\lambda^2$ for some $\lambda\in\N$?
\end{problem}
We say that the equation $p(x,y,z)=0$, $p\in \Z[x,y,z]$,  has
\emph{no local obstructions} if for every infinite arithmetic progression $P$, there exist distinct $x,y,z\in P$ that satisfy the equation. For example, the equations $x^2+y^2=2z^2$ and $16x^2+9y^2=25z^2$
 have no local obstructions.
\begin{problem}
Let $p\in \Z[x,y,z]$ be
 a homogeneous quadratic form and suppose that the equation $p(x,y,z)=0$ has no local obstructions. Is it true that
every subset  of $\N$ of positive upper density contains distinct $x,y,z$ that satisfy the equation?
\end{problem}
See \cite{GR12} for information regarding the  density regularity of the equation $x^2+y^2=2z^2$.

The theorem of S\'arkozy~\cite{Sa78} implies that for every finite partition
of the integers some cell contains integers of the form $m, m+n^2$. What can we say about $m$? Can it be a square?
\begin{problem}
Is it true that for  every  partition of the integers  into finitely many cells one cell contains
integers of the form $m^2$ and $m^2+n^2$?
\end{problem}
We remark that the answer will be positive  if one shows  that the equation $x^2-y^2=\lambda^2$ is partition regular.

\subsection{Structure of the article} In Section~\ref{S:U^2} we study the Fourier coefficients of multiplicative functions. Our basic tool
is the  orthogonality criterion of K\'atai and
we establish the structural result of Theorem~\ref{th:Decomposition-U2}
which is a more informative version of Theorem~\ref{T:DecompositionSimple} for $s=2$.

In Sections~\ref{sec:Nil} and \ref{sec:equidistribution}
 we review some facts about nilmanifolds and  state some  results
 that are instrumental for our subsequent work; the inverse theorem for the $U^s$-norms (Theorem~\ref{th:inverse}),  the quantitative Leibman theorem (Theorem~\ref{th:Leibman}), and
  the factorization theorem for polynomial sequences on nilmanifolds (Theorem~\ref{th:FactoGT}).
 We also derive some consequences that will be used later on.

  Sections~\ref{S:minorarcs1} and \ref{S:minorarcs2} are  in some sense the heart of the proof of our  structural results. Our main
 result is Theorem~\ref{th:discorrelation} where
 we prove that an arbitrary multiplicative function has small correlation  with all minor arc nilsequences.
 After proving some preparatory results in Section~\ref{S:minorarcs1} we complete  the proof of Theorem~\ref{th:discorrelation}  in Section~\ref{S:minorarcs2}.

 In Section~\ref{sec:decompUs} we prove our main structural results.  Theorem~\ref{T:DecompositionI}  is a more informative version of Theorem~\ref{T:DecompositionSimple} and
  we deduce
 Theorem~\ref{T:DecompositionII} using an iterative argument of energy increment.

 In Section~\ref{sec:aperiodic} we deal with two applications of the structural result
of  Theorem~\ref{T:DecompositionI}
 to aperiodic multiplicative functions.
 The first is  Theorem~\ref{th:aperiod_uniform} which states that
 a multiplicative function is aperiodic if and only if it is $U^s$-uniform for every $s\geq 2$.
 The second is Theorem~\ref{th:chowla} which  provides a class of homogeneous polynomials in two variables
 for which  Chowla's zero mean conjecture holds for every aperiodic completely  multiplicative function.

 Finally, in Section~\ref{S:Recurrence} we use the structural result of Theorem~\ref{T:DecompositionII}
  to prove our main partition regularity result for homogeneous quadratic equations stated in Theorem~\ref{th:partition-regular1}. 

\subsection{Notation and conventions}
 For reader's convenience, we gather here some notation that we use throughout the article.

\smallskip
 \noindent  We denote by $\N$ the set of positive integers.

 \smallskip

\noindent For $N\in \N$ we denote by $[N]$ the set $\{1,\ldots,
N\}$.

\smallskip

\noindent For a function $a$  defined on a finite set $A$ we write $
\E_{x\in A}a(x)=\frac 1{|A|}\sum_{x\in A}a(x). $

\smallskip

\noindent With $\CM$ we denote the set of multiplicative functions
$ f \colon \N\to \C$ with modulus at most $1$, and with  $\CM_1^c$ the set of completely multiplicative functions
$ f \colon \N\to \C$ with modulus exactly $1$.

\smallskip

\noindent A  kernel on $\Z_N$  is a non-negative function on $\Z_N$
with average $1$.

\smallskip

\noindent  Throughout, we assume that we are given an
integer $\ell\geq 2$, its value depends on the problem at hand,  and we leave the dependence on $\ell$
of the various parameters implicit.

\smallskip

\noindent For  $N\in \N$ we let $\tN$ be any prime with $N\leq \tN\leq\ell N$. In some cases we  specify the  value of $\tN$ and its precise dependence on $N$ depends on the application we have in mind.

\smallskip

\noindent Given $ f \in \CM$ and $N\in \N$ we let $ f_ N\colon
[\tN]\to \C$ be defined by $ f_ N= f  \cdot \one_{[N]} $. The domain
of $ f_ N$ is sometimes  thought to be $\Z_\tN$.

\smallskip  \noindent For technical reasons, throughout the
article all Fourier analysis happens on $\Z_\tN$   and all
uniformity norms are defined on $\Z_\tN$.

\smallskip
\noindent If $x$ is a real, $\e(x)$ denotes the number $e^{2\pi i
x}$, $\norm x$ denotes the distance between $x$ and the nearest
integer, $\lfloor x\rfloor$ the largest integer smaller or equal than $x$,
and $\lceil x\rceil$ the smallest integer greater or equal than $x$.

\smallskip

\noindent Given $k\in \N$ we write $\bh=(h_1,\dots,h_k)$ for a point
of $\Z^k$ and $\norm\bh=|h_1|+\dots+|h_k|$.

 \smallskip

\noindent  For
$\bu=(u_1,\dots,u_k)\in\T^k$, we write $\norm \bu=\norm{u_1}+\dots+\norm{u_k}$ and for $\bh\in\Z^k$ we write
$\bh\cdot\bu=h_1u_1+\dots+h_ku_k$.

\smallskip

\noindent If $\Phi$ is a  function on a metric space $X$ with distance
$d$, we let
$$
\norm \Phi_{\lip(X)}=\sup_{x\in X}|\Phi(x)|+\sup_{\substack{x,y\in X\\
x\neq y}}\frac{|\Phi(x)-\Phi(y)|}{d(x,y)}.
$$

\smallskip

\noindent There is a proliferation of constants in this article and
our general principles are as follows: The constants
$\ell, \ell_1,\ell_2,\ldots,$ are considered as fixed throughout the
article, and quantities depending only on these numbers are
considered as universal constants.
Quantities that depend on one or more variables are denoted by Roman
capital letters $C, D, K, L,\ldots $ if they represent large quantities,
and by low case Greek letters $\delta, \ve, \eta, \tau, \ldots$ and  $c$ if they
represent small quantities. It will be very clear from the context
when we deviate from these rules.

\section{Fourier analysis of multiplicative functions}
\label{S:U^2}
In this section  we study the Fourier coefficients of multiplicative
functions. Our   goal is to establish Theorem~\ref{th:Decomposition-U2}, that
proves the case $s=2$ of the  decomposition result
of Theorem~\ref{T:DecompositionSimple} and gives more precise information on  the
structured and uniform components.
We will  use this result in Section~\ref{sec:decompUs} as our starting point in the
 proof of the more general structure theorem  for the $U^s$-norm for $s\geq 2$.

 We recall some notation and conventions. The  integer $\ell\geq 2$ is considered as fixed throughout and we never indicate the dependence on this number.
For every $N\in \N$,    we denote by $\wt N$  a  prime with $N\leq \tN\leq\ell N$. For every $ f \in\CM$, we write $ f_ N= f \cdot  \one_{[N]}$, and we consider this as a function  defined on $\tZN$.
 Henceforth, all  convolution products are defined   on $\tZN$ and
 the Fourier coefficients of $ f_ N$ are given by
$$
\widehat{ f_ N}(\xi):=\E_{n\in\tZN} f_ N(n)\, \e(-n \xi/\wt N)
\quad \text{ for } \ \xi\in\tZN.
$$

\subsection{The K\'atai orthogonality criterion}
\label{subsec:katai} We start with a key number theoretic input
that we need in this section and which will also be used later in Sections~\ref{S:minorarcs2} and \ref{sec:aperiodic}.

\begin{lemma}[K\'atai  orthogonality criterion~\cite{K86}, see also  \cite{D74, MV77}]
\label{lem:katai} For every $\ve>0$ 
there exists
$\delta:=\delta(\ve)>0$ and $K:=K(\ve)$ such that the
following holds: If $N\geq K$  and $a\colon [N]\to \C$ is a function with
$|a|\leq 1$ and
$$
\max_{\substack{p,q\text{ \rm primes}\\1<p<q<K}}\bigl|\E_{n\in[\lfloor
N/q\rfloor]} a(pn)\, \overline{a}(qn)\bigr|< \delta,
$$
 then
$$
\sup_{ f \in\CM}\bigl|\E_{n\in[N]} f (n)\, a(n)\bigr|<\ve.
$$
\end{lemma}
\begin{remark}
The result is stated and proved in \cite{K86} for functions $a\colon [N]\to \C$  of modulus $1$, but the same
argument works for sequences with modulus at most $1$.
\end{remark}

The dependence of $\delta$ and $K$ on $\ve$ can be made
explicit (for good bounds see \cite{BSZ12}) but we do not need such extra information
here.
We give a complete proof of  Lemma~\ref{lem:katai} in a more general context in Section~\ref{sec:aperiodic}
(see Lemma~\ref{lem:KataiZd}).

%

\subsection{Fourier coefficients of multiplicative functions}
Next, we use the orthogonality criterion of K\'atai in order to prove that
the Fourier coefficients of the restriction of a multiplicative
function on an interval $[N]$ are small unless the frequency is
close to a rational with small denominator. Furthermore, the
implicit constants do not depend on the multiplicative function or
the integer $N$.

\begin{corollary}[$U^2$ non-uniformity]
\label{cor:katai}
For every $\theta>0$ there exist positive integers $N_0:=N_0(\theta)$,
$Q:=Q(\theta)$,  and $V:=V(\theta)$, such that for every $N\geq N_0$, for every $ f \in\CM$, and
every $\xi\in\tZN$, we have the following implication
\begin{equation}
\label{eq:Fourier_chi} \text{if }\
|\widehat{ f_ N}(\xi)|\geq\theta,\quad  \text{ then }\quad
\Bigl\Vert\frac{Q\xi}{\wt N}\Bigr\Vert\leq \frac{QV}{\wt N}.
\end{equation}
\end{corollary}
\begin{proof}
Let  $\delta:=\delta(\theta)$ and  $K:=K(\theta)$ be defined  by
Lemma~\ref{lem:katai} and let $Q:=K!$. Suppose that $N>K$. Let  $p$
and $p'$ be primes with $p<p'\leq K$ and  let $\xi\in\tZN$.
  If
$\xi=0$ the conclusion is obvious, otherwise, since $\tN$ is a prime greater than $K$ we have
$\norm{Q \xi/\wt N}\neq 0$.
Since $Q$
is a multiple of $p'-p$ we have
$$
0<\norm{Q\xi/\wt N} \leq \frac{Q}{p'-p} \norm{(p'-p)\xi/ \tN} \leq
Q\,\norm{(p'-p)\xi/\wt N}.
$$
Since $\wt{N}\leq \ell N$, we deduce that
 $$
|\E_{n\in[\lfloor N/p'\rfloor]}\,\e(p'n\xi/ \wt{N}) \, \e(-pn\xi/
\wt{N})| \leq \frac {2p'} {N\norm {(p'-p)\xi/\wt{N}}} \leq\frac
{2KQ}{N\norm{Q\xi/\wt N}} \leq\frac{2\, \ell KQ}{\wt N\norm{Q\xi/\wt
N}}.
$$
Let $V:=2\, \ell K/\delta$. If  $\norm{Q\xi/\wt{N}} > QV/ \wt{N}$,
then  the rightmost  term of the last inequality   is smaller than
$\delta$, and thus,
 by Lemma~\ref{lem:katai}  we have
$$
|\widehat{ f_ N}(\xi)|=
\bigl|\E_{n\in[\wt{N}]}\, f_ N(n)\e(-n\xi/\wt{N})\bigr| =
\frac{N}{\wt{N}}\bigl|\E_{n\in
[N]} f (n)\e(-n\xi/\wt{N})\bigr|<\theta
$$
contradicting \eqref{eq:Fourier_chi}. Hence, $\norm{ Q\xi /\wt
N}\leq QV/\wt N$, completing the proof.
 \end{proof}

\subsection{Definition of  kernels}
\label{subsec:kernels}
We recall that a \emph{kernel}
on  $\Z_{\wt{N}}$ is a non-negative function $\phi$ on $\Z_{\wt{N}}$
with $\E_{n\in\Z_{\wt{N}}}\phi(n)=1$. The \emph{spectrum} of a function $\phi\colon \Z_{\wt N}\to \C$
is the set
$$
\spec(\phi):=\bigl\{\xi\in\tZN\colon \widehat\phi(\xi)\neq 0\bigr\}.
$$

Next, we make some explicit choices for the constants $Q$ and $V$
   of Corollary~\ref{cor:katai}. This will enable us to compare the Fourier transforms of the
   kernels $\phi_{N,\theta}$ defined below for different values of $\theta$ and to establish the monotonicity Property~\eqref{eq:phi-increases} in Theorem~\ref{th:Decomposition-U2} below.
For every $\theta>0$ let $N_0(\theta)$ be as in Corollary~\ref{cor:katai}. For $N\geq N_0(\theta)$ we define
\begin{gather}
\notag \CA(N,\theta):=\Bigl\{\xi\in\tZN\colon \exists f \in\CM
\ \text{ such that } \ |\widehat{ f_ N}(\xi)|\geq\theta^2\Bigr\};\\
\notag W(N,q,\theta):= \max_{\xi\in \CA(N,\theta)}\Big\{ \wt{N}\Bigl\Vert q\,\frac{\xi}{\wt{N}}\Bigr\Vert\Big\};\\
\label{def:Qtheta} Q(\theta):=\min_{k\in\N}\Bigl\{k!\colon\ \limsup_{N\to+\infty} W(N,k!,\theta)<+\infty\Bigr\} ;\\
\label{def:D(theta)}V(\theta):=1+\Bigl\lceil\frac
1{Q(\theta)}\,\limsup_{N\to+\infty}  W(N,Q(\theta),\theta)
\Bigr\rceil.
\end{gather}
It follows from  Corollary~\ref{cor:katai} that the set of integers
used in the definition of $Q(\theta)$ is non-empty, hence
$Q(\theta)$ is well defined.
It follows from  the preceding
definitions that  there exists $N_1:=N_1(\theta)$ such that
\begin{align}\label{E:star}
& \text{Implication }\eqref{eq:Fourier_chi} \text{ holds for } N\geq N_1 \text{ with }\theta^2
\text{ substituted for } \theta, \  V(\theta) \text{ for } V,\\ & \notag \text{and } Q(\theta) \text{ for }Q.
\end{align}
Furthermore, for $0<\theta'\leq\theta$, we
have $Q(\theta')\geq Q(\theta)$,  and thus
\begin{equation}
\label{eq:Q-multliple} \text{for }\ 0<\theta'\leq\theta,\ \text{ the
integer }\  Q(\theta')\ \text{ is a multiple of }\ Q(\theta).
\end{equation}
Moreover, it can be checked that
\begin{equation}
\label{eq:c-devreases} V(\theta)\   \text{ increases as } \ \theta \ \text{ decreases}.
\end{equation}

Next, we use the constants just defined  to build the
kernels $\phi_{N,\theta}$ of
Theorem~\ref{th:Decomposition-U2} below.

For every $m\in \N$ and $\tN>2m$ the ``Fejer kernel'' $\phi_{N,m}$ on
$\Z_\tN$ is defined by
$$
\phi_{N,m}(x):=\sum_{-m\leq \xi\leq
m}\bigl(1-\frac{|\xi|}m\bigr)\,\e\bigl(x\,\frac{\xi}{\tN}\bigr)
$$
where  the interval $\{-m,\ldots, m\}$ is imbedded in $\Z_{\tN}$ in the obvious way.
The spectrum  of $f_{N,m}$ is the subset   $\{-m+1,\ldots,m-1\}$ of $\tZN$. Let
$Q_N(\theta)^*$ be the inverse of $Q(\theta)$ in $\Z_\tN$, that is,
the unique integer in $\{1,\dots,\tN-1\}$ such that
$Q(\theta)Q_N(\theta)^*=1 \bmod \tN$. Let
\begin{equation}
\label{eq:defN0U2}
N_0:=N_0(\theta)=\max\{N_1,2Q(\theta)V(\theta)\lceil \theta^{-2}\rceil\}.
\end{equation}
 For $N\geq N_0$
 we define
\begin{equation}
\label{eq:def-phi} \phi_{N,\theta}(x):=f_{N,
Q(\theta)V(\theta)\lceil\theta^{-4}\rceil}(Q_N(\theta)^*x).
\end{equation}
An equivalent formulation is that $f_{N,
Q(\theta)V(\theta)\lceil\theta^{-4}\rceil}(x)=\phi_{N,\theta}(Q(\theta)x)$.
The
spectrum  of the kernel $\phi_{N,\theta}$ is the set
\begin{equation}
\label{eq:def-Xi}
\Xi_{N,\theta}:=\Big\{\xi\in \tZN\colon
\Bigl\Vert\frac{Q(\theta) \xi}\tN\Bigr\Vert<
\frac{Q(\theta)V(\theta)\lceil\theta^{-4}\rceil}\tN\Big\},
\end{equation}
and we have
\begin{equation}
\label{eq:fourier-phi}
 \widehat{\phi_{N,\theta}}(\xi) =\begin{cases}
\displaystyle 1-\Bigl\Vert \frac{Q(\theta)\xi}\tN\Bigr\Vert\, \frac
\tN{ Q(\theta)V(\theta)\lceil\theta^{-4}\rceil}
&\ \  \text{if }\  \xi\in\Xi_{N,\theta} \  ;\\
0 & \ \ \text{otherwise.}
\end{cases}
\end{equation}
Note that the cardinality of  $\Xi_{N,\theta}$    is bounded by a constant that depends only on $\theta$.

\subsection{$U^2$-structure theorem for multiplicative functions}
We can now state and prove the main result of this section.
\begin{theorem}[$U^2$-structure theorem for multiplicative functions]
\label{th:Decomposition-U2}
 Let
$\theta>0$. There exist positive  integers $N_0:=N_0(\theta)$,   $Q:=Q(\theta)$,
$R:=R(\theta)$, such that for $N\geq N_0$ the following holds: Let
   the kernel $\phi_{N,\theta}$ be
defined in Section~\ref{subsec:kernels}, and  for  every $ f \in\CM$ let
$$
 f_ {N,\st}:= f_ N*\phi_{N,\theta}\ \ \text{ and }\ \
 f_ {N,\un}:= f_ N- f_ {N,\st}.
$$
Then   we have
\begin{enumerate}
\item
\label{it:SpectrumU2}
 If $\xi \in \Z_\tN$ satisfies $\widehat{ f }_{N,\st}(\xi)\neq 0$, then $\displaystyle \big|\frac{\xi}{\tN}-\frac{p}{Q}\big|\leq \frac{R}{\tN}$ for some   $p\in \{0,\ldots Q-1\}$;
\item
\label{it:decomU22} $\displaystyle
| f_ {N,\st}(n+Q)- f_ {N,\st}(n)|\leq \frac{R}{\tN}$ for every $n\in\tZN$,
where  $n+Q$ is taken \!\!\!$\mod\tN$;
\item
\label{it:decomU23}
 $\displaystyle\norm{ f_ {N,\un}}_{U^2(\tZN)}\leq\theta$.
\end{enumerate}
Moreover, if $0<\theta'\leq\theta$ and $N\geq\max\{N_0(\theta),N_0(\theta')\}$, then
\begin{equation}
\label{eq:phi-increases}
\text{  for every }\xi\in\tZN, \ \widehat{\phi_{N,\theta'}}(\xi)\geq
\widehat{\phi_{N,\theta}}(\xi)\geq  0.
\end{equation}
\end{theorem}
\begin{remarks} (1) The monotonicity Property~\eqref{eq:phi-increases} plays a central
role in the derivation of Theorem~\ref{T:DecompositionII} from
Theorem~\ref{T:DecompositionI} in
Section~\ref{subsec:proof_strong}. This  is  one of the reasons why
we construct the kernels $\phi_{N,\theta}$  explicitly in
Section~\ref{subsec:kernels}.

(2) The
 values of  $Q$ and $R$ given by  Theorem~\ref{th:Decomposition-U2}
  will be  used later in Section~\ref{sec:decompUs}, and  they do not coincide
   with the values of $Q$ and $R$ in
   Theorems~\ref{T:DecompositionII} and \ref{T:DecompositionI}.
\end{remarks}
\begin{proof}
We first show that \eqref{eq:phi-increases}  holds. Indeed, suppose
that $\theta\geq\theta'>0$ and that  $N\geq\max\{N_0(\theta),N_0(\theta')\}$. We have
to show that $\widehat{\phi_{N,\theta'}}(\xi)\geq
\widehat{\phi_{N,\theta}}(\xi)$ for every $\xi\in \Z_{\wt N}$. Using
\eqref{eq:Q-multliple} and  \eqref{eq:c-devreases}, we get that
$\Xi_{N,\theta'}$ contains the set $\Xi_{N,\theta}$. Thus,  we can
assume that $\xi$ belongs to the latter set as the estimate is
obvious otherwise. In this case, the claim follows
from~\eqref{eq:Q-multliple}, \eqref{eq:c-devreases} and the
formula~\eqref{eq:fourier-phi} giving the Fourier coefficients of
$\phi_{N,\theta}$.

Next, we show the remaining  assertions~\eqref{it:SpectrumU2}, \eqref{it:decomU22},
\eqref{it:decomU23}  of  the statement.
Let  $\theta>0$.
Let
$Q:=Q(\theta)$, $V:=V(\theta)$, $N_0(\theta)$
  be defined by
\eqref{def:Qtheta},  \eqref{def:D(theta)}, \eqref{eq:defN0U2} respectively. Suppose that
$N\geq N_0(\theta)$ and let
$\phi_{N,\theta}$ and  $\Xi_{N,\theta}$ be defined by \eqref{eq:def-phi} and
\eqref{eq:def-Xi} respectively.

If for some $ f \in\CM$ and $\xi\in\Z_\tN$ we have $\widehat{ f_ {N,\st}}(\xi)\neq 0$, then
$\widehat{\phi_{N,\theta}}(\xi)\neq 0$ and $\xi$ belongs to the set $\Xi_{N,\theta}$ defined
by~\eqref{eq:def-Xi}. Hence, Property~\eqref{it:SpectrumU2} holds, for some constant $R$ depending only on $\theta$.

Moreover, for
$ f \in\CM$ and $n\in\Z_\tN$, using the Fourier inversion formula
and the estimate $|\e(x)-1|\leq 2\pi  \norm{x}$, we get
$$
|(\phi_{N,\theta}* f_ N)(n+Q)-(\phi_{N,\theta}* f_ N)(n)|\leq
2\pi\sum_{\xi\in\tZN}|\widehat{\phi_{N,\theta}}(\xi)|\cdot\Bigl\Vert
Q\frac\xi \tN\Bigr\Vert\leq 2\pi \frac{|\Xi_{N,\theta}|
QV\lceil\theta^{-4}\rceil}{\tN},
$$
where the last estimate follows from \eqref{eq:def-Xi}. The last
term in this inequality is bounded by $R/\wt N$ for some constant
$R$ that depends only on $\theta$. This establishes
Property~\eqref{it:decomU22}.

Lastly,   since $N\geq N_0(\theta)\geq N_1(\theta)$, by \eqref{E:star}  we have that
for every $ f \in\CM$,  if $|\widehat{ f_ N}(\xi)|\geq\theta^2$,
then  $\norm{Q\xi/\tN}\leq QV/\tN$
and thus
 $\widehat {\phi_{N,\theta}}(\xi)\geq 1-\theta^4$ by~\eqref{eq:fourier-phi}.
 It follows that
 $|\widehat{ f_ N}(\xi)-\widehat{\phi_{N,\theta}* f_ N}(\xi)|$ $\leq
 \theta^4\leq\theta^2$.
 This last bound is clearly also true when $|\widehat{ f_ N}(\xi)| <\theta^2$ and thus,
using identity \eqref{eq:U2Fourier},  we get
\begin{multline*}
\norm{ f_ N-\phi_{N,\theta}* f_ N}_{U^2(\Z_\tN)}^4
=\sum_{\xi\in\tZN}
  |\widehat{ f_ N}(\xi)-\widehat{\phi_{N,\theta}* f_ N}
    (\xi)|^4\leq \\
  \theta^4
\sum_{\xi\in\tZN}|\widehat{ f_ N}(\xi)-\widehat{\phi_{N,\theta}
   * f_ N}(\xi)|^2
\leq  \theta^4 \sum_{\xi\in\tZN}|\widehat{ f_ N}(\xi)|^2\leq
   \theta^4,
\end{multline*}
where the last estimate follows from Parseval's identity. Hence,
$\norm{ f_ N-\phi_{N,\theta}* f_ N}_{U^2(\Z_\tN)}\leq\theta$,
proving Property~\eqref{it:decomU23} and completing the proof of the theorem.
\end{proof}


\subsection{A model structure theorem}
Before we enter the proof of  the $U^s$-structure theorem for $s\geq 3$
we sketch the proof of   a toy model that can serve as a guide for the much  more complicated argument that comes later on.

\begin{proposition}[Model structure theorem for multiplicative functions]
\label{prop:baby}
Let $\ve>0$. There exists $\theta:=\theta(\ve)$ such that for every
 sufficiently large $N\in\N$, depending only on  $\ve$, the decomposition $ f_ N= f_ {N,\st}+ f_ {N,\un}$ associated to $\theta$ by Theorem~\ref{th:Decomposition-U2} satisfies Properties~\eqref{it:SpectrumU2}, \eqref{it:decomU22} of this theorem, and also
\begin{equation}
\label{eq:disc-square'}
\sup_{ f \in\CM,\ \alpha\in\R}\bigl|\E_{n\in[N]} f_ {N,\un}(n)\, \e(n^2\alpha)\bigr|\leq\ve.
 \end{equation}
\end{proposition}

\begin{proof}[Proof (Sketch)] Let $\ve>0$ and $N\in \N$ be sufficiently large.
Let $\alpha\in\R$ and $ f \in\CM$ and  suppose that
\begin{equation}
\label{eq:dis-square2}
\bigl|\E_{n\in[N]} f_ {N,\un}(n)\, \e(n^2\alpha)\bigr|\geq\ve.
\end{equation}
\subsubsection*{Minor arcs}
Recall that $ f_ {N,\st}= f_ N*\phi_{N,\theta}$ where $\phi_{N,\theta}$ is a kernel on $\Z_\tN$ and the convolution is taken on $\Z_\tN$. Therefore, we have $ f_ {N,\un}= f_ N*\psi_{N,\theta}$ where the function $\psi_{N,\theta}$ satisfies $\E_{n\in\Z_\tN}|\psi_{N,\theta}|\leq 2$. Taking in account  the roundabout effects,   and using that $\tN\leq \ell N$ we deduce that there exists $k\in \Z$ with
$$
\bigl|\E_{n\in[\tN]}\one_{[N]}(n+k)\,  f_ N(n)\, \e((n+k)^2\alpha)\bigr|\geq\ve/(4\ell).
$$
Let $K$ and $\delta$ be given by Lemma~\ref{lem:katai} with $\ve/(4\ell)$ substituted for $\ve$.
From this lemma and the last estimate it follows that  there exist  $k\in \Z$ and   primes $p,p'$ with $p<p'<K$ such that
$$
\bigl|\E_{n\in [\tN]} \one_I(n)\, \e((p^2-p'^2)n^2\alpha +2(p-p')kn\alpha)\bigr|\geq\delta
$$
where $I$ is the interval  $I:=\{n\in [N]\colon pn, p'n, pn+k,p'n+k\in [N]\}$ (its length is necessarily  greater than  $\delta N$).
 We interpret this formula by saying that
 $$
 ((p^2-p'^2)n^2\alpha +2(p-p')kn\alpha)_{n\in[\wt N]} \ \text{ is not ``well'' equidistributed on the torus}.
 $$
 Using Weyl-type results (see for example \cite[Proposition 4.3]{GT12a}) we get that  there exist positive integers $Q:=Q(\ve), R:=R(\ve)$ such that
\begin{equation}\label{E:major}
\Bigl|\alpha -\frac pq\Bigr|\leq \frac R{\tN^2}\
\text{ for some }\  p\in\Z\
\text{ and  some $q$ with } \
1\leq q\leq Q.
\end{equation}
In other words, $\alpha$ belongs to a ``major arc'', that is,  it is close to a rational with a small denominator.
\subsubsection*{Major arcs} We factorize the sequence  $(n^2\alpha)_{n\in [\tN]}$ as follows
$$
n^2\alpha=\epsilon(n)+\gamma(n), \quad \text{ where }  \quad \epsilon(n):=n^2\big(\alpha-\frac{p}{q}\big), \ \ \gamma(n):=n^2\frac{p}{q}.
$$
The sequence $\epsilon$ varies slowly (this follows from \eqref{E:major}) and the sequence $\gamma$ has period $q$. After
partitioning the interval $[\tN]$ into sub-progressions where
 $\epsilon(n)$ is almost constant and $\gamma(n)$ is constant,  it is not hard to deduce
from \eqref{eq:dis-square2} that
$$
|\E_{n\in [\tN]}\one_P(n)\cdot  f_ {N,\un}(n)|>
\frac{1}{10}\frac{\ve^2}{QR}
$$
for some arithmetic progression $P\subset[\tN]$. From this
and Lemma~\ref{lem:Us-intels} we deduce that
$$
\norm{ f_ {N,\un}}_{U^2(\Z_\tN)}> \frac{1}{C} \frac{\ve^2}{QR}
=:\theta(\ve)
$$
where $C$ is a positive universal constant. This contradicts Property~\eqref{it:decomU23} of  Theorem~\ref{th:Decomposition-U2}
and completes the
proof.
\end{proof}
Our next goal is for every $s\geq 2$ to replace the estimate in \eqref{eq:disc-square'} with the estimate
 $\displaystyle\norm{ f_ {N,\un}}_{U^s(\tZN)}\leq\ve$. To this end, we  shall see (using the inverse theorem in \cite{GTZ12c}) that it
 suffices to get a strengthening of  \eqref{eq:disc-square'}
 where
the place of
 $(\e(n^2\alpha))$ takes
any $s$-step nilsequence $(\Phi(a^n\cdot e_X))$ where $\Phi$ is a
function
 on an $s$-step nilmanifold $G/\Gamma$ with Lipschitz norm at most $1$ and $a\in G$.
 This is an immensely more difficult task and  it is carried out in the next five sections.

\section{Nilmanifolds and the inverse theorem for the $U^s$-norms}
\label{sec:Nil}

In this section, we review some basic concepts on nilmanifolds and also record the inverse theorem
 for the $U^s$-norms. As most notions will be used subsequently to state theorems from~\cite{GT12a}
 we follow the notation used in~\cite{GT12a}.
  Mal'cev basis were introduced in~\cite{Mal} and  proofs of foundational properties of Mal'cev basis and rational subgroups used in this article  can be found in \cite{Co82}.
\subsection{Basic definitions}
\label{subsec:nilmanifolds}
 Let  $G$ be a connected, simply connected, $s$-step nilpotent Lie group and $\Gamma$ be a discrete co-compact subgroup.
 The \emph{commutator subgroups} $G_i$ of $G$ are defined by
$G_0=G_1:=G$ and $G_{i+1}:=[G,G_i]$ for $i\in \N$. We have $G_{s+1}=\{1_G\}$.

 The compact manifold $X:=G/\Gamma$ is called an \emph{ $s$-step nilmanifold}.
 In some cases  the degree of nilpotency does not play a particular role and we refer to $X$ as  a \emph{nilmanifold}.

 We view elements of $G/\Gamma$ as ``points'' on the
nilmanifold $X$ rather than equivalence classes, and denote  them by
$x,y,$ etc. The projection in $X$ of the unit element $1_G$ of $G$ is called the \emph{base point} of $X$ and is denoted by $e_X$.
 The action of $G$ on $X$ is denoted by
$(g,x)\mapsto g\cdot x$. The \emph{Haar measure} $m_X$ of $X$ is the unique probability measure on $X$ that is invariant under this action.

\begin{convention}
We never consider ``nude'' nilmanifolds, but assume (often implicitly) that some supplementary structure is given. First, every nilmanifold $X$ can be represented as a quotient $G/\Gamma$ in several ways, but we assume that one of them is fixed.
Moreover, we assume that $G$ is endowed with a \emph{rational filtration},   a \emph{Mal'cev basis} for $X$ adapted to the filtration, and the corresponding Riemannian metric. We  define these objects next.
\end{convention}

\begin{definition}[\cite{GT12a}]
Let $G$ be a connected, simply connected $s$-step nilpotent
Lie group, and let $\Gamma$ be a discrete co-compact subgroup of $G$.
A \emph{rational subgroup} of $G$ is a connected, simply connected, closed subgroup $G'$ of $G$ such that $G'\cap \Gamma$ is co-compact in $G'$.
\end{definition}
It is known that the commutator subgroups $G_i$ are rational (see for example \cite[theorem~5.1.1 and Corollary~5.2.2]{Co82}).
More properties of rational subgroups are given in Appendix~\ref{ap:A}.
\begin{definition}[\cite{GT12a}]
A \emph{filtration}
$G_\bullet$ on $G$ is a sequence of rational subgroups
$$G_\bullet:=\bigl\{ G=G^{(0)}=G^{(1)}\supset G^{(2)}\supset\dots\supset G^{(t)}\supset G^{(t+1)}=\{1_G\}=G^{(t+2)}=\cdots\bigr\}
$$
which has the property that $[G^{(i)},G^{(j)}]\subset G^{(i+j)}$ for all integers
$i,j\geq 0$.
The \emph{degree} of the filtration $G_\bullet$ is the smallest integer $t$ such that $G^{(t+1)}=\{1_G\}$. The filtration is \emph{rational} if the groups $G^{(i)}$ are rational.

The \emph{natural filtration} is the \emph{lower central series} that consists  of the commutator subgroups $G_i$, $i\geq 0$,  of $G$.
 It is a rational filtration and has degree $s$ when $G$ is $s$-step nilpotent.

\end{definition}
 Let $G^{(i)}$, $i\geq 0$ be a filtration. We remark that as  $[G,G^{(i)}]\subset G^{(i)}$,  we have that $G^{(i)}$ is a normal subgroup of $G$ for $i\in \N$.  Since
$ G^{(2)}\supset G_2$, the quotient group $G/G^{(2)}$ is Abelian and isomorphic to $\R^q$ for some $q\geq 0$.

\begin{definition}
Let $X:=G/\Gamma$ be an $s$-step nilmanifold  and $G_\bullet$ be a filtration. We let $m:=\dim(G)$ and
$m_i:=\dim(G^{(i)})$ for $i\geq 0$.
A basis $\CX:=\{\xi_1,\dots,\xi_m\}$ for the Lie algebra $\mathfrak g$ of $G$ over $\R$ is called a
\emph{Mal'cev basis for $X$} adapted to $G_\bullet$ if the following conditions hold:
\begin{enumerate}
\item
 For each $j=0,\dots,m-1$, $\mathfrak h_j:=\Span(\xi_{j+1},\dots,\xi_m)$ is a Lie
algebra ideal in $\mathfrak g$, and hence $H_j:=\exp(\mathfrak h_j)$  is a normal Lie subgroup of $G$;
\item
For every $0\leq i\leq s$ we have $G^{(i)}=H_{m-m_i}$;
\item
\label{it:mal3}
 Each $g\in G$ can be written uniquely as
 $\exp(t_1\xi_1)\exp(t_2\xi_2)\cdots\exp(t_m\xi_m)$ for $t_1,\dots,t_m\in\R$;
\item
\label{it:MalcevGamma}
$\Gamma$ consists precisely of those elements which, when written in the above form,
have all $t_i\in\Z$.
\end{enumerate}
\end{definition}
 It follows from~\eqref{it:mal3} that the map
$$
 (t_1,\dots,t_m)\mapsto\exp(t_1\xi_1)\cdots \exp(t_m\xi_m)
$$
is a diffeomorphism from $\R^m$ onto $G$;
the numbers $t_1,\dots,t_m$ associated to an element $g\in G$ in this way are called the \emph{coordinates of $g$ in the basis $\CX$}.

  It can be shown that there exists a Mal'cev basis adapted to any   rational filtration $G_\bullet$; see the remark following Proposition~2.1 in \cite{GT12a} which is based on \cite[Proposition 5.3.2]{Co82}   and ultimately on~\cite{Mal}.

\subsection{The metric on $G$ and on $X$}
Let $\mathfrak g$ be endowed with the Euclidean structure making  the Mal'cev basis $\CX$  an orthonormal basis.
 This induces a  Riemannian structure on $G$ that is invariant under right translations. The group  $G$ is endowed with the corresponding geodesic distance, which we denote by  $d_G$. This distance is invariant under right translations\footnote{We remark that in \cite{GT12a} the authors use a different metric, but it is equivalent with $d_G$, and the implied constant depends only on $X$ and the choice of the Mal'cev basis,  so this does not make any difference for us.}.

Let the space $X:=G/\Gamma$ be endowed with the quotient metric
$d_X$. Writing $p\colon G\to X$ for the quotient map, the metric
$d_X$ is defined by
$$
d_X(x,y)=\inf_{g,h\in G}\{d_G(g,h)\colon p(g)=x,\ p(h)=y\}.
$$
Since  $\Gamma$ is discrete it follows that the infimum
is attained.

For $k\in \N$ and $\Phi\in\CC^k(X)$,  $\norm \Phi_{\CC^k(X)}$ denotes the usual $\CC^k$-norm.
We frequently  use
the fact that if $\Phi$ belongs to $\CC^1(X)$, then $\norm
\Phi_{\lip(X)}\leq\norm \Phi_{\CC^1(X)}$. We also use  some simple facts that follow immediately from the smoothness of the multiplication $G\times G\to G$.

\begin{lemma}
\label{lemp:ap1} Let $F$ be a bounded subset  of $G$. There exists
a constant $C>0$ such that
\begin{enumerate}
\item
For every $g,h,h'\in F$ we have  $d_G(gh, gh')\leq Cd_G(h,h')$;
\item
For every  $x,x'\in X$ and
 $g\in F$ we have
$d_X(g\cdot x,g\cdot x')\leq Cd_X(x,x')$;
\end{enumerate}
Moreover, for every $k\in \N$ there exists a constant $C_k$ such that
\begin{enumerate}
\setcounter{enumi}{2}
\item
  For   every  $\Phi\in\CC^k(X)$ and  $g\in F$,
writing $\Phi_g(x):=\Phi(g\cdot x)$, we have $\norm{\Phi_g}_{\CC^k(X)}\leq
C_k\norm \Phi_{\CC^k(X)}$.
\end{enumerate}
\end{lemma}

\begin{lemma}
\label{lem:GammaGj}
There exists $\delta>0$ such that, for $j=1,\ldots, s$, if $\gamma\in\Gamma$ and $u\in G_j$ satisfy $d_G(\gamma,u)<\delta$, then $\gamma\in G_j$.
\end{lemma}
\begin{proof} This follows immediately from the classical fact that $G_j\cap \Gamma$ is co-compact in $G_j$.
\end{proof}

\subsection{Sub-nilmanifolds}
\label{subsec:subnil}
We proceed with some basic facts regarding sub-nilmanifolds.

\begin{definition}
 A \emph{sub-nilmanifold} of $X$ is a nilmanifold $X':=G'/\Gamma'$ where $G'$ is a rational subgroup of $G$ and $\Gamma':=G'\cap \Gamma$.   We constantly identify $X'$  with the closed sub-nilmanifold $G'\cdot e_X$ of $X$. In particular, the base point   $e_{X'}$ of $X'$ is identified with the base point $e_X$ of $X$.
 \end{definition}

 \begin{convention}
If $X:=G/\Gamma$ is a nilmanifold and $G$ is endowed with a rational filtration $G_\bullet$ and if $X':=G'/(G'\cap\Gamma)$ is a sub-nilmanifold of $X$, then we implicitly assume that $G'$ is endowed with the \emph{induced rational filtration} defined by $G'^{(j)}:=G'\cap G^{(j)}$, $j\in \N$.
\end{convention}
In general, there is no natural method to define a Mal'cev basis for $X'$ from a Mal'cev basis for $X$ and we cannot assume that the inclusion map $X'\to X$ is  an isometry. However, this inclusion is a smooth embedding and it follows that there exists a positive constant $C:=C(X',X)$ such that
\begin{equation}
\label{eq:embeddingXprime}
C\inv d_{X}(x,y)\leq d_{X'}(x,y)\leq C d_{X}(x,y)\ \text{ for every  }x,y\in X'.
\end{equation}

\subsection{Vertical and horizontal torus and corresponding characters}
\label{subsec:vertical}
Let $X:=G/\Gamma$ be an $s$-step nilmanifold, $m:=\dim(G)$ and $r:=\dim(G_s)$.   The \emph{vertical torus} is the connected compact Abelian Lie group $G_s/(G_s\cap\Gamma)$.
  Since the restriction to  $G_s\cap\Gamma$ of the action of $G$ on $X$ is trivial,  the vertical torus acts on $X$, and this action is clearly free. It follows from the definition of the distance on $X$ that the vertical torus acts by isometries.
Let $\wt X$ be the quotient of $X$ under this action. Then $\wt X$ is an $(s-1)$-step nilmanifold and can be written as
$\wt X:=\wt G/\wt\Gamma$ where $\wt G:=G/G_s$ and $\wt\Gamma:=\Gamma/(\Gamma\cap G_s)$.

We endow $\wt G$ with a  Mal'cev basis such that, in Mal'cev coordinates, the projection $G\to\wt G$ is given by $(t_1,\dots,t_m)\mapsto(t_1,\dots,t_{m-r})$.
The distance $d_{\wt G}$ on $\wt G$ corresponding to this basis is the quotient distance  induced by  $d_G$, and the distance $d_{\wt X}$ on $\wt X$ is the quotient distance induced by  $d_X$.

Furthermore, the  Mal'cev basis of $X$ induces an
isometric identification between $G_s$ and $\R^r$, and thus of the
vertical torus endowed with the quotient metric, with $\T^r$ endowed
with its usual metric.  In order to avoid confusion,  elements of $G_s$ are written as $u,v,\ldots$ when we use the multiplicative notation, and as $\bu=(u_1,\dots,u_r)$, $\bv=(v_1,\dots v_r)$,\ldots  when we identify $G_s$ with $\R^r$ and use the additive notation; the same convention is used for the vertical torus.

\begin{definition}[Vertical characters and nilcharacters] Let $X$ be an $s$-step nilmanifold and $r:=\dim(G_s)$.  A \emph{vertical character} is a continuous group homomorphism $\xi\colon G_s\to\T$ with a trivial restriction on $G_s\cap\Gamma$;
 it can also be thought of  as a  character of the vertical torus.
 The group of vertical characters is then identified with $\Z^r$, where  $\bh=(h_1,\dots,h_r)\in\Z^r$ corresponds to the group homomorphism $\xi$ given by
$$
\xi(\bt):=\bh\cdot\bt \bmod 1= h_1t_1+\dots+h_rt_r  \bmod{1}\ \text{ for }\ \bt=(t_1,\dots,t_r)\in\R^r=G_s.
$$
We define the  \emph{norm of $\xi$} to be
$$
\norm\xi:=\norm\bh=|h_1|+\dots+|h_r|.
$$
A function $\Phi\colon X\to \C$ is a    \emph{nilcharacter with frequency $\bh$} if $\Phi(\bt\cdot
x)=\e(\bh\cdot\bt)\, \Phi(x)$ for every $\bt\in G_s$ and
every $x\in X$.\footnote{In \cite{GT12a}  a function with this property is said to have vertical oscillation $\bh$.} 
\end{definition}

\begin{definition}[Maximal torus and horizontal characters]
Let $X:=G/\Gamma$ be an $s$-step nilmanifold,  let $m:=\dim(G)$ and $m_2:=\dim(G_2)$.
 The Mal'cev basis induces an
isometric identification between the \emph{horizontal torus}
$G/(G_2\Gamma)$, endowed with the quotient metric, and $\T^{m-m_2}$,
endowed with its usual metric. A \emph{horizontal character} is a
continuous group homomorphism $\eta\colon G\to\T$  with  a trivial
restriction on $\Gamma$.  In Mal'cev coordinates, it is given by
$\eta(x_1,\dots, x_m)=\ell_1x_1+\dots+\ell_{m-m_2}x_{m-m_2} \bmod 1$ for $(x_1,\dots,x_m)\in \R^m$, where $\ell_1,\dots,\ell_{m-m_2}$ are integers called the coefficients of $\eta$.   We let
$$
\norm{\eta}:=|\ell_1|+\cdots+|\ell_{m-m_2}|.
$$
The horizontal character $\eta$ factors   through the horizontal torus, and  induces a character
  given by  $\balpha\mapsto {\bf \ell}\cdot\balpha:=\ell_1\alpha_1+\dots+\ell_{m-m_2}\alpha_{m-m_2}$ for
$\balpha=(\alpha_1,\dots,\alpha_{m-m_2})\in\T^{m-m_2}$.
\end{definition}

\subsection{The $U^s$-inverse theorem}
\label{sec:inverse}
%
%
We are going to use the following inverse theorem of Green, Tao, and Ziegler that gives a criterion for checking that a function $a\colon \Z_N\to \C$ has $U^s$-norm bounded away from zero.
\begin{theorem}[Inverse theorem for the $U^s$-norms~\mbox{\cite[Theorem 1.3]{GT12b}}]
\label{th:inverse}
Let    $s\geq 2$ be an integer   and $\ve$ be a positive real that is smaller than $1$.
There exist an $(s-1)$-step nilmanifold $X:=G/\Gamma$ and    $\delta>0$, both depending on $s$ and $\ve $ only, such that the following holds: For every $N\in \N$, if $a\colon \Z_N\to\C$ has modulus at most $1$ and satisfies
$$
\norm a_{U^s(\Z_N)}\geq\ve,
$$
then there exist
 $g\in G$  and a function $\Phi\colon X\to\C$ with $\norm \Phi_{\lip(X)}\leq 1$, such that
$$
\bigl|\E_{n\in[N]} a(n)\, \Phi(g^n\cdot e_X)\bigr|\geq \delta.
$$
\end{theorem}
There are two differences between this theorem and the form it is stated in~\cite{GT12b}.
First, the result is stated in \cite{GT12b} with a finite family of nilmanifolds instead of a single one;
but as the authors of \cite{GT12b} also remark one can use a single nilmanifold.   More importantly, the result is stated for the norm $U^s[N]$ instead of the norm $U^s(\Z_N)$. The present statement follows immediately from the result in \cite{GT12b} and Lemma~\ref{lem:NormsUs} in the Appendix.

A sequence of the form $\Phi(g^n\cdot e_X)$ where $\Phi$ is only
assumed to be continuous is called  a \emph{basic nilsequence} in
\cite{BHK05}; if in addition we assume  that $\Phi$ is Lipschitz, then
we call it a
\emph{nilsequence of bounded complexity} a notion first used in \cite{GT08}.

Let us remark that for the partition regularity results of Sections~\ref{SS:quadratic} and \ref{SS:parametric}
we only need to use  the $U^3$-inverse theorem; an independent and much simpler proof of this inverse theorem
  can be found in \cite{GT08}.

\section{Quantitative equidistribution and factorization on nilmanifolds}
\label{sec:equidistribution}

 In this section we state a quantitative equidistribution result and a factorization theorem
 for polynomial sequences on nilmanifolds, both proved by Green and Tao in \cite{GT12a}, and also derive some  consequences that will be used later on.

\subsection{Polynomial sequences in a group}
\label{SS:PolyGroup}
We start with the definition of a polynomial sequence on an arbitrary group.
\begin{definition} Let
$G$ be a group  endowed with a filtration $G_\bullet$ and $(g(n))_{n\in \N}$ be a sequence in $G$. For $h\in \N$, we define the sequence $\partial_hg$ by $\partial_hg(n):=g(n+h)g(n)\inv$, $n\in \N$.
We say that the sequence $g$ is a polynomial sequence with coefficients in the filtration $G_\bullet$ if
$\partial_{h_i}\dots\partial_{h_1}g$ takes values in $G^{(i)}$ for every $i\in \N$ and  $h_1,\dots,h_i\in \N$. We write $\poly(G_\bullet)$ for the family of polynomial sequences with coefficients in $G_\bullet$.
If  the filtration $G_\bullet$  has  degree $d$  we say that the polynomial sequence has  \emph{degree at most} $d$.
\end{definition}

The following equivalent definition is given in~\cite[Lemma~6.7]{GT12a} (see also~\cite{L98,L02}):
\begin{equivdefinition}
 A polynomial sequence  with coefficients in the filtration $G_\bullet$  of degree $d$ is a sequence $(g(n))_{n\in \N}$ of the form
\begin{equation}
\label{eq:def-poly-seq}
g(n)=a_0a_1^na_2^{\binom n2}\dots a_d^{\binom nd}\ \text{ where }\ a_j\in G^{(j)} \ \text{ for  }\ j=0,\ldots,d.
\end{equation}
\end{equivdefinition}
\begin{remarks}
$(1)$
The extra flexibility coming from the fact that we consider
  polynomial sequences with respect
 to arbitrary filtrations, not just the natural one,  will be  used in an essential way.

$(2)$
The set
$\poly(G_\bullet)$ is a group with operation the pointwise multiplication of sequences~\cite[Proposition 6.2]{GT12a}, a result initially due to Leibman~\cite{L98, L02} when $G_\bullet$ is the natural filtration.

$(3)$
It can be seen (see~\cite[Remarks below Corollary 6.8]{GT12a}) that if $G$ is $s$-step nilpotent, then
 every sequence $g\colon \N\to G$ of the form
$g(n):=a_1^{p_1(n)} \cdots a_k^{p_k(n)}$ with $a_1,\ldots, a_k\in G$
and $p_1,\ldots, p_k\in \Z[t]$ of degree at most $d$, is a polynomial sequence with coefficients in some filtration $G_\bullet$ of $G$ of degree at most
$ds$.
\end{remarks}

When  $G=\T$, unless stated explicitly, we assume that $\T$  is endowed  with the filtration of degree
 $d\in\N$ given by $\T^{(j)}=\T$ for  $j\leq d$ and $\T^{(j)}=\{0\}$ for $j>d$.
In this case, a polynomial sequence of degree at most $d$  in $\T$  can be expressed alternatively in two different ways:
\begin{align}
\label{eq:coef-poly}
\phi(n)&=\alpha_0+\alpha_1\binom n1+\alpha_2\binom n2+\dots+\alpha_d\binom nd \\
\label{eq:coef-poly2}
&= \alpha'_0+\alpha'_1n+\alpha'_2n^2+\dots+\alpha'_dn^d
\end{align}
for some $\alpha_0,\alpha_1,\alpha_2,\dots,\alpha_d,\alpha'_0,\alpha'_1,\alpha'_2,\dots,\alpha'_d\in\T$. The choice between these two representations depends on the problem at hand. Similar comments apply  for polynomial sequences in $\T^m$.

\begin{definition}(Smoothness norms)
Let $(\phi(n))_{n\in\N}$ be  a polynomial sequence of degree at most $d$  in $\T$
of the form ~\eqref{eq:coef-poly}.  For every $N\in\N$ we define the \emph{smoothness norm}
$$
\norm\phi_{C^\infty[N]}:=\max_{1\leq j\leq d}N^j\norm{\alpha_j},
$$
where, as usual,  $\norm{\alpha}$ denotes the distance of $\alpha$ to the nearest integer.
\end{definition}
If  a   polynomial sequence is  given by \eqref{eq:coef-poly2}, then we define
$
\norm\phi_{C^\infty[N]}':=\max_{1\leq j\leq d}N^j\norm{\alpha_j'}.
$ It is easy to check that there
exist positive constants $c:=c(d), C:=C(d)$ such that
$$
c\norm\phi_{C^\infty[N]}\leq \norm\phi_{C^\infty[N]}'\leq C\norm\phi_{C^\infty[N]},
 $$
 so the two norms  can be used interchangeably without affecting our arguments.

The smoothness norm is designed to capture the concept of a slowly-varying polynomial sequence.
 Indeed, for every $d\in \N$ there exists $C:=C(d)>0$ such that, for every  polynomial sequence $\phi$ of degree $d$
  on $\T$ (or $\T^m$) and every $n\in[N]$, we have
$$
\norm{\phi(n)-\phi(n-1)}\leq \frac CN\norm\phi_{C^\infty[N]}.
$$

It is immediate to check that for $1\leq N'\leq N$ and for  every  polynomial sequence  $\phi$ of degree at most $d$, we have
$$
\norm\phi_{C^\infty[N']}\leq\norm\phi_{C^\infty[N]}\leq \bigl(\frac N{N'}\bigr)^d\norm\phi_{C^\infty[N']}.
$$
We can also show that
 for $b\in\Z$ and $\phi_b(n):=\phi(n+b)$, we have
\begin{equation}
\label{eq:normPsib}
\norm{\phi_b}_{C^\infty[N]}
\leq
 \big( \frac {N+1}{N}\big)^{|b|}\,\norm{\phi}_{C^\infty[N]} .
\end{equation}
Indeed,  let us write $\displaystyle\phi_b(n)= \sum_{j=0}^d\beta_j\binom nj$. By a direct computation, we get
$$
\text{for }b\geq 0,\ \beta_i=\sum_{j=0}^{d-i}\binom b j\alpha_{i+j}\ ;\ \ \text{for }b<0,\
\beta_i= \sum_{j=0}^{d-i}(-1)^j\binom{-b-1}j \alpha_{i+j},
$$
where, as usual, $\binom n p= 0$  for $p>n$. Hence,
for $b\geq0$ and $i=1,\ldots, d$, we have
$$
N^i\norm{\beta_i}\leq N^i\sum_{j=0}^{d-i}\binom b j\norm{\alpha_{i+j}}
\leq\norm{\phi}_{C^\infty[N]} \sum_{j=0}^{d-i}\binom b j\frac{1}{N^j} \leq
\norm{\phi}_{C^\infty[N]} \big(\frac{N+1}{N}\big)^b.
$$
For $b<0$ we get a similar estimate with $-b-1$ in place of $b$. In both cases the asserted estimate
\eqref{eq:normPsib} follows immediately.

The next lemma is a modification of a particular case of \cite[Lemma 8.4]{GT12a}.
\begin{lemma}
\label{lem:8.4}
Let $d,q,r, N\in\N$  and $a,b$ be integers with  $a\neq 0$, $|a|\leq q$, and $|b|\leq rN$.
There exist $C:=C(d,q,r)>0$ and $\ell:=\ell(a,d)\in \N$  such that if  $\phi\colon\N\to\T$ is a
polynomial sequence of degree at most $d$ and  $\psi$ is given by $\psi(n):=\phi(an+b)$, then
$$
\norm{\ell \phi}_{C^\infty[N]}\leq C \norm{\psi}_{C^\infty[N]}.
$$
\end{lemma}
\begin{proof}
Writing $\phi_b(n):=\phi(n+b)$ and using \eqref{eq:normPsib} and that $|b|\leq rN$, we get
$$
\norm{\phi}_{C^\infty[N]}\leq \big(\frac {N+1}N \big)^{|b|}\norm{\phi_b}_{C^\infty[N]}
\leq C_1\norm{\phi_b}_{C^\infty[N]}
$$
for some $C_1:=C_1(r)$. Furthermore, since $\psi(n)=\phi_b(an)$,  one easily checks that
$$
\norm{|a|^d \phi_b}_{C^\infty[N]}\leq |a|^{d-1} \norm{\psi}_{C^\infty[N]}.
$$
Combining the above we get the asserted estimate for $\ell:=|a|^d$ and $C:=C_1 q^d$.
\end{proof}
\subsection{The quantitative Leibman theorem}
We are going to work with  the following  notion of  equidistribution on a nilmanifold:
\begin{definition}
Let $X:=G/\Gamma$ be a  nilmanifold,
$N\in \N$, $(g(n))_{n\in[N]}$ be a finite sequence in $G$, and $\delta>0$.
 The sequence $(g(n)\cdot e_X)_{n\in[N]}$ is \emph{totally $\delta$-equidistributed in $X$}, if for every arithmetic progression $P\subset[N]$ and for every Lipschitz function $\Phi$ on $X$ with $\norm\Phi_{\lip(X)}\leq 1$ and $\int_X\Phi\,dm_X=0$, we have
\begin{equation}\label{E:deftotequi}
\bigl|\E_{n\in[N]}\one_P(n)\, \Phi(g(n)\cdot e_X)\bigr|\leq\delta.
\end{equation}
\end{definition}

\begin{remark}
The distance on $X$, and as a consequence the notion of equidistribution of a sequence in $X$, depends on the choice of a Mal'cev basis on $G$, which in turn depends on the chosen rational filtration $G_\bullet$. As remarked in Section~\ref{subsec:subnil}, if $X'$ is a sub-nilmanifold of $X$, then  there is no natural choice for the Mal'cev basis of $X'$ and thus a sequence in $X$ that is $\delta$-equidistributed in $X$ is only $(C\delta)$-equidistributed in $X'$, where the constant $C$ depends on the choice of the two Mal'cev basis.
\end{remark}
To avoid confusion we remind the reader of the following convention that we make throughout the article:
\begin{convention}
 If $X:=G/\Gamma$ is a nilmanifold,  $G$ is implicitly endowed with some rational filtration $G_\bullet$. A polynomial sequence in $G$ is always assumed to have coefficients in this filtration, that is, it belongs to $\poly(G_\bullet)$.
As the degree of a polynomial sequence in $G$  is bounded by $d$ where $d$ is the degree of $G_\bullet$,
all statements below implicitly impose a restriction on the degree of the polynomial sequence under consideration.
\end{convention}
The next result gives a convenient criterion for establishing equidistribution properties of polynomial sequences
of nilmanifolds.
\begin{theorem}[Quantitative Leibman Theorem~\mbox{\cite[Theorem 2.9]{GT12a}}]
\label{th:Leibman}
Let  $X:=G/\Gamma$ be a nilmanifold and $ \ve>0$. There exists
$D:=D(X, \ve)>0$ such that the following holds: For every
$N\in\N$,  if $g\in\poly(G_\bullet)$ and  $(g(n)\cdot e_X)_{n\in[N]}$ is not totally  $\ve$-equidistributed in $X$,
  then there exists
a non-trivial horizontal character $\eta$ such that
$$
0<\norm\eta\leq D\quad \text{ and }\quad \norm{\eta\circ
g}_{C^\infty[N]}\leq D.
$$
\end{theorem}
 \begin{remarks} (1) For every  $g\in\poly(G_\bullet)$ and  every horizontal character $\eta$, the sequence $\eta\circ g$ is a polynomial sequence  in $\T$ of degree at most  $d$, where $d$ is the degree of $G_\bullet$.

(2) Theorem~\ref{th:Leibman}  will be used
in this form
but it is proved in~\cite{GT12a} under the stronger hypothesis that
the sequence is not ``\emph{$\ve$-equidistributed in $X$}'', meaning, \eqref{E:deftotequi} fails for $P:=[N]$.
 We deduce  Theorem~\ref{th:Leibman} from this result next.
\end{remarks}
\begin{proof}
Since the sequence $(g(n)\cdot e_X)_{n\in[N]}$ is not totally  $\ve$-equidistributed in $X$, there exist an arithmetic progression $P\subset[N]$ and a Lipschitz function $\Phi$ on $X$ such that
$$
\norm\Phi_{\lip(X)}\leq 1,\ \ \int_X\Phi\,dm_X=0, \ \text{ and }\
\bigl|\E_{n\in[N]}\one_P(n)\, \Phi(g(n)\cdot e_X)\bigr|\geq \ve.
$$
We write $P=\{an+b\colon n\in [N']\}$ where $N'$ is the length of $P$,  $a$  is its step, and $b\in [N]$.
Note  that one necessarily has   $N'\geq\ve N$,  thus   $a\leq 1/\ve$.
Then
$$
\bigl|\E_{n\in[N']}\Phi(h(n)\cdot e_X)\bigr|\geq \ve
$$
where  $h(n):=g(an+b)$. Hence,    the sequence $(h(n)\cdot e_X)_{n\in[N']}$ is not $\ve$-equidistributed in $X$.
Note also  that $h\in\poly(G_\bullet)$; this follows from the first definition in Section~\ref{SS:PolyGroup}.  Using the variant of  Theorem~\ref{th:Leibman} that is proved in  \cite{GT12a}, we deduce that there exists $D:=D(X, \ve)>0$
and a non-trivial  horizontal character $\theta$ such that $\norm\theta\leq D$ and $\norm{\theta\circ h}_{C^\infty[N']}\leq D$.
Writing $\phi(n):=\theta(g(n))$ and $\psi(n):=\theta(h(n))$ we have
$\psi(n)=\phi(an+b)$. Using Lemma~\ref{lem:8.4}  with $q:=1/\ve$ and $r:=1$
we get that  there exist  $C:=C(d,\ve)>0$ and $\ell:=\ell(a,d)\in\N$ such that
$$
\norm{\ell\cdot \theta\circ g}_{C^\infty[N]}=\norm{\ell \phi}_{C^\infty[N]}\leq C\norm\psi_{C^\infty[N]}
\leq C\bigl(\frac N{N'}\bigr)^d
\norm\psi_{C^\infty[N']}\leq C\ve^{-d}D
$$
where $d$ is the degree of the filtration $G_\bullet$. Letting $\eta:=\ell\, \theta$ we have
$\norm{\eta}\leq D \ell$ and the result follows.
\end{proof}

We are also going to use frequently the following   converse of Theorem~\ref{th:Leibman}:
\begin{lemma}[A  converse to  Theorem~\ref{th:Leibman}]
\label{lem:Leibman_Inverse}
 Let $X:=G/\Gamma$ be a
nilmanifold. There exists $c:=c(X)>0$ such that for
 every $D\in \N$ and every  sufficiently large $N\in \N$, depending only on $D$ and $X$,  the following holds: If $g\in\poly(G_\bullet)$
and there exists a non-trivial horizontal character $\eta$ of $X$
with $\norm\eta\leq D$ and $\norm{\eta\circ g}_{C^\infty[N]}\leq D$,
then the sequence $(g(n)\cdot e_X)_{n\in[N]}$ is not totally
$(cD^{-2})$-equidistributed in $X$.
\end{lemma}
\begin{proof}
Let $d$ be the degree of the filtration $G_\bullet$.
Since $\norm{\eta\circ g}_{C^\infty[N]}\leq D,$ we have
$$
\eta(g(n))=\sum_{0\leq j\leq d}\alpha_i\binom nj \ \text{ for some }
\ \alpha_0,\dots,\alpha_d\in\T\ \text{ with } \ \norm{\alpha_j}\leq \frac{D}{N^j},\
\text{ for } j=1,\ldots,d,
$$
$$
\bigl|\e\big(\eta(g(n))\bigr)-\e(\alpha_0)\bigr|\leq \frac{1}{2},\ \ \text{ for }\
1\leq n\leq c_1\frac{N}{D},
$$
for some positive constant  $c_1:=c_1(d)$.
Suppose  that $N\geq 4D/c_1$. Then
$$
\bigl|\E_{n\leq \lfloor c_1N/D\rfloor }\e\bigl(\eta(g(n))\bigl)\bigr|\geq
\frac{1}{2},
$$
which gives
$$
\bigl|\E_{n\in [N] }\one_{[\lfloor c_1N/D\rfloor]}(n)
\e\bigl(\eta(g(n))\bigr)\bigr|\geq
\frac{c_1}{2D}-\frac{1}{N}\geq \frac{c_1}{4D}.
$$
Furthermore, since $\norm\eta\leq D$,  the function $x\mapsto
\e(\eta(x))$, defined on $X$, is Lipschitz with constant at most
$C_1D$ for some $C_1:=C_1(X)$, and    has integral $0$ since $\eta$ is
a non-trivial horizontal character. Therefore,  the sequence
$(g(n)\cdot e_X)_{n\in[N]}$ is not totally $(cD^{-2})
$-equidistributed in $X$ where  $c:= c_1/(4C_1)$, completing the proof.
\end{proof}

\subsection{Some consequences of the quantitative Leibman theorem}

 We give two corollaries that are going to be used in subsequent sections.
  We caution the reader that  in both statements  the polynomiality of the sequence and
   the quantitative Leibman theorem are  used in a crucial way.
\begin{corollary}
\label{cor:XXprime}
Let $X:=G/\Gamma$ be a nilmanifold and  $\Gamma'$ be a discrete subgroup of $G$ containing $\Gamma$. Let $X':=G/\Gamma'$ and suppose that $G$ is endowed with the same rational filtration $G_\bullet$ for both nilmanifolds $X$ and $ X'$.
For every $\ve>0$ there exists $\delta:=\delta(X,X',\ve)>0$ such that the  following holds: For every sufficiently large $N\in \N$, depending only on $X$, $X'$, $\ve$,
if  $g\in\poly(G_\bullet)$ and  $(g(n)\cdot e_{X})_{n\in[N]}$ is totally $\delta$-equidistributed in $X'$, then  $(g(n)\cdot e_X)_{n\in[N]}$ is totally $\ve$-equidistributed in $X$.
\end{corollary}
\begin{proof}
Since $\Gamma$ is co-compact and $\Gamma'$ is closed  in $G$, $\Gamma$ is co-compact in $\Gamma'$; since $\Gamma'$ is discrete, $\Gamma$ has finite index in $\Gamma'$.
 It follows that  the natural projection $X\to X'$ is finite to one
 and there exists an  $\ell\in\N$, depending on $X$ and $X'$, such that $\gamma^\ell\in \Gamma$ for every $\gamma\in\Gamma'$. Therefore, for every horizontal character $\eta$ of $X$ (meaning a group homomorphism $G\to\T$ with a trivial restriction to $\Gamma$), $\eta^\ell$ has a trivial restriction to $\Gamma'$ and thus is a horizontal character of $X'$.

Suppose now that the sequence $(g(n)\cdot e_X)_{n\in[N]}$ is not totally $\ve$-equidistributed in $X$.
By Theorem~\ref{th:Leibman} there exist $D:=D(X,\ve)$ and a horizontal character $\eta$ of $X$ with
$\norm\eta\leq D$ and $\norm{\eta\circ g}_{C^\infty[N]}\leq D$. Then $\eta^\ell$ is a horizontal character of $X'$ such that $\norm{\eta^\ell}\leq C\ell\norm\eta\leq C\ell D $  for some $C:=C(X,X')$\footnote{The constant $C$ arises from the fact that the identifications $G/(G_2\Gamma)=\T^{m-m_2}$ and $G/(G_2\Gamma')=\T^{m-m_2}$ are different.}
 and $\norm{\eta^\ell\circ g}_{C^\infty[N]}\leq \ell D$.
Lemma~\ref{lem:Leibman_Inverse} then provides a $\delta:=\delta(X,X',\ve)>0$ such that the sequence
$(g(n)\cdot e_{X})_{n\in[N]}$ is not totally $\delta$-equidistributed in $X'$.
This completes the proof.
\end{proof}

Properties of rational elements are given in Appendix~\ref{ap:A}, we only recall here that an element $g$ of $G$ is rational if $g^n\in\Gamma$ for some $n\in\N$.

\begin{corollary}
\label{cor:equid-alpha}
Let $X:=G/\Gamma$ be a  nilmanifold and $G'$ be a rational subgroup of $G$. Let  $X':=G'/(G'\cap\Gamma)$,    $\alpha$ be a rational element of $G$,  $G'_\alpha:=\alpha\inv G'\alpha$, and $X'_\alpha:= G'_\alpha/(G'_\alpha\cap\Gamma)$.
 Then there  exists a function $\rho_{X,X',\alpha}\colon \R_+\to\R_+$ with $\rho_{X,X',\alpha}(t)\to 0$ as $t\to 0^+$
  such that  the following holds:
For every sufficiently large $N\in \N$, depending only on $X,X',\alpha$, if $h\in\poly(G'_\bullet)$ and  $(h(n)\cdot e_{X})_{n\in [N]}$ is totally $t$-equidistributed in $X'$, then $( \alpha\inv h(n)\alpha\cdot e_{X})_{n\in [N]}$ is totally $\rho_{X,X',\alpha}(t)$-equidistributed in $X'_\alpha$.
\end{corollary}
\begin{remark}
Recall that since  $G'$ is a rational subgroup of $G$, $G'\cap\Gamma$ is co-compact in $G'$ and  $X'$ is identified with the sub-nilmanifold $G'\cdot  e_{X}$ of $X$.
 By Lemma~\ref{lem:Ap0}, $G'_\alpha$ is also a rational subgroup of $G$ and thus
$G'_\alpha\cap\Gamma$ is co-compact in $G'_\alpha$ and $X'_\alpha= G'_\alpha\cdot  e_{X}$. Furthermore, we have
$h_\alpha\in \poly(G'_{\alpha\bullet})$ where  $h_\alpha(n):=\alpha\inv h(n)\alpha$ and
$G'_{\alpha\bullet}:=\alpha\inv G'_\bullet\alpha$.
\end{remark}
\begin{proof}
To ease notation, in this proof we leave the dependence on $X$, $X'$, $\alpha$ implicit.

We start by using Lemma~\ref{lem:finite-index} in the Appendix, it gives   that  $G'\cap\Gamma\cap (\alpha\inv \Gamma\alpha)$ has finite index in the two groups $G'\cap\Gamma$ and $G'\cap (\alpha\inv\Gamma\alpha)$.
In particular, $G'\cap\Gamma\cap (\alpha\inv\Gamma\alpha)$ is discrete and co-compact in $G'$. We let $ \wt{X}_\alpha:=G'/(G'\cap\Gamma\cap (\alpha\inv\Gamma\alpha))$.

By Corollary~\ref{cor:XXprime},  there exists a function $\psi\colon\R_+\to\R_+$ with $\psi(t)\to 0$ as $t\to 0^+$,  such that the following holds: If $N$ is sufficiently large and the polynomial sequence
$(h(n))_{n\in[N]}$ in $G'$ is such that the sequence
$(h(n)\cdot e_{X})_{n\in[N]}$ is totally $t$-equidistributed in $X'$, then  the sequence $(h(n)\cdot e_{\wt{X}_\alpha})_{n\in[N]}$ is totally $\psi(t)$-equidistributed in $\wt{X}_\alpha$.

Furthermore, $G'\cap \alpha\inv\Gamma\alpha$ is discrete and co-compact in $G'$ and we let
$X''_\alpha:=G'/(G'\cap (\alpha\inv\Gamma\alpha))$.
The natural projection $\wt{X}_\alpha\to X''_\alpha$  is Lipschitz with Lipschitz constant $C$. If the sequence $h$ is as above, then  the image of the sequence $(h(n)\cdot e_{\wt{X}_\alpha})_{n\in[N]}$ under this projection is
$(h(n)\cdot e_{ X''_\alpha})_{n\in[N]}$ and this sequence is totally $(C\psi(t))$-equidistributed in $X''_\alpha$.

The conjugacy map $g\mapsto \alpha\inv g\alpha$ is an isomorphism from $ G'$ onto $G'_\alpha$ and maps $\alpha\inv\Gamma\alpha$ onto $\Gamma$. Thus, it induces a diffeomorphism from the nilmanifold $X''_\alpha$ onto $X'_\alpha$. If the Mal'cev basis of  $X''_\alpha$ is chosen so that its image under the conjugacy is the Mal'cev basis of $X'_\alpha$, this diffeomorphism is an isometry.
Then the image of the sequence
$(h(n)\cdot e_{ X''_\alpha})_{n\in[N]}$  under this isometry is
$(\alpha\inv h(n)\alpha\cdot e_X)_{n\in[N]}$, and thus this last sequence is totally $(C\psi(t))$-equidistributed in $X'_\alpha$. This completes the proof.
\end{proof}


\subsection{The factorization theorem \cite{GT12a}}\label{sec:facto}
Recall that a nilmanifold $X:=G/\Gamma$ is implicitly endowed with a filtration $G_\bullet$. This
defines a Mal'cev basis for $X$ which in turn is used to define a metric on $X$.
 Following our conventions, every
  sub-nilmanifold $X':=G'/(G'\cap \Gamma)$ is endowed with the filtration $G'_\bullet$ given by $G'^{(j)}:=G'\cap G^{(j)}$ for every $j\in \N$.
 \begin{definition} Let $M,N\in \N$.
\begin{enumerate}
\item An element $\gamma\in G$ is {\em $M$-rational} if  $\gamma^m\in\Gamma$ for some integer $m$ with $1\leq m\leq M$; for more properties of rational elements see Appendix~\ref{ap:A}.
\item
 A finite sequence $(g(n))_{n\in[N]}$ in $G$ is {\em $M$-rational} if all its terms are $M$-rational.
\item
A finite sequence $(\epsilon(n))_{n\in[N]}$ in $G$ is {\em $(M,N)$-smooth} if
$d_G(1_G,\epsilon(n))\leq M$ for every $n\in[N]$ and  $d_G(\epsilon(n),\epsilon(n+1))\leq M/N$ for every $n\in[N-1]$.
\end{enumerate}
\end{definition}

The next result of Green and Tao~\cite{GT12a} is used multiple times subsequently.

\begin{theorem}[Factorization of polynomial sequences~\mbox{\cite[Theorem 1.19]{GT12a}}]
\label{th:FactoGT}
Let $X:=G/\Gamma$ be a  nilmanifold.
For  every $M\in \N$ there exists a finite family $\CF(M)$ of sub-nilmanifolds of $X$, that increase with $M$,  each of the form $X':=G'/\Gamma'$, where $G'$ is a rational subgroup of $G$ and $\Gamma':=G'\cap \Gamma$, such that the following holds:
For every function $\omega\colon \N\to\R^+$   there exists a positive integer  $M_1:=M_1(X,\omega)$, and for every $N\in\N$ and every polynomial sequence $g\in\poly(G_\bullet)$, there exist $M\in\N$ with $ M\leq M_1$, a sub-nilmanifold $X'\in \CF(M)$, and a factorization
 $$
g(n)=\epsilon(n)g'(n)\gamma(n), \quad n\in [N],
$$
where $\epsilon,g',\gamma\in \poly(G_\bullet)$ and
\begin{enumerate}
\item
\label{it:smooth0}
$\epsilon\colon [N]\to G$ is $(M,N)$-smooth;
\item
\label{it:equid0}
 $g'\in \poly(G'_\bullet)$  and the sequence $(g'(n)\cdot e_X)_{n\in[N]}$ is totally $\omega(M)$-equidistributed in $X'$
 with the metric $d_{X'}$ induced by the filtration $G'_\bullet$;
\item
\label{it:periodic0}
 $\gamma\colon [N]\to G$ is $M$-rational and $(\gamma(n)\cdot e_X)_{n\in [N]}$ has period at most $M$.
\end{enumerate}
\end{theorem}

\begin{remarks}

(1)
In~\cite{GT12a} this result is stated only for a function $\omega$ of the form $\omega(M)=M^A$ for some $A>0$ but the same argument shows that it holds for all functions $\omega\colon \N \to \R^+$ (in fact  this more general statement is  proven in \cite[Lemma~2.10]{GT10}).

 (2)
 In~\cite{GT12a}, the family $\CF(M)$ is defined as the collection of all sub-nilmanifolds $X':=G'/(G'\cap\Gamma)$ of $X$ admitting a Mal'cev basis that consists of  $M$-rational combinations of the elements of the Mal'cev basis $\CX$ of $X$.  We do not need this precise description of $\CF(M)$; what is important for us is   that
 each  $\CF(M)$ is finite and the  family
$(\CF(M))_{M\in \N}$   depends only on $X$ (and the filtration $G_\bullet$).
The number $M_1$ in Theorem~\ref{th:FactoGT} corresponds to the quantity written as $M_0^{O_{A,m,d}(1)}$ in \cite[Theorem 1.19]{GT12a}.
In our case, what is  important  is that  $M_1$
depends only on $X$ and on $\omega$.

(3)  Simple examples (see \cite[Section~1]{GT12a}) show  that if  $G_\bullet$ is the natural filtration in $G$, then
the
filtration $G'_\bullet$ may not be the natural one in $G'$. Furthermore,   even if we start with
a ``linear'' sequence   $(g^n)_{n\in[N]}$ in $G$, we  may end up with a ``quadratic'' sequence
$(h^{n^2})_{n\in[N]}$ in $G'$ where $h\neq \text{id}_G$.
\end{remarks}

\subsection{Some constructions  related to the factorization theorem}
\label{subsec:constr-fact}
We record here some terminology and constructions  related to
Theorem~\ref{th:FactoGT} that are used multiple times in the sequel.
Throughout, we assume that a nilmanifold $X:=G/\Gamma$ with a rational filtration $G_\bullet$ is given; then Theorem~\ref{th:FactoGT}  provides  for
any given  $M\in\N$ a  finite family of sub-nilmanifolds $\CF(M)$ of $X$ that increases with $M$.

By Corollary~\ref{cor:M-rat}, for every $M\in \N$ there exists  a
finite subset $\Sigma(M)$ of $G$, consisting of
$M$-rational elements, such that  every $M$-rational  element
$\beta\in G$ can be written as $\beta=\alpha\alpha'$ with $\alpha\in\Sigma(M)$ and $\alpha'\in\Gamma$.  We can assume  that $\one_G\in\Sigma(M)$.
Let $X':=G'/( G'\cap \Gamma)$ be a nilmanifold belonging to the family $\CF(M)$ and let
$\alpha$ belong to the finite set $\Sigma(M)$. By Lemma~\ref{lem:Ap0}, $\alpha\inv G'\alpha$ is a rational subgroup of $G$; we let $$
G'_\alpha:= \alpha\inv G'\alpha\ \text{ and }\ X'_\alpha:=G'_\alpha/(G'_\alpha\cap \Gamma)
=G'_\alpha\cdot e_{X}.
$$
Recall that $G'$ and each group $G'_\alpha$ is  endowed with the induced filtration given by $G'^{(i)}_\alpha:=G^{(i)}\cap G'_\alpha$. Since $G^{(i)}$ is a normal subgroup of $G$, we have $G'^{(i)}_\alpha= \alpha\inv G'^{(i)}\alpha$.
Finally,  let
$$
\CF'(M):=\{X'_\alpha\colon X'\in\CF(M),\ \alpha\in\Sigma(M)\}.
$$

\subsubsection{Defining the constant $H(X,M)$}
\label{subsubsec:defH}
By Lemma~\ref{lemp:ap1}, there exists a positive integer  $H:=H(X,M)$ with
 the following properties:
\begin{enumerate}
\item
\label{it:conj}
For every $\alpha\in \Sigma(M)$ and for every $g\in G$ with $d_G(g,1_G)\leq M$,  we have $d_G(\alpha\inv g\alpha,1_G)\leq Hd_G(g,1_G)$;

\item
\label{eq:kappa1}
For every $\alpha\in \Sigma(M)$, every $g\in G$  with
$d_G(g,\one_G)\leq M$,  and every  $x,y\in X$, we have  $d_X(g\alpha\cdot
x,g\alpha\cdot y)\leq H d_X(x,y)$;
\item
\label{it:Lip-H}
 Therefore, for every $\Phi\in \lip(X)$, every $\alpha\in \Sigma(M)$, and every $g\in G$  with $d_G(g,\one_G)\leq M$,  we have
 $\norm{\Phi_{g\alpha}}_{\lip(X)}\leq
H\norm\Phi_{\lip(X)}$ where $\Phi_{g\alpha}(x):=\Phi(g\alpha\cdot x)$.
\end{enumerate}

Note that the distance on a nilmanifold $X'_\alpha \in \CF'(M)$ is not the distance induced by its inclusion in $X$.
 However, the inclusion $X'_\alpha\subset X$ is a smooth embedding and thus we can assume that
\begin{enumerate}
\setcounter{enumi}{3}
\item
\label{eq:LipXXprime}
For  every nilmanifold $X'_\alpha\in \CF'(M)$ and for every
$x,y\in X'_\alpha,$ we have
$$
H\inv d_{X'_\alpha}(x,y)\leq d_X(x,y)\leq Hd_{X'_\alpha}(x,y).
$$
\end{enumerate}
By Corollary~\ref{cor:equid-alpha}, there exists a function $\rho_X\colon\N\times\R_+\to\R_+$, such that $\rho_X(M,t)$
decreases to $0$ as $t\to 0^+$ and $M$ is fixed, and  satisfies:
\begin{enumerate}
\setcounter{enumi}{4}
\item
\label{it:def-rho}
For every nilmanifold $X':=G'/(G'\cap \Gamma)\in\CF(M)$, for every $\alpha\in\Sigma(M)$ and every $t>0$, if $N\in \N$ is sufficiently large depending on $X$, $M$,  and $t$, and if
$(h(n))_{n\in[N]}$ is a polynomial sequence in $G'$ such that the sequence
$(h(n)\cdot e_X)_{n\in[N]}$ is totally $t$-equidistributed in $X'$, then
 the sequence  $(\alpha\inv h(n)\alpha)_{n\in\N}$ belongs to
 $\poly(G'_{\alpha\bullet})$ and  $(\alpha\inv h(n)\alpha\cdot e_{X})_{n\in[N]}$
is totally $\rho_X(M,t)$-equidistributed in $X'_\alpha$.
\end{enumerate}

\subsubsection{Correlation estimates and factorization}
\label{susubsec:fact}
Let $N\in \N$ and suppose that the sequence $g\colon \N\to G$ factorizes as
$$
g(n)=\epsilon(n)g'(n)\gamma(n), \quad n\in [N],
$$
where  $\epsilon$ is $(M,N)$-smooth,
 $g'\colon \N\to G$ is an arbitrary sequence,
$\gamma$ is $M$-rational and $(\gamma(n)\cdot e_X)_{n\in [N]}$ has   period at most $M$.
Note  that we do not assume that the sequence $g$ is polynomial and we do not impose an equidistribution assumption on $g'$.


We are often given a correlation estimate of the form
\begin{equation}\label{E:g-correlation}
|\E_{n\in [N]} a(n)\, \Phi(g(n)\cdot e_X)|\geq \delta,
\end{equation}
 for some $\delta\in(0,1)$,
 $a\colon[N]\to\C$  bounded by $1$, and  $\Phi\in \lip(X)$ with
$\norm\Phi_{\lip(X)}\leq 1.$
 We want to deduce a similar estimate with $g'(n)$, or some other closely related sequence,  in the place of $g(n)$. We do this as follows:

Let
\begin{equation}
\label{eq:def-L}
L:=\bigl\lfloor N \frac{\delta}{16\,H^2M^2}\rfloor, \qquad N_1:=\lceil 16\,H^2M^2/\delta\rceil
\end{equation}
where $H:=H(M,X)$ satisfies Properties~\eqref{it:conj}-\eqref{eq:LipXXprime} in Section~\ref{subsubsec:defH}.
Henceforth, we assume  that  $N\geq N_1$, then  $L\geq 1$ and
\begin{equation}
\label{eq:boundL}
N\frac \delta{32\,H^2M^2}\leq L\leq N\frac \delta{16\,H^2M^2}.
\end{equation}
Since $\delta\leq 1$ and $H\geq 1$, we have that
$$
2ML\leq N.
$$
Let $K$ be the (least) period of the sequence $(\gamma(n)\cdot e_X)_{n\in[N]}$ and let   $P\subset [N]$ be an arithmetic progression of step $K$ and length between $L$ and $2L$. Recall that $K\leq M$. Let $n_0$ be the smallest element of the progression $P$.
We write
\begin{gather}
\label{eq:defalpha}
\gamma(n_0)=\alpha\alpha'\ \text{ where }\ \alpha\in\Sigma(M)\ \text{ and }\ \alpha'\in\Gamma;\\
\label{eq:defg''}
g_\alpha'(n):= \alpha\inv g'(n)\alpha\ \text{ for }\ n\in[N];\\
\label{eq:defPhi'}
\Phi_\alpha(x):=\Phi( \epsilon(n_0)\alpha \cdot x)\ \text{ for }\ x\in X.
\end{gather}

For every  $n\in P$ we have
$$
\gamma(n)\cdot e_X=\gamma(n_0)\cdot e_X=\alpha\alpha'\cdot e_X,
$$
 hence
$$
\Phi(g(n)\cdot e_X)=
\Phi_\alpha\bigl(\alpha\inv\epsilon(n_0)\inv\epsilon(n)\, \alpha \, g_\alpha'(n)\, \alpha'\cdot e_X\bigr).
$$
Furthermore, for $n\in P$ we have $n-n_0=jK$ for some $j$ with $0\leq j<2L$ and thus $0\leq n-n_0\leq 2LK\leq 2LM$. Since the sequence $\epsilon$ is $(M,N)$-smooth we have
$d_G(\epsilon(n_0)\inv\epsilon(n),1_G)\leq 2M^2L/ N$, and by Property~\eqref{it:conj} above we get
$$
d_G(\alpha\inv \epsilon(n_0)\inv\epsilon(n)\alpha,1_G)\leq 2HM^2L/ N.
$$
Property~\eqref{it:Lip-H}  gives that
\begin{equation}
\label{eq:Phi'lip}
\norm{\Phi_\alpha}_{\lip(X)}\leq H.
\end{equation}
Combining the above, we get for every $n\in P$ that
$$
\bigl|a(n)\, \Phi(g(n)\cdot e_X)-a(n)\, \Phi_\alpha (g_\alpha'(n)\cdot e_X)\bigr|\leq 2H^2M^2\frac L{ N}.
$$
Averaging  on $[N]$, we get (recall that $P$ has at most $2L$ elements)
\begin{multline}
\label{eq:bound-on-P}
\bigl|\E_{n\in[N]}\one_{P}(n)\, a(n)\, \Phi(g(n)\cdot e_X) -
\E_{n\in[N]}\one_{P}(n)\, a(n) \, \Phi_\alpha(g_\alpha'(n)\cdot e_X )|\\
\leq
4H^2M^2\bigl(\frac L{ N}\bigr)^2
\leq \frac{\delta^2}{64\,H^2M^2}.
\end{multline}

Since  $N\geq N_1$ we have $L\geq 1$ and~\eqref{eq:boundL} holds.  Since
$2MK\leq 2ML\leq N$, we can  partition  the interval $[N]$ into arithmetic progressions of step $K$ and  length between $L$ and $2L$. The number of these progressions is bounded by $N/L$ and it  follows from~\eqref{E:g-correlation} that for one of them, say for $P_1$, we have
$$
\bigl|\E_{n\in[N]}\one_{P_1}(n)\, a(n)\, \Phi(g(n)\cdot e_X)\bigr|\geq\delta\frac LN.
$$
Estimate~\eqref{eq:bound-on-P} gives
\begin{equation}
\label{eq:New-Bound-on-P}
\bigl|\E_{n\in[N]}\one_{P_1}(n)\, a(n)\, \Phi_\alpha(g_\alpha'(n)\cdot e_X) \bigr|\geq\delta\frac LN-
\frac{\delta^2}{64\,H^2M^2}\geq \frac{\delta^2}{64\,H^2M^2}
\end{equation}
where the last inequality follows from~\eqref{eq:boundL}.

\section{Minor arc nilsequences-Preparatory work}
\label{S:minorarcs1}
 In this section and the next one our goal is to show that bounded multiplicative functions have small correlation with all minor-arc nilsequences,
 that is, nilsequences that arise from  totally equidistributed polynomial
 sequences on nilmanifolds.
This result is a central point in the proof of Theorems~\ref{T:DecompositionSimple} and \ref{T:DecompositionII} and it can be viewed as the higher order analogue of  Corollary~\ref{cor:katai}; linear sequences arising from numbers that are not well approximated by rationals with small denominators correspond to totally  equidistributed polynomial sequences on nilmanifolds.

 We remind the reader that for every nilmanifold $X:=G/\Gamma$, $G$ is endowed with a rational  filtration $G_\bullet$ and that polynomial sequences in $G$ are assumed to have coefficients in this filtration and in particular have  degree  bounded by the degree of  the filtration.

\begin{theorem}[Key discorrelation estimate]
\label{th:discorrelation}
Let
 $X:=G/\Gamma$  be a  nilmanifold   and $\tau>0$.
  There exist  $\sigma:=\sigma(X,\tau)>0$ and $N_0:=N_0(X,\tau)$ such that for every  $N\geq N_0$  the following holds:
Suppose that $g\in \poly(G_\bullet)$  and
\begin{equation}
\label{Assumption}
|\E_{n\in[N]} \one_P(n)\,
 f (n)\, \Phi(g(n+k)\cdot e_X)|\geq\tau,
\end{equation}
for some integer  $k\in [-N,N]$, $ f \in \CM$,  $\Phi\colon X\to \C$  with  $\norm \Phi_{\lip(X)}\leq 1 $ and $\int \Phi\, dm_X=0$, and
arithmetic progression  $P$  in $[N]$.
Then
\begin{equation}
\label{Conclusion}
\text{the sequence } \ (g(n)\cdot e_X)_{n\in [N]} \text{ is  not totally } \sigma\text{-equidistributed in }\  X.
 \end{equation}
\end{theorem}
 Recall that any sequence $g(n)=a_1^{p_1(n)}\cdots a_d^{p_d(n)}$, where $a_1,\ldots, a_d\in G$
and $p_1,\ldots,p_d\in \Z[t]$, belongs to $\poly(G_\bullet)$ for some appropriately chosen rational filtration $G_\bullet$. So
Theorem~\ref{th:discorrelation}  applies to all such sequences and
we get Theorem~\ref{T:DaboussiGeneral} as a direct consequence.

\subsection{Main ideas of the proof}\label{SS:ideasdisc}
Before moving to the rather delicate details  used
 in establishing Theorem~\ref{th:discorrelation}, we explain the skeleton of the proof for
  a variant of this result that
 contains some key  ideas used in the  proof of  Theorem~\ref{th:discorrelation} and
   suppresses
several technicalities
that obscure understanding. The main source of simplification in the  sketch given below
comes from  the   infinite nature  of the problem and the restriction to linear polynomial sequences.

Suppose that we seek to prove the following infinitary result:
If
 $X:=G/\Gamma$ is an $s$-step nilmanifold, with $G_s$ non-trivial,  and for some $a\in G$   the sequence
  $(a^n\cdot e_X)_{n\in \N}$  is  totally equidistributed in   $X$, then for every $\Phi\in C(X)$ with $\int \Phi\, dm_X=0$ we have
\begin{equation}
\label{E:lower1'}
\lim_{N\to +\infty}\sup_{ f \in \CM}|\E_{n\in[N]}
 f (n)\, \Phi(a^n\cdot e_X)|=0.
\end{equation}

 Using a vertical Fourier decomposition we reduce matters to the case where
 $\Phi$  is a nilcharacter of $X$ with non-zero frequency.
The orthogonality criterion of  K\'atai shows  that in order to prove \eqref{E:lower1'} it suffices
to show that for all distinct $p,q\in \N$ we have
  \begin{equation}\label{E:afterKatai}
\lim_{N\to +\infty} \E_{n\in[N]}
\Phi(a^{pn}\cdot e_X)\cdot \overline\Phi(a^{qn}\cdot e_X)=0.
  \end{equation}
  This motivates the study  of equidistribution properties of the sequence
  \begin{equation}\label{E:prodseq}
  \big((a^{pn}\cdot e_X, a^{qn}\cdot e_X)\big)_{n\in\N}.
  \end{equation}
  Lets take for granted that this sequence is equidistributed on a sub-nilmanifold $Y:=H/\Delta$ of $X\times X$.
We will be done if we manage to show that
\begin{equation}\label{E:intzero}
\int_Y (\Phi \otimes \overline{\Phi})\, dm_Y=0.
\end{equation}
This seemingly simple task turns out to be quite challenging, as the explicit structure of the possible nilmanifolds $Y$
seems very difficult to  determine (we have only managed to do this in the case $s=2$).\footnote{It is tempting to believe that  $Y:=H/\Delta$
where $H:=\{(g^pu_1,g^qu_2)\colon g\in G, u_1,u_2\in G_2\}$, $\Delta:=H\cap (\Gamma\times \Gamma)$. But this fails
even for the simplest non-Abelian nilmanifolds.
For example, let $G$ be the Heisenberg group, meaning,  $G:=\R^3$ with multiplication given by the formula  $(x,y,z)\cdot(x',y',z'):=(x+x',y+y'; z+z'+xy')$. Let $\Gamma:=\Z^3$,  $X:=G/\Gamma$, and $a:=(\alpha,\beta,0) $ with $\alpha$ and $\beta$ rationally independent. It is known that the sequence
 $(a^n\cdot e_X)_{n\in\N}$ is totally equidistributed in $X$. But if $\alpha$, $\beta$, and $\alpha \beta$ are rationally dependent, then for distinct integers $p,q$ the sequence  $(a^{pn}\cdot e_X,a^{qn}\cdot e_X)_{n\in\N}$ turns out to be  equidistributed in a sub-nilmanifold  of $X\times X$ that is strictly ``smaller'' than  $Y$.}
 Nevertheless, it is possible to extract some partial  information about the group $H$  defining the nilmanifold $Y$ that suffices for our purposes. To do this, as a first step, we study equidistribution properties of the  projection of  the sequence \eqref{E:prodseq} to the horizontal torus of the nilmanifold $X\times X$. Using the total equidistribution assumption for the sequence $(a^n\cdot e_X)_{n\in\N}$ we get the following inclusion
\begin{equation}\label{E:linearkey}
\{(g^p,g^q)\colon g\in G\}\subset H\cdot (G_2\times G_2).
\end{equation}
Taking iterated commutators of  the elements on the left we deduce the following key property
  \begin{equation}\label{E:keykey}
(u^{p^s},u^{q^s})\in H\  \text{ for every } \   u\in G_s.
\end{equation}
This easily implies that  the function $\Phi\otimes\overline{\Phi}$  is a nilcharacter of $Y$
   with non-zero frequency, hence it satisfies the sought-after zero integral  property stated in  \eqref{E:intzero}.

When the place of $a^n$ in \eqref{E:lower1'}
takes the  general polynomial sequence $g(n)$, the first few steps of our argument
remain the same. One important difference is that the inclusion \eqref{E:linearkey} fails in general  and
it has to be replaced with a more
complicated one.
For instance, suppose that
$g(n):=a^{n}b^{n^2}$ for some $a,b\in G$. Then we show that
there exist normal subgroups $G^1, G^2$  of $G$ such that $ G^1 \cdot G^2=G$
and
$$
\{(g_1^{p}g_2^{p^2},g_1^{q}g_2^{q^2})\colon g_1\in G^1, g_2\in G^2\}\subset H\cdot (G_2\times G_2).
$$
By taking iterated commutators of elements on the left we establish a property analogous to \eqref{E:keykey}, namely,
$$
\text{the set }\  U:=\{u\in G_s\colon (u^{p^j},u^{q^j})\in H \text{ for some } j\in \N\} \ \text{ generates } \ G_s.
$$
From this we can again extract the sought-after property \eqref{E:intzero}.

This gives a rather accurate  summary  of the skeleton of  the proof of Theorem~\ref{th:discorrelation} in the
idealized setting of infinitary mathematics.
 Unfortunately, the natural habitat of   Theorem~\ref{th:discorrelation}, is the world of finitary mathematics, and this
   adds
 a serious amount of technical complexity in the implementation of the previous plan.
 In Section~\ref{S:minorarcs2}  we state and prove two key ingredients
needed in  the proof of Theorem~\ref{th:discorrelation};  these  are   Propositions~\ref{P:PolysTorusQuant''} and
~\ref{P:alg}.  In the next section we combine these ingredients in order to implement
the previously sketched plan and finish the proof of Theorem~\ref{th:discorrelation}.

\subsection{Notation and conventions}
In this section and the next one,  $d$ and $m$ are positive integers representing the dimension of a torus $\T^m$ and the degree of a polynomial sequence on $\T^m$ respectively. Furthermore, $p,q$ are distinct positive integers.
We consider all these integers  as fixed throughout, and we stress that all other parameters introduced in this section  and the next  one depend implicitly on $d,m,p,q$.
This is not going to create problems for us, as in the course of proving Theorem~\ref{th:discorrelation}
 the integers $d,m,p,q$ can be taken to be  bounded by a constant that depends on $X$ and $\tau$ only.

We continue to represent elements of $\T^m$ and $\Z^m$ by bold letters. For $\bx\in\T^m$ and $\bh\in\Z^m$ we write $\norm \bx$ for the distance of $\bx$ from  $\bzero$, $\norm \bh=|h_1|+\dots+|h_m|$, and $\bh\cdot \bx=h_1x_1+\dots+h_mx_m$. Vectors consisting of $d$ elements of  $\T^m$ are written
$(\bg_1,\dots,\bg_d)$. Beware of the possible confusion between $\bg_j\in\T^m$ and $g_j$ representing the $j^{\text{th}}$-coordinate of $\bg\in \T^m$. Sequences in $\T^m$ are written as $\bg(n)$.

For finite sequences in $\T^m$ we use a slightly modified definition of smoothness: A finite sequence $(\bg(n))_{n\in[N]}$ on $\T^m$  is $(M,N)$-smooth if $\norm{\bg(n+1)-\bg(n)}\leq M/N$ for every $n\in[N-1]$.
An element $\balpha$ of $\T^m$ is $M$-rational if $n\balpha=0$ for some positive integer $n\leq M$. Throughout, by a sub-torus of $\T^m$ we mean a closed connected subgroup of $\T^m$ (perhaps the trivial one). A sub-torus of $\T^m$ is $M$-rational
 if its lift in $\R^m$ has a basis of vectors with integer coordinates of absolute value at most $M$.

\subsection{A property of  factorizations on the torus}
The next result will be used multiple times
in order to derive inclusions between various sub-tori of $\T^m$.

\begin{lemma}
\label{L:contained}
Let $L_1, L_2\in \N$ and $S_1,S_2$ be two  sub-tori of  $\T^m$. There exist $\delta_1:=\delta_1(S_1,S_2,L_1,L_2)>0$ and $N_1:=N_1(S_1,S_2, L_1,L_2)$ such that the following holds:
Let $N\geq N_1$ be an integer  and $\bg\colon[N]\to \T^m$ be an arbitrary sequence that admits the following factorizations
\begin{equation}\label{E:same}
\bg(n)= \bepsilon_i(n)+\bg_i'(n)+\bgamma_i(n), \quad n\in [N], \qquad i=1,2,
\end{equation}
where
\begin{enumerate}
\item
\label{it:smooth-1}
$\bepsilon_i\colon [N]\to \T^m$ are $(L_i,N)$-smooth for $i=1,2$;
\item
\label{it:equid-1a}
 $(\bg_1'(n))_{n\in [N]}$ takes values on  $S_1$ and is totally $\delta_1$-equidistributed on $S_1$;

\item
\label{it:equid-1b} $(\bg_2'(n))_{n\in [N]}$ takes values on  $S_2$;
\item
\label{it:periodic-1}
 $\bgamma_i\colon [N]\to \T^m$   have period at most $L_i$
 for $i=1,2$.
 \end{enumerate}
Then $S_1\subset S_2$.
\end{lemma}

\begin{remark}
It is important for applications that  we do not
impose an equidistribution assumption on $\bg_2'$.
\end{remark}
\begin{proof}
Replacing $L_1$ and $L_2$ by $L:=L_1L_2$ we reduce to the case where $L_1=L_2=L$ and we can assume that the sequences $\bgamma_1$ and $\bgamma_2$ have the same period $Q\leq L$.

We argue by contradiction. Suppose that $S_1$ is not a subset of  $S_2$.
 Then there exists a (multiplicative)  character $\theta$ of $\T^m$, with values on the unit circle,
 such that
\begin{equation}\label{E:S12}
\theta(\bx)=1 \ \text{ for every } \ \ \bx\in S_2 \ \text{ and } \ \int\theta\ dm_{S_1}=0.
\end{equation}
Let $A:=A(S_1,S_2)$ be the Lipschitz constant of $\theta$.

Since  $\theta(\bx)=1$ for $\bx\in S_2$ and $\bg_2'$ takes values in $S_2$, identity \eqref{E:same} gives that
\begin{equation}
\label{E:313}
\theta( \bepsilon_3(n)+\bg_1'(n)+\bgamma_3(n))=1, \quad \text{ for } \ n\in [N],
\end{equation}
where $\bepsilon_3:=\bepsilon_1-\bepsilon_2$ is $(2L,N)$-smooth and $\bgamma_3:=\bgamma_1-\bgamma_2$
has period $Q$. Let $P$ be the arithmetic progression
$\{Q, 2Q,\ldots, Q\lfloor cN\rfloor\}$ where $c:=1/(4AL^2)$. The progression is well defined as long as
$N\geq N_1:=4AL^2$. Then
\begin{equation}
\label{E:gamma3}
\bgamma_3(n)\ \text{ is constant on }\  P
\end{equation}
and for $n\in P$ we have
\begin{equation}\label{E:esta}
|1-\theta\big(\bepsilon_3(n)-\bepsilon_3(Q)\big)|=|\theta\big(\bepsilon_3(n)\big)-\theta\big(\bepsilon_3(Q)\big)|
\leq A \norm{\bepsilon_3(n)-\bepsilon_3(Q)}
\leq 2AL\frac{n-Q}{N}\leq \frac{1}{2}.
\end{equation}
 We get
\begin{align*}
|\E_{n\in [N]} \one_P(n)\cdot \theta(\bg_1'(n))|
&=|\E_{n\in [N]} \one_P(n)\cdot \theta(\bg_1'(n)-\bg_1'(Q))|
\\
&=|\E_{n\in [N]} \one_P(n)\cdot \theta(\bg_1'(n)+\bgamma_3(n)-\bg_1'(Q)-\bgamma_3(Q))\text{ by~\eqref{E:gamma3}}
\\
&=|\E_{n\in [N]} \one_P(n)\cdot \theta(\bepsilon_3(Q)-\bepsilon_3(n))|
\text{ by~\eqref{E:313}}
\\
& \geq  \frac {\lfloor cN\rfloor}N\,\bigl(1-\frac 12\bigr)
>\frac 1{16\,AL^2}\text{ by~\eqref{E:esta} and since }|P|=\lfloor cN\rfloor.
\end{align*}
On the other hand,
since  by assumption $(\bg_1'(n))_{n\in [N]}$
 is totally $\delta_1$-equidistributed on $S_1$, $\theta$ has Lipschitz constant $A$, and  $\int \theta\, dm_{S_1}=0$  (by  \eqref{E:S12}), we get
$$
|\E_{n\in [N]} \one_P(n)\cdot \theta(\bg_1'(n))|\leq A \delta_1.
$$
Hence, for $\delta_1:= 1/(16\,A^2L^2)$ we get a contradiction, completing the proof.
\end{proof}

\subsection{Simultaneous factorization of  monomials  on the torus}
In the course of proving Proposition~\ref{P:PolysTorusQuant''} we need to factorize simultaneously several polynomial sequences on the torus
and also make sure that the output of the factorization preserves some of the properties of the original sequences. As  it does not seem possible to  extract such information from Theorem~\ref{th:FactoGT}, using it as  a black box, we modify its proof on the torus
in order to get the following result that suits our needs:
\begin{theorem}[Simultaneous factorization of  monomials  on the torus]
\label{th:FactoGTmulti}
For  every $M\in \N$ there exists a finite family $\CF_2(M)$ of sub-tori of $\T^m$, that increases with  $M$,
such that the following holds:
For every function $\omega_2\colon \N\to\R^+$   there exist a positive integer  $M_2:=M_2(\omega_2)$ such that  for every
$N\in\N$   and every $\balpha_1,\ldots, \balpha_d \in \T^m$, there exist $M\in\N$ with $ M\leq M_2$,
sub-tori $T_1,\ldots, T_d$ of $\T^m$, belonging to the family $\CF_2(M)$, and for $j=1,\ldots, d,$ a factorization
\begin{equation}
\label{eq:facto1}
n^j\balpha_j=\beeta_j(n)+n^j\balpha_j'+\btheta_j(n), \quad n\in [N],
\end{equation}
where
\begin{enumerate}
\item
\label{it:smooth}
$\beeta_j\colon [N]\to \T^m$ is $(M,N)$-smooth;
\item
\label{it:equid}
$\balpha'_j\in T_j$ and
$(n^j\balpha'_j)_{n\in[N]}$  is totally $\omega_2(M)$-equidistributed on $T_j$;
\item
\label{it:periodic}
$\btheta_j\colon [N] \to \T^m$  is $M$-rational and   has period at most $M$.
\end{enumerate}
\end{theorem}
\begin{proof}
For every $M\in\N$, we let
$$
\CF_2(M):=\{T\subset \T^m\colon T \text{ is an  $M$-rational torus}\}.
$$

The proof is going to be carried out by an iterative procedure that  terminates after finitely many steps.

\subsection*{The data}
At each step $i=1,2,\dots,$ we have a constant $C_i:=C_i(\omega_2)$ and for $j=1,\dots,d$ a $C_i$-rational torus $T_{j,i}\subset\T^m$ and factorizations
\begin{equation}
\label{eq:decompi-1j}
 n^j\balpha_j=\beeta_{j,i}(n)+n^j\balpha_{j,i}+\btheta_{j,i}(n), \quad n\in [N],
\end{equation}
where
\begin{enumerate}[(a)]
\item
\label{it:smooth1}
$\beeta_{j,i}\colon [N]\to \T^m$ is $(C_{i},N)$-smooth;
\item
\label{it:equid1}
  $\balpha_{j,i}\in T_{j,i}$;
\item
\label{it:periodic1}
 $\btheta_{j,i}\colon [N]\to \T^m$ is $C_i$-rational and has   period at most $C_{i}$.
\end{enumerate}

 \subsection*{Initialization} We initialize our data.
 Let $C_{1}:=1$ and  for $j=1,\ldots,d$ let   $T_{j,1}:=\T^m\in\CF_2(C_{1})$ and $\balpha_{j,1}:=\balpha_j$.
We have the trivial factorization  $n^j\balpha_j = \beeta_{j,1}(n)+n^j\balpha_{j,1}+\btheta_{j,1}(n)$, where the sequence $\beeta_{j,1}$ is identically zero and thus $(C_{1},N)$-smooth, and the sequence $\btheta_{j,1}$ is identically zero and thus
 $C_{1}$-rational and has period at most $C_{1}$.

\subsection*{Test of termination}
If   $(n^j\balpha_{j,i})_{n\in[N]}$ is totally $\omega_2(C_{i})$-equidistributed on $T_{j,i}$
for $j=1,\ldots, d$,   then we set  $M:=C_{i}$, and  terminate the process\footnote{Note that this is necessarily the case when all the tori $T_{j,i}$ are trivial.}.
If not, we proceed to the next step.

\subsection*{Iteration} Our assumption is that there exists   $j_0\in \{1,\ldots, d\}$
such  that the sequence $(n^{j_0}\balpha_{j_0,i})_{n\in[N]}$ is not totally $\omega_2(C_{i})$-equidistributed on
the torus  $T_{j_0,i}$.
 Since $T_{j_0,i}$ is $C_{i}$-rational,
by   Theorem~\ref{th:Leibman} applied for  the torus $T_{j_0,i}$ there exist $A_1:=A_1(C_{i},\omega_2)>0$
 and a non-trivial character $\eta$ of $T_{j_0,i}$, of the form
 $\bx\mapsto\bk\cdot\bx$ for some $\bk\in\Z^m$,  with
$\norm{\bk}\leq A_1$ and $\norm{\bk\cdot \balpha_{j_0,i}}\leq A_1/N^{j_0}$.
We write $T_{j_0,i+1}^*$ for the kernel of $\eta$ in $T_{j_0,i}$, and $T_{j_0,i+1}$ for the connected component of $0$ in $T_{j_0,i+1}^*$. Then the torus  $T_{j_0,i+1}$  is $A_2$-rational, for some $A_2:= A_2(C_{i},\omega_2)$. We can write   $\balpha_{j_0,i}=\bbeta+\balpha^*$,  where
$\balpha^*\in T_{j_0,i+1}^*$ and $\norm\bbeta\leq A_3/N^{j_0}$ for some  $A_3:=A_3(C_{i},\omega_2)$. Furthermore,  we can
write  $\balpha*=\balpha_{j_0,i+1}+\bdelta$, where $\balpha_{j_0,i+1}\in T_{j_0,i+1}$ and $\bdelta$ is $A_4$-rational for some $A_4:=A_4(C_{i},\omega_2)$.
We let
$$
C_{i+1}:=\max\{C_{i}+dA_3, A_4C_{i}\}.
 $$
 Since $\balpha_{j_0,i}=\bbeta+\balpha_{j_0,i+1}+\bdelta$, using \eqref{eq:decompi-1j} we get
$$
n^{j_0}\balpha_{j_0}= \beeta_{j_0,i+1}(n) + n^{j_0}\balpha_{j_0,i+1}+ \btheta_{j_0,i+1}(n), \quad n\in [N],
$$
where
$$
\beeta_{j_0,i+1}(n):=\beeta_{j_0,i}(n)+ n^{j_0}\bbeta, \quad \btheta_{j_0,i+1}(n):= \btheta_{j_0,i}(n)+n^{j_0}\bdelta.
$$
Note that
$(\beeta_{j_0,i+1}(n))_{n\in[N]}$ is $C_{i+1}$-smooth
 and
$(\btheta_{j_0,i+1}(n))_{n\in[N]}$ is $C_{i+1}$-rational and has   period at most $C_{i+1}$.

For $j\neq j_0$ we do not modify the factorization of the step $i$, that is, we let
$$
\beeta_{j,i+1}:=\beeta_{j,i}, \quad \balpha_{j,i+1}:=\balpha_{j,i}, \quad
\btheta_{j,i+1}:=\btheta_{j,i}, \   \text{ for } j\neq j_0.
$$
Since $C_{i+1}\geq C_{i}$, we have that  $(\beeta_{j,i+1}(n))_{n\in[N]}$ is $C_{i+1}$-smooth and
$(\btheta_{j,i+1}(n))_{n\in[N]}$ is $C_{i+1}$-rational and has   period at most $C_{i+1}$. We have  thus produced for
 $j=1,\ldots, d$  factorizations similar to~\eqref{eq:decompi-1j}, with $i+1$ substituted for $i$ that satisfy Properties \eqref{it:smooth1}-\eqref{it:periodic1}.

\subsection*{The output}
At each step, the dimension of exactly one of the tori $T_{j,i}$ decreases by one, and thus the iteration stops after $k\leq dm$ steps
at which point the above described test has a positive outcome.   For $M:=C_k$,
we obtain  the factorizations~\eqref{eq:facto1}, satisfying the required properties~\eqref{it:smooth}, \eqref{it:equid}, \eqref{it:periodic}. Moreover, for $j=1,\ldots,d$, the torus $T_{j,k}$ is $M$-rational and thus belongs to the family $\CF_2(M)$.  Finally, at each step,
 $C_{i+1}$ is bounded by a quantity that depends only on $C_{i}$ and on $\omega_2$, and thus $M\leq M_2$ for some $M_2:=M_2(\omega_2)$. This completes the proof.
\end{proof}

\subsection{Quantitative equidistribution of product sequences on the torus}

Given a ``sufficiently'' totally equidistributed sequence $(\bg(n))_{n\in[N]}$ on some torus $\T^m$
we study here equidistribution properties of the product sequence   $(\bg(pn),\bg(qn))_{n\in[N]}$ on $\T^m\times \T^m$ where
$p,q$ are distinct positive integers. Our goal is to show that the product sequence is  ``sufficiently'' equidistibuted on a sub-torus of $\T^m\times \T^m$
that contains an ample supply   of ``interesting'' elements;  for instance, we show that this sub-torus
 contains non-diagonal elements of $\T^m\times \T^m$. We start with a simple but key observation:

\begin{lemma}
\label{L:PolysTorusQuant}
 Let $\ve_3>0$.  There exist $\delta_3:=\delta_3(\ve_3)$ and $N_3:=N_3(\ve_3)$  such that the  following holds:
 Let $N\geq N_3$ and for $j=1,\ldots,d$, let $\balpha_j\in\T^m$  and suppose that the sequence $(n^j\balpha_j)_{n\in [N]}$  takes values on  a sub-torus $T_j$ of $\T^m$ and
is  totally $\delta_3$-equidistributed on $T_j$.
 Then
\begin{enumerate}
\item
\label{it:equid11}
   The sequence $(n\balpha_1+\dots+n^d\balpha_d)_{n\in [N]}$  is totally $\ve_3$-equidistributed on the torus $T:=T_1+\dots+T_d$.
   \item
\label{it:equid12}
For $j=1,\ldots,d$, the sequence  $(n^j(p^j\balpha_j, q^j\balpha_j))_{n\in [N]}$ is totally $\ve_3$-equi\-distri\-buted on
the sub-torus  $T_{j,p,q}:=\{ (p^j\bx,q^j\bx): \bx\in T_j \}$ of $\T^{2m}=\T^m\times\T^m$.
\end{enumerate}
\end{lemma}
\begin{proof}
We prove \eqref{it:equid11}.
Let $\bg(n):=n\balpha_1+\dots+n^d\balpha_d$ and suppose
that the sequence  $(\bg(n))_{n\in [N]}$  is not totally $\ve_3$-equidistributed on the torus $T$.
Applying Theorem~\ref{th:Leibman} on this torus we get that there exists
a constant $C_1:=C_1(\ve_3)$ and   $\bk\in \Z^m$ with $\norm\bk\leq C_1$ such that  $\bk\cdot \bx\neq 0$ for some $\bx\in T$  and
\begin{equation}\label{E:smoothj}
\norm{\bk\cdot \balpha_j}\leq C_1/N^j \ \text{ for  }\ j=1,\ldots,d.
\end{equation}
Then, for some $j\in \{1,\ldots, d\}$ and some $\bx\in T_j$ we have $\bk\cdot \bx\neq 0$. Using this, relation  \eqref{E:smoothj}, and applying
Theorem~\ref{lem:Leibman_Inverse} for the torus $T_j$, we get  that for some
 $\delta_j':=\delta_j'(\ve_3)>0$ the sequence
$(n^j\balpha_j)_{n\in[N]}$ is not  totally $\delta_j'$-equidistributed on $T_j$. Hence, for  $\delta':=\min\{\delta'_1,\ldots, \delta'_d\}$
 (in place of $\delta_3$) Property~\eqref{it:equid11}
is satisfied.

We prove \eqref{it:equid12}.  Suppose that for some $j\in\{1,\ldots,d\}$,
the sequence $(n^j(p^j\balpha_j, q^j\balpha_j))_{n\in [N]}$ is not
 totally $\ve_3$-equidistributed on the torus $T_{j,p,q}$.
Applying Theorem~\ref{th:Leibman} on this torus we get that there exists
a  positive real $C_2:=C_2(\ve_3)$ and   $(\bk_1,\bk_2)\in \Z^{m}\times\Z^m$ with
$\norm{\bk_1}+\norm{\bk_2}\leq C_2$, such that
\begin{equation}\label{E:k1k2}
\bk_1\cdot\bx_1+\bk_2\cdot \bx_2\neq 0 \ \text{ for some } \ (\bx_1,\bx_2)\in T_{j,p,q}
\end{equation}
 and that
$\norm{(\bk_1,\bk_2)\cdot(p^j\balpha_j, q^j\balpha_j)}\leq C_2/N^j$,
or equivalently,
\begin{equation}\label{E:pq}
\norm{(p^j\bk_1+q^j\bk_2)\cdot \balpha_j}\leq C_2/N^j.
\end{equation}
Furthermore, writing $(\bx_1,\bx_2)=(p^j\bx,q^j\bx)$ for some $\bx\in T_j$ we get by \eqref{E:k1k2} that
$(p^j\bk_1+q^j\bk_2)\cdot \bx\neq 0$ and thus $p^j\bk_1+q^j\bk_2\neq 0$. On the other hand,
$\norm{p^j\bk_1+q^j\bk_2}\leq (p^d+q^d)C_2$,
and \eqref{E:pq} combined with Theorem~\ref{lem:Leibman_Inverse} for the torus $\T^m$ gives that for some $\delta''_j:=\delta''_j(\ve_3)$ the sequence
$(n^j\alpha_j)_{n\in[N]}$ is not  totally $\delta''_j$-equidistributed.
Hence, for  $\delta'':=\min\{\delta''_1,\ldots, \delta''_d\}$
 (in place of $\delta_3$) Property~\eqref{it:equid12}
is satisfied.

Letting $\delta_3:=\min\{\delta', \delta''\}$ completes  the proof.
\end{proof}

Combining Lemma~\ref{L:contained},   Lemma~\ref{L:PolysTorusQuant}, and Theorem~\ref{th:FactoGTmulti},
we prove the following factorization result on the torus that is crucial for our purposes:
\begin{proposition}[Equidistribution of $\bh(n)$ in $\T^m$]
\label{P:PolysTorusQuant''}
For  every $M\in \N$ there exists a finite family $\CF_4(M)$ of sub-tori of $\T^m$, that increases with $M$,
 and
for every  function $\omega_4\colon \N\to\R^+$   there exist  positive integers  $M_4:=M_4(\omega_4)$,  $N_4:=N_4(\omega_4)$, and a positive real $\delta_4:=\delta_4(\omega_4)$,  such that the following holds:
 Let $N\geq N_4$ be an integer,  $\balpha_1, \ldots, \balpha_d \in\T^m$,
and   $\bg(n)=\balpha_1n+\dots+\balpha_dn^d$,
 and suppose
 that
\begin{equation}
\label{eq:HypProp6.5}
\text{ the sequence } (\bg(n))_{n\in [N]} \text{ is totally  } \delta_4\text{-equidistributed on } \T^m.
\end{equation}
 Then there exist $M\in\N$ with $ M\leq M_4$, and sub-tori $T_j$, $j=1,\ldots,d$, of $\T^m$, belonging to the family $\CF_4(M)$,
such that
\begin{equation}\label{E:T1Td}
T_1+\dots+T_d=\T^m,
\end{equation}
and   the sequence   $(\bh(n))_{n\in [N]}$  on $\T^{2m}$ defined by
 $
 \bh(n):=(\bg(pn),\bg(qn))
 $
 can be factorized as follows
$$
\bh(n)=\bepsilon(n)+\bh'(n)+\bgamma(n), \quad n\in [N],
$$
where $\bepsilon(n)$, $\bh'(n)$, $\bgamma(n)$ are polynomial sequences on $\T^{2m}$ such that
\begin{enumerate}
\item
\label{it:smooth2}
$\bepsilon\colon [N]\to \T^{2m}$ is $(M,N)$-smooth;
\item
\label{it:equid2}
 $(\bh'(n))_{n\in [N]}$ takes values and  is  totally $\omega_4(M)$-equidistributed  on the sub-torus
\begin{equation}
\label{eq:defRT1Td}
R_{T_1,\ldots,T_d}:=\{(p\bx_1+\dots+p^d\bx_d\,,\, q\bx_1+\dots+q^d\bx_d)\colon \bx_j\in T_j \text{ for } j=1,\ldots,d\}
\end{equation}
of $\T^{2m}$;
\item
\label{it:periodic2}
 $\bgamma\colon [N]\to \T^{2m}$ is $M$-rational and  has period at most $M$.
 \end{enumerate}
\end{proposition}
\begin{proof}
Throughout this proof when we write
``for every sufficiently large $N$'', we mean for every $N\in \N$ that is larger than a constant that depends on  $\omega_4$.

 Let  $\omega_4'\colon \N\to \R^+$  be a  function that will be specified later and depends only on  $\omega_4$ (its defining properties
are given by \eqref{E:defomega3a} and \eqref{E:defomega3b} below).

  For every $M\in\N$, Theorem~\ref{th:FactoGTmulti} applied on $\T^m$ provides a finite family $\CF_2(M)$ of sub-tori of $\T^m$, that increases with $M$, and we define as $\CF_4(M)$  the family spanned by $\CF_2(M)$, the torus  $\T^m$, and is invariant under addition of tori.
Applying  Theorem~\ref{th:FactoGTmulti} on $\T^m$ with the function $\omega_4'$, we get
 a positive integer $M_4:=M_4(\omega'_4)$, such that the following holds:
  For every $\balpha_1,\ldots, \balpha_d \in \T^m$ there exists $M\in \N$ with
$M\leq M_4$ and  sub-tori  $T_1,\ldots, T_d$ of $\T^m$  belonging to the family  $\CF_4(M)$, and
for $j=1,\ldots, d,$  factorizations
\begin{equation}
\label{E:factorus1}
 n^j\balpha_j=\beeta_j(n)+n^j\balpha'_j+\btheta_j(n), \quad n\in [N],
\end{equation}
where
\begin{enumerate}[(a)]
\item
\label{it:smooth7}
$\beeta_j\colon [N]\to \T^m$ is $(M,N)$-smooth;
\item
\label{it:equid7}
$\balpha'_j\in T_j$ and the sequence
 $(n^j\balpha'_j)_{n\in [N]}$
 is
 totally $\omega_4'(M)$-equidistributed on $T_j$;
 \item
\label{it:periodic7}
 $\btheta_j\colon [N]\to \T^m$  is $M$-rational and  has period at most $M$.
 \end{enumerate}

 We are going to show that for an appropriate choice of $\omega_4'$ and $\delta_4$ we can use these data to get a factorization for $\bh$ that satisfies  Properties~\eqref{it:smooth2}-\eqref{it:periodic2} and \eqref{E:T1Td}. To this end, write
$$
\bh(n)=\bepsilon(n)+\bh'(n)+\bgamma(n), \quad n\in [N],
$$
where
\begin{gather*}
\bepsilon(n)
 := (\beeta_1(pn),\beeta_1(qn)) + (\beeta_2(pn),\beeta_2(qn))+
    \cdots + (\beeta_d(pn),\beeta_d(qn));\\
\bh'(n)
 :=n(p\balpha'_1,q\balpha'_1)+n^2(p^2\balpha'_2,q^2\balpha'_2)+\cdots
+n^d(p^d\balpha'_d,q^d\balpha'_d);\\
\bgamma(n)
 := (\btheta_1(pn),\btheta_1(qn)) + (\btheta_2(pn),\btheta_2(qn)) +\cdots +
    (\btheta_d(pn),\btheta_d(qn)).
\end{gather*}
Note that
$\bgamma$ is $M^d$-rational and has period at most $M^d$.
Replacing $M_4$ with $C M_4^{d}$ and $M$ with $C M^{d}$ for some constant $C$ that depends only on
$d,m,p,q$, we have that the sequence $\bepsilon$ is $(M,N)$-smooth, and
 Properties~\eqref{it:smooth2} and \eqref{it:periodic2} of the proposition are satisfied.

We move now to  Property~\eqref{it:equid2}. Note first that $\bh'(n)$ takes values on the torus $R_{T_1,\ldots,T_d}$. By \eqref{it:equid7}, for $j=1,\ldots,d$ the sequence
$(n^j\balpha'_j)_{n\in [N]}$ is
totally $\omega_4'(M)$-equidistributed on $T_j$.
Using this property and Part~\eqref{it:equid12} of Lemma~\ref{L:PolysTorusQuant},
we get that for every sufficiently large $N$, for $j=1,\ldots, d$, the sequence
\begin{equation}\label{E:pjalpha}
(n^j(p^j\balpha'_j,q^j\balpha'_j))_{n\in [N]}
\end{equation}
is
$\rho'(\omega_4'(M))$-equidistributed on the  torus
$$
\{(p^j\bx,q^j\bx): \bx\in T_j\},
$$
 where $\rho'\colon \R^+\to \R^+$ is a function
  that
decreases to  $0$ as $t\to 0^+$ and depends only on $d,m,p,q$. We are now in position to apply
Part~\eqref{it:equid11} of Lemma~\ref{L:PolysTorusQuant} on the
torus $\T^{2m}$ for the sequences in \eqref{E:pjalpha}. It
gives that the sequence $(\bh'(n))_{n\in [N]}$ is $\omega_4(M)$-equidistributed on
$R_{T_1,\ldots, T_d}$
as long as the function $\omega_4'\colon \N\to \R^+$ satisfies
\begin{equation}
\label{E:defomega3a}
\rho'(\omega_4'(M))\leq \delta_3(\omega_4(M)),
\quad \text{for every }\ M\in \N,
\end{equation}
where $\delta_3$ was defined in Lemma~\ref{L:PolysTorusQuant}.
As $\delta_3>0$ and $\rho'(t)\to 0$ as $t\to 0^+$, such an $\omega_4'$ exists.

It remains to establish \eqref{E:T1Td}, that is, that  $T_1+\dots+T_d=\T^m$.
To get this, we need to impose two additional conditions, one on $\omega_4'$ and one on
the degree of total equidistribution $\delta_4$ of $(g(n))_{n\in [N]}$ that was left unspecified until this point.
We choose $\omega_4'\colon \N\to \R^+$ so that    in addition to \eqref{E:defomega3a} it satisfies
\begin{equation}
\label{E:defomega3b}
 \omega_4'(M)\leq \min_{S_1, S_2\in \CF_4(M)}\{\delta_1(S_1,S_2,1, dM^d)\}, \quad \text{for every }\ M\in \N,
 \end{equation}
where $\delta_1(S_1,S_2,1,M^2)$ was defined on Lemma~\ref{L:contained}.  Now $M_4$ is well defined and
we  let
 \begin{equation}\label{E:delta4}
 \delta_4:=\min_{1\leq M\leq M_4}\{\omega_4'(M)\}.
 \end{equation}
We assume that Property~\eqref{eq:HypProp6.5} holds for this value of $\delta_4$ and   Properties
\eqref{it:smooth7}, \eqref{it:equid7}, \eqref{it:periodic7} hold for some $M\leq M_4$.

  Next, note that we have two factorizations for the sequence
 $(\bg(n))_{n\in [N]}$. The first is the trivial one: $\bg(n)=0+\bg(n)+0$  where $\bg(n)$ takes values in $\T^m$ and is totally $\omega_4'(M)$-equidistributed on $\T^m$ (by \eqref{E:delta4}).
 The second
 is given by \eqref{E:factorus1}:
$$
 \bg(n)=\beeta(n)+\bg'(n)+\btheta(n), \quad n\in[N],
$$
  where
\begin{gather*}
\beeta(n):=\beeta_1(n)+\dots+\beeta_d(n) \text{ is }(dM,N)
\text{-smooth \quad  (by \eqref{it:smooth7})};\\
\bg'(n):= n\balpha_1'+\dots+n^d\balpha_d'^d\text{  takes values in }  T_1+\dots+T_d
\quad \text{ (by \eqref{it:equid7})};\\
\btheta(n):=\btheta_1(n)+\dots+\btheta_d(n)\text{ is } M^d\text{-rational and has period at most }
M^d \quad\text{ (by \eqref{it:periodic7}).}
\end{gather*}
Since by assumption $\CF_4(M)$ contains $\T^m$, $T_1$, \ldots, $T_d$,  and is closed under addition of tori, we have  $\T^m$,  $T_1+\dots+T_d\in \CF_4(M)$. Furthermore,  by~\eqref{E:defomega3b} we have $\omega_4'(M)\leq \delta_1(\T^m,T_1+\dots+T_d,1,d M^d)$. Hence, Lemma~\ref{L:contained} is applicable, and gives that
 for every sufficiently large $N$ we have $\T^m\subset T_1+\dots+T_d$. It follows that $\T^m= T_1+\dots+T_d$ and the   proof is complete.
\end{proof}

\subsection{A key algebraic fact}
Our goal is to establish the following key property.
\begin{proposition}
\label{P:alg}
Let $G$ be an $s$-step nilpotent group and  $H$ be a subgroup of $G\times G$. Suppose
that there exist normal subgroups $G^1,\dots,G^d$  of $G$ such that
\begin{enumerate}
\item
$G=G^1\cdots G^d$;
\item
$\{(g_1^p\cdots g_d^{p^d}\,,\,g_1^{q}\cdots g_d^{q^d})\colon g_1\in G^1,\ldots, g_d\in G^d\}\subset  H\cdot (G_2\times G_2).$
\end{enumerate}
Then the set $U:=\{g\in G_s\colon (g^{p^j},g^{q^j})\in H \text{ for some } j\in \N\}$  generates  $G_s$.
\end{proposition}
\begin{proof}
For $j=1,\ldots,d$,   let
$$
H^j:=H\cap\bigl\{ (g^{p^j}u\,,\,g^{q^j}u')\colon g\in G^j,\ u,u'\in G_2\bigr\}.
$$
Then $H^j$ is a normal subgroup of $H$.
For $k\in \N$ and  $\vec j=(j_1,\dots,j_k)\in\{1,\dots,d\}^k$, we write
$$
G_{\vec j}:= [\cdots[[G^{j_1},G^{j_2}],G^{j_3}]\cdots]\ \text{ and }\
H_{\vec j}:= [\cdots[[H^{j_1},H^{j_2}],H^{j_3}]\cdots]
$$
for the iterated commutator groups.

Since  the subgroups $G^j$ of $G$ are normal, every group  $G_{\vec j}$ is  normal. Moreover,
since $G=G^1\cdots G^d$, for every $k\in \{1,\ldots, d\}$ the group $G_k$ is the product of the groups $G_{\vec j}$ for
$\vec j\in\{1,\dots,d\}^k$.

 \begin{claim}\label{Claim1} Let $\vec j=(j_1,\dots,j_k)$ with coordinates in $ \{1,\ldots, d\}$  and $j:=j_1+\cdots +j_k$.
 For every $g\in G_{\vec j}$ there exist $u,u'\in G_{k+1}$ such that
$(g^{p^j}u,g^{q^j}u')\in H_{\vec j}$.
\end{claim}
 We prove the claim by induction on $k$. If $k=1$, then $\vec j=(j_1)$ so  $G_{\vec j}=G^{j_1}$,  $H_{\vec j}=H^{j_1}$, and the announced property follows immediately from the definition of the group $H^{j_1}$  and the hypothesis. Let $k>1$ and suppose that the result holds for $k-1$, we are going to show that it holds for $k$.  We let
$$
A:=\bigl\{ g\in G_{\vec j}\colon\exists u,u'\in G_{k+1},\ (g^{p^j}u,g^{q^j}u')\in H_{\vec j}\bigr\}
$$
and we have to prove that $A= G_{\vec j}$. We claim first that $A$ is a subgroup of $G_{\vec j}$.
Indeed, let $g,h\in A$, then $g,h\in  G_{\vec j}$ and
there exist $u,u',v,v'\in G_{k+1}$  such that $(g^{p^j}u,g^{q^j}u')$ and  $(h^{p^j}v,h^{q^j}v')$
belong to $H_{\vec j}$. Then $(g^{p^j}uh^{p^j}v, g^{q^j}u'h^{q^j}v')
\in H_{\vec j}$ and furthermore
$$
g^{p^j}uh^{p^j}v=(gh)^{p^j}\bmod G_{k+1}, \quad
g^{q^j}u'h^{q^j}v'=(gh)^{q^j}\bmod G_{k+1}.
$$
 Hence, $gh\in A$. Furthermore,
$(u^{-1}g^{-p^j}, u'^{-1}g^{-q^j})= (g^{-p^j}u_1, g^{-q^j}u_2)\in H_{\vec j}$
for some $u_1,u_2\in G_{k+1}$. Hence, $g^{-1}\in A$. It follows that $A$ is a group.

We let $\vec i:=(j_1,\dots,j_{k-1})$ and  $i:=j_1+\cdots +j_{k-1}$. Then $G_{\vec j}$ is the group spanned by elements  $[h,z]$ with $h\in G_{\vec i}$ and $z\in G^{j_k}$, and as $A$ is a group, it suffices to prove that each element
of this form belongs to $A$.
By the induction hypothesis, there exist $u,u'\in G_k$ with
$(h^{p^i}u, h^{q^i}u')\in H_{\vec i}$ and by the first step there exist $v,v'\in G_2$ with $(z^{p^{j_k}}v,z^{q^{j_k}}v')\in H^{j_k}$. The commutator
$\bigl([h^{p^i}u, z^{p^{j_k}}v]\,,\, [h^{q^i}u',z^{q^{j_k}}v']\bigr)$ of these two elements belongs to $H_{\vec j}$. Furthermore,
$$
[h^{p^i}u, z^{p^{j_k}}v]=[h,z]^{p^j}\bmod G_{k+1}, \quad [h^{q^i}u',z^{q^{j_k}}v']=[h,z]^{q^j}\bmod G_{k+1}.
$$
Hence, $[h,z]\in H_{\vec j}$. This completes the proof of Claim~\ref{Claim1}.

Taking $k=s$ and using that $G_{s+1}$ is trivial  we get:
\begin{claim}
\label{cl:two}
Let $\vec j:=(j_1,\ldots, j_s)$ with coordinates in $\{1,\ldots, d\}$ and $j:=j_1+\dots+j_s$. Then  for every $g\in G_{\vec j}$ we have $(g^{p^j},g^{q^j})\in  H_s$.
\end{claim}

We are now ready to show that the set $U$ generates $G_s$. As already noticed, $G_s$ is the product
of the groups $G_{\vec j}$  for $\vec j\in\{1,\dots,d\}^s$. Hence, it suffices to show that the set $U$ contains all these groups.
So let $\vec j:=(j_1,\ldots, j_s)$ with $j_i\in \{1,\ldots, d\}$, $i=1,\ldots, s$,  and suppose that $g\in G_{\vec j}$.
By Claim~\ref{cl:two} we have  $(g^{p^j},g^{q^j})\in  H_s\subset H$ for   $j=j_1+\dots+j_s$, which proves that $g\in U$.
This completes the proof of   Proposition~\ref{P:alg}.\end{proof}

 \section{Minor arc nilsequences-Proof of the discorrelation estimate}
 \label{S:minorarcs2}
Our goal in this section is to prove Theorem~\ref{th:discorrelation}.
%
Suppose that the group $G$  is $s$-step nilpotent. The proof goes by induction on $s$.
 We assume that either $s=1$, or that $s\geq 2$, and  the result holds for $(s-1)$-step nilmanifolds.
We are going to show that it holds for $s$-step nilmanifolds.

\subsection{Reduction to the case of a nilcharacter}
\label{SS:7.3}
We start with   some reductions, similar to those made in the proof of~\cite[Lemma~3.7]{GT12a}.  We let $r:=\dim(G_s)$ and $t:=\dim(X)$.  Suppose that \eqref{Assumption} holds.

There exists a constant $A_1:=A_1(X,\tau)$ and a function $\Phi'$ on $X$ with
$$
\norm{\Phi-\Phi'}_\infty\leq\frac\tau 2, \ \
\int_X\Phi'\,dm_X=0, \ \text{ and }\ \norm{\Phi'}_{\CC^{2t}(X)}\leq A_1.
$$
 Then $|\E_{n\in[N]} \one_P(n) f (n)\Phi'(g(n+k)\cdot e_X)|>\tau/2$. Therefore, substituting $\Phi'$ for $\Phi$,  up to a change in the constants, we can assume that
$$
\norm\Phi_{\CC^{2 t}(X)}\leq 1.
$$
We proceed now to a ``vertical Fourier decomposition''.
Using the Mal'cev basis of $G$, we   identify the vertical torus $G_s/(G_s\cap\Gamma)$ with $\T^r$ and its dual group with $\Z^r$. For  $\bh\in\Z^r$ let
$$
\Phi_\bh(x):=\int_{\T^r}\e(-\bh\cdot\bu)\, \Phi(\bu\cdot x)\, dm_{\T^r}(\bu).
$$
We have
$$
\norm{\Phi_{\bh}}_{\lip(X)}\leq 1,\qquad  \int_X\Phi_\bh\,dm_X=0,
$$
and $\Phi_\bh$ is a nilcharacter of frequency $\bh$, that is,
$$\Phi_\bh(\bv\cdot x)=\e(\bh\cdot\bv)\, \Phi_\bh(x)\ \ \text{ for every }\ \bv\in\T^r\text{ and every }\ x\in X.
$$
Moreover, since $\norm\Phi_{\CC^{2t}(X)}\leq 1$, we have
$$
\norm{\Phi_\bh}_\infty\leq A_2(1+\norm\bh)^{-2t}
$$
for some constant $A_2:=A_2(X)$ and
$$
\Phi(x)=\sum_{\bh\in\Z^r}\Phi_\bh(x)\ \text{ for every }\ x\in X.
$$
It follows from~\eqref{Assumption} that there exists $\bh\in\Z^r$, with $\norm\bh\leq A_3$ and
\begin{equation}\label{E:tau1}
|\E_{n\in[N]} \one_P(n)
 \, f(n)\, \Phi_\bh(g(n+k)\cdot e_X)|>\tau_1
\end{equation}
for some positive reals $\tau_1:=\tau_1(X,\tau)$ and  $A_3:=A_3(X,\tau)$. Therefore,
 we can assume that~\eqref{Assumption} holds with $\tau_1$ in place of $\tau$ for some nilcharacter
  $\Phi$  of frequency $\bh$  and $\norm\bh\leq A_3$.

\subsection{Reduction to the case of a non-zero frequency}
\label{SS:7.4}
First, suppose that $s=1$ and  $\bh=0$. We have $G=G_s$ and $\Phi_\bzero$ is constant. Since the integral of $\Phi$ is equal to zero, $\Phi$ is identically zero and we have a contradiction by \eqref{Assumption}.

 Suppose now that $s\geq 2$ and that $\bh=0$. As in Section~\ref{subsec:vertical},
 we let $ \wt G:=G/G_s$ and $\wt\Gamma:=\Gamma/(\Gamma\cap G_s)$. Then the $(s-1)$-step nilmanifold  $\wt X:= \wt G/\wt\Gamma$ is identified with the quotient of $X$ under the action of the vertical torus $\T^r$. Let $\pi\colon X\to\wt X$ be  the natural projection.
  Since $\Phi$ is a nilcharacter with frequency $0$, it  can be written as $\Phi:=\wt\Phi\circ\pi$ for some function $\wt\Phi$ on $\wt X$ and we have $\int_{\wt X}\wt\Phi\,dm_{\wt X}=0$ and $\norm{\wt \Phi}_{\lip(\wt X)}\leq A_4$ for some constant $A_4:=A_4(X)$. Let $\wt g$ be the image of the polynomial sequence $g$ in $\wt G$ under the natural projection. Then $\wt g\in \poly(\wt G_\bullet)$  where  $\wt G^{(j)}:=(G^{(j)}G_s)/G_s$ for every $j$. We have
$$
|\E_{n\in [N]}\one_P(n) f (n)\wt\Phi(\wt g(n+k)\cdot e_{\wt X}\bigr)|=|\E_{n\in[N]} \one_P(n) f (n)\Phi(g(n+k)\cdot e_X)|>\tau_1.
 $$
 Assuming that $N$ is sufficiently large, depending on $X$ and on $\tau$, by the induction hypothesis we get that the sequence $(\wt{g}(n)\cdot e_{\wt{X}})_{n\in [N]}$ is not totally $\sigma'$-equidistributed for some $\sigma':=\sigma'(X,\tau)$. This implies a similar property for the sequence $(g(n)\cdot e_X)_{n\in [N]}$ and completes the induction in the case where $\bh=0$.

We can therefore assume that the frequency $\bh$ of the nilcharacter $\Phi$ is non-zero.

\subsection{Reduction to the case where  $k=0$ and $g(0)=1_G$}
\label{SS:7.5}

Suppose that the conclusion~\eqref{Conclusion}
 holds  for some $N_0$ and $\sigma$,
  under the stronger assumption that the hypothesis~\eqref{Assumption}
 holds for $k=0$ and for sequences that satisfy $g(0)=1_G$. We are going to show that it holds without these assumptions.
Let $\tau>0$ and $N\geq N_0$. Let $F\subset G$ be a bounded fundamental domain of the projection $G\to X$ (we assume that $F$ is fixed given $X$).  By the first statement of Lemma~\ref{lemp:ap1} there exists a constant $C_1>0$ such that
\begin{equation}
\label{E:C2}
C_1\inv d_X(x,x')\leq
d_X(g\cdot x,g\cdot x')\leq C_1 d_X(x,x')\  \text{ for every }g\in F \text{ and  }x,x'\in X.
\end{equation}

Let the sequence $g\in \poly(G_\bullet)$ be given as above and   $k\in \N$.  We write
$$
g(k)=a_k\gamma_k\ \text{ where }\ a_k \in F\ \text{ and }\ \gamma_k\in\Gamma.
$$
Let $\wt g\colon [N]\to G$ be defined by
$$
\wt g(n):=a_k\inv g(n+k) \gamma_k\inv.
$$
  Then $\wt g(0)=1_G$. By the first definition in Section~\ref{SS:PolyGroup}   and since the subgroups $G^{(i)}$ of $G$ are normal, we have   $\wt g\in \poly(G_\bullet)$
 and
for every $n\in\N$ we have
$$
g(n+k)\cdot e_X= a_k \, \wt g(n)\cdot e_X.
$$
We let
$$
\Phi_k(x):=\Phi(a_k\cdot x).
$$
Since $\Phi$ is a nilcharacter with non-zero frequency $\bh\in \Z^r$, for every $k\in \N$, $\Phi_k$ is also a nilcharacter with
the same frequency.  Since $a_k$ belongs to  $F$ for every $k\in \N$ and $\norm{\Phi}_{\lip(X)}\leq 1$,  we get by~\eqref{E:C2} that  $\norm{\Phi_k}_{\lip(X)}\leq C_1$.
We let $\wt\Phi_k:=\Phi_k/C_1$. Then  $\norm{\wt\Phi_k}_{\lip(X)}\leq 1$, $\int \wt\Phi_k \ d m_X=0$, and
estimate \eqref{E:tau1} implies that
$$
|\E_{n\in[N]} \one_P(n)\,
 f (n)\, \wt\Phi_k(\wt g(n)\cdot e_X)|\geq\tau_2
$$
for some $\tau_2:=\tau_2(X,\tau)>0$.
We are now in a situation where the additional hypothesis are satisfied, that is, $k=0$ and $\wt{g}(0)=1_G$. We deduce that
the sequence $(\wt g(n)\cdot e_X)_{n\in[N]}$ is not totally $\sigma_1$-equidistributed in $X$ for some $\sigma_1>0$.
Let $\eta$ be the horizontal character provided by Theorem~\ref{th:Leibman}.
We have that $\eta(\wt g(n))=\eta(g(k))\inv \eta(g(n+k))$. Applying  Lemma~\ref{lem:8.4} with $\phi(n):=\eta(g(n+k))$ and $\psi(n):=\eta(g(n))$ and then applying  Lemma~\ref{lem:Leibman_Inverse},
 we deduce
that there exist a positive integer
$N_0'$  and a positive real $\sigma_2$,  such that if $N\geq N'_0$,
then the  sequence $(g(n)\cdot e_X)_{n\in[N]}$ is not totally $\sigma_2$-equidistributed in $X$.
Hence, in establishing  Theorem~\ref{th:discorrelation}, we can assume that $k=0$ and $g(0)=1_G$.

Therefore, in the rest of this proof, we can and will assume that $g(0)=1_G$,  and that
\begin{equation}
\label{E:lower2}
|\E_{n\in[N]} \one_P(n) \,  f(n)\, \Phi(g(n)\cdot e_X)|\geq\tau_2
\end{equation}
 for some $\tau_2:=\tau_2(X,\tau)>0$, where
 $$
 \Phi \ \text{  is a nil-character
 with non-zero frequency.}
 $$

\subsection{Using  the orthogonality criterion of K\'atai}
\label{SS:7.6}
Combining the lower bound \eqref{E:lower2} and  Lemma~\ref{lem:katai} we get that there exists a positive integer  $K:=K(X,\tau_2)=K(X,\tau)$,  primes $p,q$ with $p<q<K$, and a positive real $\tau_3:=\tau_3(X,\tau_2)=\tau_3(X,\tau)$, such that
$$
|\E_{n\in[N]} \one_{[N/q]}(n) \one_{P}(pn) \, \one_{P}(qn)\,
\Phi(g(pn)\cdot e_X)\cdot \overline{\Phi}(g(qn)\cdot e_X)|\geq\tau_3.
$$
Let $P_1\subset[N]$ be an arithmetic progression such that $\one_{[N/q]}(n)\one_{P}(pn)\one_P(qn)=\one_{P_1}(n)$. Then the last inequality can be rewritten as
\begin{equation}
\label{eq:PhiKatai}
|\E_{n\in[N]} \one_{P_1}(n)\,
\Phi(g(pn)\cdot e_X)\cdot \overline\Phi(g(qn)\cdot e_X)|\geq\tau_3.
\end{equation}
We remark that the pairs $(p,q)$ with $p,q<K$ belong to some finite set that depends only on $X$ and $\tau$. Therefore, from this point on  we can and will assume that $p$ and $q$ are fixed distinct primes.
Almost all parameters defined below  will depend on $p$ and $q$, and in order to ease our notation a bit, this dependence is going to   be left implicit.

\subsection{Using the factorization theorem on $X\times X$}
\label{SS:7.8}
Recall that $g\in \poly(G_\bullet)$. Let $G\times G$ be endowed with the product filtration and the product Mal'cev basis.
 We define the sequence $h\colon [N]\to G\times G$  by
\begin{equation}\label{def:h}
h(n):=(g(pn),g(qn)), \quad n\in [N].
\end{equation}
 For every $i\in\N$ and  $k_1,\dots,k_i\in\Z$ we have (recall that $\partial_{h}g(n):=g(n+h)g(n)^{-1}$)
$$
\partial_{k_i}\dots\partial_{k_1}h(n)=
\bigl(\bigl(\partial_{pk_i}\dots\partial_{pk_1} g)(pn),
 \bigl(\partial_{qk_i}\dots\partial_{pk_1} g\bigr)(qn)\bigr)\in G^{(j)}\times G^{(j)}=(G\times G)^{(j)},
$$
and thus  $h\in \poly((G\times G)_\bullet)$.
We can rewrite~\eqref{eq:PhiKatai}
as
\begin{equation}
\label{eq:PhiKatai2}
|\E_{n\in[N]} \one_{P_1}(n)\,
(\Phi\otimes\overline\Phi)(h(n)\cdot e_{X\times X})|\geq\tau_3.
\end{equation}

 For the sequence $h\in\poly((G\times G)_\bullet)$,  given by \eqref{def:h}, we apply Theorem~\ref{th:FactoGT} for a function $\omega_5\colon \N\to \R^+$ that will be determined shortly (its defining relation is \eqref{eq:choseOmega}) and depends only on $X, \tau$. We get   families $\CF_5(M)$, $M\in \N$,  of sub-nilmanifolds of $X\times X$ (which do not depend on $\omega_5$), that increase with $M$, a constant $M_5:=M_5(X,\omega_5)$, an integer $M^*$ with   $1\leq M^*\leq M_5$, a closed and connected
  rational  subgroup $H$ of $G\times G$, a nilmanifold $Y:= H/(H\cap (\Gamma\times \Gamma))$ belonging to the family $\CF_5(M^*)$, and a factorization
$$
h(n)=\epsilon(n)\, h'(n) \, \gamma(n), \quad n \in [N],
$$
where
\begin{enumerate}
\item
\label{it:smooth3}
$\epsilon\colon [N]\to G\times G$ is $(M^*,N)$-smooth;
\item
\label{it:equid3}
 $h'\in \poly(H_\bullet)$  and $(h'(n)\cdot e_{X\times X})_{n\in[N]}$ is totally $\omega_5(M^*)$-equidistributed in $Y$
 with the metric $d_{Y}$ induced by the filtration $H_\bullet$;
\item
\label{it:periodic3}
 $\gamma\colon [N]\to G\times G$ is $M^*$-rational and $(\gamma(n)\cdot e_{X\times X})_{n\in [N]}$ has period at most $M^*$.
\end{enumerate}
 We start from \eqref{eq:PhiKatai2} and use  the argument of Section~\ref{subsec:constr-fact}  with $X\times X$ and $G\times G$ substituted for $X$ and $G$ respectively. We get  a rational element $\alpha$ belonging to the finite subset $\Sigma(M^*)$ of $G\times G$, an integer $n_0\in[N]$,  and an arithmetic progression $P_2\subset P_1$, such that the function $(\Phi\otimes\overline\Phi)_\alpha$ and the sequence $h_\alpha'$ defined by
\begin{gather}\label{E:FFa}
(\Phi\otimes\overline\Phi)_\alpha(y):=(\Phi\otimes\overline\Phi)(\epsilon(n_0)\alpha\cdot y)\ \text{ for  }y\in Y;\\
\notag h_\alpha'(n):=\alpha\inv h'(n)\alpha\ \text{ for }n\in[N],
\end{gather}
satisfy
\begin{equation}
\label{eq:PhiKatai2'}
|\E_{n\in[N]} \one_{P_2}(n)\,
(\Phi\otimes\overline\Phi)_\alpha(h_\alpha'(n)\cdot e_{X\times X})|\geq\tau_4(M^*),
\end{equation}
where $$
\tau_4(M^*):=\tau_3^2/(64\,H_1(M^*)^2{M^*}^2)
$$ and $H_1(M^*)$ is the constant defined as $H(M^*)$ in Section~\ref{subsec:constr-fact}. We proceed as in Section~\ref{subsec:constr-fact}  with $H$ in place of $G'$ and $Y$ in place of $X'$, we let
$$
H_\alpha:=\alpha\inv H \alpha \  \text{ and } \ Y_\alpha:= H_\alpha\cdot e_Y.
$$
We have that $ h_\alpha'\in\poly(H_{\alpha\bullet})$  and by Property~\eqref{it:def-rho} of Section~\ref{subsubsec:defH},
    if $N$ is large enough, depending on $X,M^*, \omega_5(M^*)$, then for every $\alpha\in \Sigma(M^*)$ and $Y\in \CF_5(M^*)$ we have
\begin{equation}
\label{eq:equidhalpha}
(h_\alpha'(n)\cdot e_Y)_{n\in[N]}\text{ is totally }
\rho_{X\times X}(M^*,\omega_5(M^*))\text{-equidistributed in } Y_\alpha.
\end{equation}
Note that  since $\norm\Phi_{\lip(X)}\leq 1$  we have $\norm{\Phi\otimes\overline\Phi}_{\lip(X\times X)}\leq C_2$
for some positive real $C_2:=C_2(X)$.

We can now define the function $\omega_5$.
Recall that for every fixed $M\in \N$
we have  $\rho_{X\times X}(M,t)\to 0$ as $t\to 0^+$. Therefore,  there exists $\omega_5\colon \N\to \R^+$ such that
\begin{equation}
\label{eq:choseOmega}
\rho_{X\times X}(M,\omega_5(M)) < \tau_4(M)H_1(M)^{-2}C_2\inv\ \text{ for every }  M\in \N.
\end{equation}
Note that there is no
circularity in defining these parameters, as
$\CF_5(M)$, $\tau_4$, $H_1$,  $C_2$ do not depend on $\omega_5$. Furthermore, the function
$\omega_5$  and the integer $M_5$ depend only on $X$   and on $\tau$, hence there exists $N_5:=N_5(X,\tau)$ such that
\begin{equation}\label{EE:N_1}
\text{ if } \ N\geq N_5,\  \text{ then } \
\eqref{eq:equidhalpha} \text{ holds for } M^*\in \N \text{ with } M^*\leq M_5 \text{ defined as above}.
\end{equation}

\subsection{Reduction to a zero mean property}
\label{SS:reduce}
We work with the value of $M^*$ found in the previous subsection and assume that $N\geq N_5$. Suppose for the moment that
\begin{equation}\label{E:Reducezero}
\int_{Y_\alpha} (\Phi\otimes\overline\Phi)_\alpha \ dm_{Y_\alpha}=0.
\end{equation}
Since $M^*\leq M_5$,  by \eqref{EE:N_1}   the sequence
$(h'_\alpha(n)\cdot e_{X\times X})_{n\in[N]}$   is totally
$\rho_{X\times X}(M^*,\omega_5(M^*))$-equidistributed in $Y_\alpha$.
Furthermore, it follows from~\eqref{eq:Phi'lip}
 that   $\norm{(\Phi\otimes\overline\Phi)_\alpha}_{\lip(X\times X)}\leq C_2H_1(M^*)$
and using Property~\eqref{eq:LipXXprime} of Section~\ref{subsubsec:defH} we get
$\norm{(\Phi\otimes\overline\Phi)_\alpha|_{Y_\alpha}
}_{\lip(Y_\alpha)}\leq C_2H_1(M^*)^2$.  It follows that
$$
|\E_{n\in[N]} \one_{P_2}(n)\,
(\Phi\otimes\overline\Phi)_\alpha(h_\alpha'(n)\cdot e_{X\times X})|
\leq \rho_{X\times X}(M^*,\omega_5(M^*))\cdot C_2H_1(M^*)^2<\tau_4(M^*)
$$
by~\eqref{eq:choseOmega}, contradicting~\eqref{eq:PhiKatai2'}.
 Hence, \eqref{E:Reducezero} cannot hold.

 Therefore, in order to complete the proof of Theorem~\ref{th:discorrelation} it remains to show that there exist a positive real $\sigma$ and a positive integer $N_6$, both depending on $X$ and $\tau$ only,   such that if for some $N\geq N_6$  the sequence $(g(n)\cdot e_X)_{n\in [N]}$ is  totally $\sigma$-equidistributed in $X$,
then \eqref{E:Reducezero} holds, where $\Phi, \alpha, Y_\alpha$ are as before.
The values of $\sigma$ and $N_6$  are going to be  determined in Section~\ref{SS:7.9}.

\subsection{Reduction to an algebraic property}
\label{SS:reduceH}
Suppose for the moment  that the group $H$ defined in Section~\ref{SS:7.8} satisfies the following property:
 \begin{equation}
 \label{E:invariance3}
\text{the set }\  U:=\{u\in G_s\colon (u^{p^j},u^{q^j})\in H \text{ for some } j\in \N\} \ \text{ generates } \ G_s
\end{equation}
 where we use multiplicative notation for $G_s$. We claim that then \eqref{E:Reducezero} holds. To see this,  notice first that since
 $\Phi$ is a nilcharacter of $X$ with non-zero frequency,   there exists   a non-trivial multiplicative character $\theta\colon G_s\to \T$ such that
$$
  \Phi(u\cdot y)= \theta(u) \cdot \Phi(y), \ \text{ for every } y\in Y, \ u\in G_s.
$$
Since $\theta$ is non-trivial, there exists $u\in G_s$ such that $\theta(u)$ is irrational.
By~\eqref{E:invariance3}, there exists $u$ with $\theta(u)$  irrational and  such  that $(u^{p^j},u^{q^j})\in H$ for some $j\in\N$.
 Then $\theta(u^{p^j-q^j})\neq 1$ and
  \begin{equation}\label{E:PhiInv}
  (\Phi \otimes \overline{\Phi})_\alpha((u^{p^j},u^{q^j})\cdot y)= \theta(u^{p^j-q^j}) \cdot (\Phi \otimes \overline{\Phi})_\alpha(y),
  \end{equation}
  where $(\Phi \otimes \overline{\Phi})_\alpha$ is defined in \eqref{E:FFa} and we used that $(u^{p^j},u^{q^j}) $
    belongs in the center of $G\times G$ and hence commutes with the element $\alpha$.
 Left multiplication by $(u^{p^j},u^{q^j})$ is a
measure preserving transformation on $Y_\alpha$ and thus, after  integrating the last relation on this set, we obtain~\eqref{E:Reducezero}.

 Thus, at this point we have reduced matters to
showing  that there exist a positive real $\sigma$ and a positive integer $N_6$, both depending on $X$ and $\tau$ only,   such that if for some $N\geq N_6$  the sequence $(g(n)\cdot e_X)_{n\in [N]}$ is  totally $\sigma$-equidistributed in $X$,
 then Property~\eqref{E:invariance3} holds.
  We show this in the  final part of this section using the tools developed in Section~\ref{S:minorarcs1}.


\subsection{Our plan and definition of parameters}
   There are two key ingredients  involved in the proof of Property~\eqref{E:invariance3}.
    The first is  Proposition~\ref{P:PolysTorusQuant''}  that  gives information about the action of the sequence $(h(n)\cdot e_{X\times X})_{n\in [N]}$
on the horizontal torus of $X\times X$. This is the place where we use  our assumption that  the sequence $(g(n)\cdot e_X)_{n\in[N]}$
  is  $\sigma$-equidistributed in $X$ for $\sigma$ suitably small and  $N$  sufficiently large.
   Using Proposition~\ref{P:PolysTorusQuant''}   one can then extract information about the group $H\cdot (G_2\times G_2)$ and our second key ingredient is
  the purely algebraic Proposition~\ref{P:alg} that   utilizes
  this information in order to prove Property~\eqref{E:invariance3}.

\subsubsection{Some notation}
\label{SS:7.9}
To facilitate reading, before proceeding to the main body of the proof of Property~\eqref{E:invariance3}, we
introduce some notation and  we
 organize some data and parameters that are spread out in this and the previous   section.
These parameters  are going to be used in the definition of $\sigma$ and the range of eligible $N$'s used in
the statement of  Theorem~\ref{th:discorrelation}.

Recall that $X:=G/\Gamma$ and that $G$ is endowed with the rational filtration $G_\bullet$.  We denote by  $d$  the degree of $G_\bullet$, and thus all polynomial sequences under consideration have degree at most $d$. The distinct primes $p,q$ were introduced in Section~\ref{SS:7.6} and are bounded by a constant that depends on $X$ and $\tau$ only.

We write $Z:=G/(G_2\Gamma)$ for the horizontal torus of $X$ and $m$ for its dimension.
 We identify $Z=\T^m$, the identification being given by the Mal'cev basis of $G$. We write $\pi_Z\colon G\to Z= \T^m$ for the natural projection and let $\pi_{Z\times Z}:=\pi_Z\times\pi_Z$. Let $A_5:=A_5(X)$ be a Lipschitz constant for the maps $\pi_Z$ and $\pi_{Z\times Z}$.

In Lemma~\ref{L:contained}, for all $L_1,L_2\in\N$ and sub-tori $S_1,S_2$ of $\T^m$, we defined a positive real
$\delta_1:=\delta_1(S_1,S_2,L_1,L_2)$ and  a positive integer $N_1:=N_1(S_1,S_2, L_1,L_2)$.

Proposition~\ref{P:PolysTorusQuant''} defines finite families $\CF_4(M)$, $M\in \N$,  of sub-tori of $\T^m$ (which do not depend on $\omega_4$), as well as
integers $M_4:=M_4(\omega_4)$,  $N_4:=N_4(\omega_4)$, and a positive real $\delta_4:=\delta_4(\omega_4)$.
 Also, if $T_1,\dots,T_d$ are sub-tori of $\T^m$, the sub-torus
$R_{T_1,\dots,T_d}$ of $T^{2m}$ was defined by~\eqref{eq:defRT1Td}.

In Section~\ref{SS:7.8}  we defined positive integers $M_5:=M_5(X,\tau)$, $N_5:=N_5(X,\tau)$, and for $1\leq M\leq M_5$, finite families  $\CF_5(M)$ of rational sub-nilmanifolds  of $X\times X$ that increase with $M$.  Then  the family $\CF_5(M_5)$ is the largest and to each nilmanifold $Y\in\CF_5(M_5)$
  we assigned
a rational subgroup $H$ of $G\times G$.
We let $\CF_5'(M_5)$ denote the family of the corresponding subgroups of $G\times G$.
Note that for every $H\in\CF'_5(M_5)$, $\pi_{Z\times Z}(H)$ is a sub-torus of $\T^{2m}$.


\subsubsection{Defining $\sigma$  and the range of eligible $N$}

We define the function $\omega_4\colon \N\to \R^+$ by
\begin{equation}\label{E:condomega4}
\omega_4(M):=\min_{ T_1,\ldots, T_d\in \CF_4(M); \, H\in \CF_5'(M_5) }
\bigl\{
\delta_1\bigl(R_{T_1,\ldots,T_d},\pi_{Z\times Z}(H),M,A_5M_5\bigr)
\bigr\}, \quad M\in \N,
\end{equation}
and we let
\begin{equation}\label{E:condsigma}
\sigma :=A_5^{-1} \delta_4(\omega_4).
\end{equation}
Finally, we let
\begin{equation}\label{E:N2}
\wt{N}_1:=\max_{
T_1,\ldots, T_d\in \CF_4(M_4);\, 1\leq M\leq M_4;  \, H\in \CF_5'(M_5)}
\bigl\{N_1\bigl(R_{T_1,\ldots,T_d},\pi_{Z\times Z}(H),M,A_5M_5\bigr)\bigr\}
\end{equation}
and
\begin{equation}\label{E:N5}
N_6:=\max\{\wt{N}_1,N_4, N_5\}.
\end{equation}
Note that all the above defined parameters and  the function $\omega_4\colon \N\to \R^+$ depend only on $X$ and $\tau$.
Henceforth, we assume that
 \begin{equation}
 \label{E:hypothesis}
 N\geq N_6\  \text{ and the sequence } \ (g(n)\cdot e_X)_{n\in [N]}
\ \text{  is  totally }\  \sigma\text{-equidistributed in } X.
\end{equation}
Under this assumption we plan to establish  Property~\eqref{E:invariance3}.
Recall that  Property~\eqref{E:invariance3} implies Property~\eqref{E:Reducezero}, and this in turn
 suffices to complete the proof of Theorem~\ref{th:discorrelation}.
 \subsection{Proof of the algebraic property}
Our first goal is to   extract some information about the group $H\cdot (G_2\times G_2)$.
 The idea is to compare two factorizations of the projection   of the sequence $h$ to
 $Z\times Z$. The first is the  one
we get by projecting the factorization of $h$ given in Section~\ref{SS:7.8} to
$Z\times Z$. The second
is the one  we get after  imposing  an equidistribution assumption on  the projection
of the sequence $g$ on $Z$ and using Proposition~\ref{P:PolysTorusQuant''}. The two factorizations
involve total equidistribution properties on the subtori  $\pi_{Z\times Z}(H)$ and $R_{T_1,\ldots, T_d}$.
 Assuming that we have ``sufficient'' total equidistribution in the second  case
we are going to show  using Lemma~\ref{L:contained}
 that $R_{T_1,\ldots, T_d}\subset \pi_{Z\times Z}(H)$. This then easily  implies that the group $H$ satisfies the hypothesis of Proposition~\ref{P:alg} and the conclusion of this proposition then
  enables us to deduce    Property~\eqref{E:invariance3}. We proceed now to the details.

The sequence
$$
\bg(n):=\pi_Z(g(n)), \quad n\in \N,
$$ is a polynomial sequence of degree at most $d$ in $Z=\T^m$ with $\pi_Z(g(0))=0$. As explained in Section~\ref{SS:PolyGroup}, we can write this sequence as
$$
\bg(n)=\pi_Z(g(n))=\balpha_1n+\dots+\balpha_d n^d
$$
for some  $\balpha_1,\balpha_2,\dots,\balpha_d\in\T^m$.
 Recall also that
$$
h(n)=(g(pn),g(qn)), \quad n\in [N].
$$

\subsubsection{First factorization of $\bf{h}$}\label{SS:facto1}
Recall that in Section~\ref{SS:7.8} we defined an integer $M^*\leq M_5$, a nilmanifold $Y=H/(H\cap (\Gamma\times \Gamma))$ belonging to a family $\CF_5(M^*)\subset \CF_5(M_5)$  with $H\in\CF'_5(M_5)$,  and for $N\geq N_5$ a
 factorization
$$
h(n)=\epsilon(n) h'(n) \gamma(n), \quad n\in [N],
$$
that satisfies Properties~\eqref{it:smooth3}-\eqref{it:periodic3}  stated in Section~\ref{SS:7.8}.
Projecting both sides of this identity to $Z\times Z =\T^{2m}$ we get
 and a factorization
\begin{equation}\label{E:facto1}
  \bh(n)=\bepsilon_1(n) + \bh_1(n)+ \bgamma_1(n), \quad n\in [N],
 \end{equation}
where
\begin{gather*}
\bh(n)         :=\pi_{Z\times Z}(h(n)), \\
\bepsilon_1(n) :=\pi_{Z\times Z}(\epsilon(n)),\
 \bh_1(n):=\pi_{Z\times Z}(h'(n)),\
  \bgamma_1(n) :=\pi_{Z\times Z}(\gamma(n)).
\end{gather*}
For $N\geq N_6$ (then $N\geq N_5$ by \eqref{E:N5})  we have
\begin{enumerate}
\item
\label{it:smooth3'}
$\bepsilon_1\colon [N]\to \T^{2m}$ is $(A_5M_5,N)$-smooth;
\item
\label{it:equid3'}
 $\bh_1\colon [N]\to \pi_{Z\times Z}(H)$
 is a polynomial sequence of  degree $d$;
\item
\label{it:periodic3'}
 $\bgamma_1\colon [N]\to \T^{2m}$ is $M_5$-rational and
 has period at most $M_5$,
\end{enumerate}
where we used  that the Lipschitz constant of $\pi_{Z\times Z}\colon G\times G\to
Z\times Z$
is at most $A_5$.
Here we do not use the equidistribution properties of the sequences $h'$ and $\bh_1$;
 what is important is that $h'$ takes values in $H$ and thus $\bh_1$  takes values in the sub-torus
$\pi_{Z\times Z}(H)$ of $\T^{2m}$.

\subsubsection{ Second factorization of $\bh$.}
\label{SS:facto2}
Recall that $\bg(n)=\pi_Z( g(n))
=\balpha_1n+\cdots+\balpha_dn^d$ and $\bh(n)=\pi_{Z\times Z}(h(n))$.
We have
$$
\bh(n)=(\bg(pn),\bg(qn)).
$$
Our assumption \eqref{E:hypothesis} and
 the defining property of  $\sigma$, given in  \eqref{E:condsigma},  implies that the sequence
$$
(\bg(n))_{n\in[N]}\  \text{ is  totally }\   \delta_4(\omega_4)\text{-equidistributed in } \T^m,
$$
where we used the fact that the Lipschitz constant of  $\pi_{Z}\colon G\to Z$
is at most $A_5$.

Hence, Proposition~\ref{P:PolysTorusQuant''} applies. Recall that  $M_4:=M_4(\omega_4)$,  $N_4:=N_4(\omega_4)$ were defined by this proposition and that $N_6\geq N_4$ by \eqref{E:N5}.
Therefore, by  Proposition~\ref{P:PolysTorusQuant''},
for every $N\geq N_6$, there exists
a positive integer $M^{**}\leq M_4$, and subtori $T_1,\ldots, T_d\in \CF_4(M^{**})$ of $\T^m$
such that
\begin{equation}
\label{E:STR}
T_1+\cdots +T_d=\T^m,
\end{equation}
and   the sequence   $(\bh(n))_{n\in [N]}$ can be factorized as follows
  \begin{equation}\label{E:facto2}
 \bh(n)=\bepsilon_2(n)+\bh_2(n)+\bgamma_2(n), \quad n\in [N],
  \end{equation}
  where $\bepsilon_2(n)$, $\bh_2(n)$, $\bgamma_2(n)$ are polynomial sequences on $\T^{2m}$ such that
\begin{enumerate}
\item
\label{it:smooth4}
 $\bepsilon_2(n)$ is $(M^{**},N)$-smooth;
\item
\label{it:equid4}
 $\bh_2(n)$ takes values and is  totally $ \omega_4(M^{**})$-equidistributed
  in the  sub-torus
\begin{equation}
\label{E:RT}
 R_{T_1,\ldots,T_d}:=
 \bigl\{(p\bx_1+\cdots+ p^d \bx_d\,,\, q\bx_1+\cdots+ q^d \bx_d)\colon \bx_j\in T_j \text{ for } j=1,\ldots,d
 \bigr\}
 \end{equation}
of $\T^{2m}$;
\item
\label{it:periodic4}
$\bgamma_2(n)$ is $M^{**}$-rational and has period at most $M^{**}$.
\end{enumerate}

\subsubsection{Using the two factorizations}
For $N\geq N_6$, in   Sections~\ref{SS:facto1} and \ref{SS:facto2} we have defined
 the factorizations $\bh(n)=\bepsilon_1(n) + \bh_1(n)+ \bgamma_1(n)$ and
$\bh(n)=\bepsilon_2(n) + \bh_2(n)+ \bgamma_2(n)$ and the integer $M^{**}\leq M_4$.
By the defining property of $\omega_4$, given in~\eqref{E:condomega4},  and since
$T_1,\dots,T_d\in\CF_4(M^{**})$ and  $H\in\CF'_5(M_5)$,
for $N\geq N_6$  we have
$$
\omega_4(M^{**})\leq \delta_1(R_{T_1,\ldots,T_d},\pi_{Z\times Z}(H),M^{**},A_5M_5).
$$
 We have $N\geq N_6\geq\wt N_1$ by~\eqref{E:N5}. Furthermore, since $M^{**}\leq M_4$ we have $\CF_4(M^{**})\subset\CF_4(M_4)$, thus
$T_1,\dots,T_d\in\CF_4(M^{**})$  and by~\eqref{E:N2} we obtain that $N\geq \wt N_1$ which is greater than  $N_1(R_{T_1,\ldots,T_d},\pi_{Z\times Z}(H),M^{**},A_5M_5)$.
Hence,   Lemma~\ref{L:contained} applies, and gives that
 \begin{equation}
 \label{H'=Zpq}
R_{T_1,\ldots,T_d}\subset \pi_{Z\times Z}(H).
\end{equation}

\subsubsection{End of the proof}
For $j=1\ldots, d$, let
$$
G^j:=\pi_{Z}^{-1}(T_j).
$$
Note that for $j=1,\ldots, d$ the group  $G^j$ contains $G_2$ and thus $G^j$ is a normal subgroup of $G$.
 It follows immediately from~\eqref{E:STR} that
$$
 G^1 \cdots  G^d =G.
$$
Let
$$
W:=
\bigl\{(g_1^p\cdots g_d^{p^d}\,,\,g_1^{q}\cdots g_d^{q^d})\colon g_1\in G^1,\ldots, g_d\in G^d
\bigr\}.
$$
Since for $j=1,\dots,d$ we have $\pi_Z(G^j)=T_j$, it follows that $\pi_{Z\times Z}(W)$ is included in the torus
$R_{T_1,\dots,T_d}$ given by~\eqref{E:RT}. Hence, \eqref{H'=Zpq} gives that
$\pi_{Z\times Z}(W)\subset\pi_{Z\times Z}(H)$ which implies that
$$
 W\subset H \cdot (G_2\times G_2).
$$

We have just established that if \eqref{E:hypothesis} holds, then   the group  $H$ satisfies the hypothesis of  Proposition~\ref{P:alg}.
We deduce that $H$ satisfies Property \eqref{E:invariance3}  and as
   explained in Section~\ref{SS:reduceH}, this completes the proof of Theorem~\ref{th:discorrelation}.

\section{The $U^s$-structure theorems}
\label{sec:decompUs}
In this section,  our main goal
is to  prove  Theorems~\ref{T:DecompositionSimple} and \ref{T:DecompositionII}.
The proof of the second result is based on the following more informative variant of the first
 result:
\begin{theorem}[Structure theorem for multiplicative functions I$'$]
\label{T:DecompositionI}
 Let $s\geq 2$ and  $\ve>0$. There exists $\theta_0:=\theta_0(s,\ve)$ such that for  $0<\theta<\theta_0$  there exist
positive integers  $N_0$, $Q$,
$R$, depending on $s, \ve, \theta$ only,  such that the following holds: For
 every  $N\geq N_0$ and    every $ f \in\CM$, the function $ f_ N$ admits
the decomposition
$$
  f_ N(n)= f_ {N,\st}(n)+ f_ {N,\un}(n) \quad \text{ for every }\  n\in
 \tZN,
$$
 where  the functions $ f_ {N,\st}$ and $ f_ {N,\un}$  satisfy:
\begin{enumerate}
\item
\label{it:weakUs-1}  $ f_ {N,\st}= f_ N*\phi_{N,\theta}$, where
$\phi_{N,\theta}$ is the  kernel on $\Z_\tN$ defined  by
\eqref{eq:def-phi}, is independent of $ f $, and
  the convolution product is  defined in $\tZN$;
\item\label{it:weakUs-2}
 If $\xi \in \Z_\tN$ satisfies $\widehat{ f }_{N,\st}(\xi)\neq 0$, then $\displaystyle \big|\frac{\xi}{\tN}-\frac{p}{Q}\big|\leq \frac{R}{\tN}$ for some   $p\in \{0,\ldots Q-1\}$;
 \item
 \label{it:weakUs-4}
$\displaystyle
| f_ {N,\st}(n+Q)- f_ {N,\st}(n)|\leq \frac{R}{\tN}$  for every
$n\in\Z_\tN$,
where  $n+Q$ is taken $\bmod\tN$;
\item
\label{it:weakUs-3}
 $\norm{ f_ {N,\un}}_{U^s(\tZN)}\leq\ve$.
\end{enumerate}
\end{theorem}
We stress the fact that the value of  $\theta_0,  Q, R$  do not depend on $ f \in \CM$ and $N\in \N$,
$N_0$ does not depend on $ f  \in \CM$,  and these values  are not the same as the ones given in Theorem~\ref{th:Decomposition-U2}.
Recall that $\wt N$
 is any  prime between $N$ and   $ \ell   N$, where $\ell$ is a positive integer that is fixed throughout this argument.

The proof of Theorem~\ref{T:DecompositionI} is given in Sections~\ref{SS:UseU2}-\ref{subsec:convolution}. The
 proof of Theorem~\ref{T:DecompositionII} is given in
 Section~\ref{subsec:proof_strong}.
 Before proceeding to the details we  sketch the  proof strategy for Theorem~\ref{T:DecompositionI}.

\subsection{Some preliminary remarks and proof strategy}
\label{subsec:prelim}
Our
proof strategy  follows in part the general ideas of an argument of Green
and Tao from \cite{GT08a,GT12b} where $U^s$-uniformity
of  the M\"{o}bius function was established. In our case,
 we
are faced with  some  important additional difficulties.
The first is the need to establish discorrelation estimates for general
multiplicative functions not just the M\"{o}bius and it is important
 for  applications to establish estimates with implied constants independent of
the elements of $\CM$.
 Another difficulty is
that we cannot simply hope to prove that  all multiplicative functions with zero mean
are  $U^s$-uniform, not even for $s=2$ (see the examples in
Section~\ref{subsec:decomposition}).
 To
compensate for the lack of $U^2$-uniformity of a normalized multiplicative function $ f $,  we  subtract from
 it   a suitable ``structured component''  $ f_ \st$ given by
 Theorem~\ref{th:Decomposition-U2}, so
that $ f_ \un:= f - f_ \st$ has extremely small  $U^2$-norm. Our goal
is then to show that  $ f_ \un$ has small $U^s$-norm
(see Proposition~\ref{th:decomp-Us}).
 In view of the $U^s$-inverse
theorem  (see Theorem~\ref{th:inverse}), this would
follow if we show that $ f_ \un$  has very small correlation with
all $(s-1)$-step nilsequences of bounded complexity. This then  becomes
our main goal (see Proposition~\ref{P:NilCorr}).

The factorization theorem for polynomial sequences
(Theorem~\ref{th:FactoGT}) practically allows us    to treat
correlation with major arc and minor arc  nilsequences
separately.
 Orthogonality to  major arc (approximately periodic) nilsequences  can be deduced from  the  $U^2$-uniformity of $ f_ \un$. To handle  the much more difficult case of minor arc (totally equidistributed) nilsequences,  Theorem~\ref{th:discorrelation} comes to the  rescue as it shows
  that such sequences are asymptotically orthogonal to  all multiplicative functions.
  The function $ f_ \un$ is not multiplicative though, but this can be  taken care by the fact that $ f_ \un= f - f_ \st$ and
  the fact that $ f_ \st$  can be recovered from $ f $ by taking a convolution product with a kernel.
  Using these properties it is possible to transfer estimates from $ f $ to $ f_ \un$.
   Combining the above, we get
  the needed orthogonality of $ f_ \un$ to all $(s-1)$-step nilsequences of bounded
  complexity. Furthermore, a close inspection of the argument shows that all implied constants are independent of $ f $.
This suffices to complete the proof of  Theorem~\ref{T:DecompositionI}.

 Although the previous sketch communicates the basic ideas behind the proof of Theorem~\ref{T:DecompositionI},
the various results needed to implement this plan come with a significant number
 of parameters that one has to juggle with,
  making the bookkeeping  rather cumbersome.
We use  Section~\ref{subsec:beginning} to organize some of these data.

We start the proof with two successive reductions. The first one uses
the $s=2$ case of Theorem~\ref{T:DecompositionI} established in Theorem~\ref{th:Decomposition-U2}. The second  uses the  inverse theorem for the $U^s$-norms (see Theorem~\ref{th:inverse}).

\subsection{Using the $U^2$-structure theorem} \label{SS:UseU2}
 An immediate consequence of Theorem~\ref{th:Decomposition-U2} is that in order to prove
Theorem~\ref{T:DecompositionI} it suffices to prove the following result:
\begin{proposition}
\label{th:decomp-Us}
 Let $s\in \N$ and  $\ve>0$. There exists $\theta_0>0$ such that for  every
 $\theta$ with  $0<\theta\leq\theta_0$, and every sufficiently large $N$,
 the decomposition $ f_ N= f_ {N,\st}+ f_ {N,\un}$ associated to $\theta$ by Theorem~\ref{th:Decomposition-U2} satisfies Properties~\eqref{it:SpectrumU2} and  \eqref{it:decomU22} of this theorem, and also
$$
\norm{ f_ {N,\un}}_{U^s(\tZN)}\leq\ve.
$$
\end{proposition}
So our  next goal becomes to prove Proposition~\ref{th:decomp-Us}.
%
%

\subsection{Using the inverse theorem for the $U^s$-norms}
\label{subsec:using-inverse}
An immediate consequence of the $U^s$-inverse theorem stated in  Theorem~\ref{th:inverse} is that in order to prove
Proposition~\ref{th:decomp-Us} it suffices to prove the following result:
\begin{proposition}\label{P:NilCorr}
Let $X:=G/\Gamma$ be a nilmanifold  with the natural filtration and $\delta>0$. There exists $\theta_0>0$   such that
for every $\theta$ with $0<\theta<\theta_0$ and   every
 sufficiently large $N$, the decomposition $ f_ N= f_ {N,\st}+ f_ {N,\un}$ associated to $\theta$ by Theorem~\ref{th:Decomposition-U2} satisfies Properties~\eqref{it:SpectrumU2} and \eqref{it:decomU22} of this theorem, and also
\begin{equation}
\label{eq:disc-square}
\sup_{ f , g, \Phi }\bigl|\E_{n\in[\tN]} f_ {N,\un}(n)\, \Phi(g^n\cdot e_X)\bigr|\leq\delta,
 \end{equation}
 where $ f $ ranges over $\CM$, $g$ over $G$, and $\Phi\colon X\to \C$  over all
 functions with  $\norm \Phi_{\lip(X)}\leq 1 $.
\end{proposition}
We are going to prove Proposition~\ref{P:NilCorr} and thus finish the proof of Theorems~\ref{T:DecompositionSimple} and \ref{T:DecompositionI} in
Sections~\ref{subsec:beginning}-\ref{subsec:convolution}.

\subsection{Setting up the stage}
\label{subsec:beginning}
In this subsection, we  define and organize
some data  that will be used in the proof of
Proposition~\ref{P:NilCorr}. We take some extra care
to do this before the main body of its proof in order to make sure
that there is no circularity in the admittedly complicated
collection of
 choices involved.

As $X:=G/\Gamma$
is going to be a fixed nilmanifold throughout the argument (recall that it was
determined in Section~\ref{subsec:using-inverse} and depends only on
$s$ and $\ve$), in order to
ease notation:
\begin{itemize}
\item[]
\em Henceforth, we leave the dependence on $X$ implicit.
\end{itemize}
Remember also that $G$ is endowed with the natural filtration.

 We first define several objects,
 that depend on  a positive real $\delta$, and  a positive
integer parameter $M$ that we consider for the moment as a free variable.
The explicit choice of $M$ takes place in Section~\ref{subsec:using-factor} and depends on various other choices
  that will be made subsequently;
what is important though is that it is
 bounded by a positive constant that depends only
on $\delta$ (and, following our convention, on $X$).
The families $\CF(M)$, $M\in\N$,  of nilmanifolds appearing in Theorem~\ref{th:FactoGT} play a central role in our constructions, and it is important to remark that they do not depend on the choice of the function $\omega$
in the same theorem, allowing us to postpone the definition of this function.

\subsubsection{Building families of nilmanifolds}
\label{subsec:building}

For  the nilmanifold $X:=G/\Gamma$, with the natural filtration,
 Theorem~\ref{th:FactoGT}  defines  for every $M\in\N$ a finite family $\CF(M)$ of sub-nilmanifolds of $X$. In Section~\ref{subsec:constr-fact} we defined
  a finite subset $\Sigma(M)$ of $G$, and for every nilmanifold $X':=G'/(G'\cap\Gamma)\in\CF(M)$ and every $\alpha\in \Sigma(M)$ we defined the sub-nilmanifold $X'_\alpha:=G'_\alpha/(G'_\alpha\cap\Gamma)$ of $X$ where $G'_\alpha:= \alpha\inv G'\alpha$,
  and as usual, we consider the induced filtration in $G'_\alpha$ (which is not the natural filtration in $G'_\alpha$). We let
$$
\CF'(M):=\{ X'_\alpha\colon X'\in\CF(M),\ \alpha\in\Sigma(M)\}.
$$
Furthermore, let
$$
H(M)\  \text{ and  } \
\rho\colon \N\times\R_+\to\R_+$$
 be
so that Properties~\eqref{it:conj}-\eqref{it:def-rho} of Section~\ref{subsubsec:defH} are satisfied.

\subsubsection{Restating the discorrelation estimates}
The following claim follows immediately from  Theorem~\ref{th:discorrelation} applied to each nilmanifold in the finite family $\CF'(M)$ and is the central
ingredient in the proof of   Proposition~\ref{P:NilCorr}.
\begin{claim}
\label{cl:def-sigma}
Let $M\in\N$ and  $\tau>0$. Then   there exist $\sigma:=\sigma(M,\tau)>0$ and $N_1:=N_1(M,\tau)\in \N$  such that for every  $N\geq N_1$ the following property holds:
  Let
 $X'_\alpha \in \CF'(M)$
and   $h\in \poly(G'_{\alpha\bullet})$ be such that
$$
\bigl|\E_{n\in[N]}\one_P(n)\,  f (n)\, \Phi(h(n+k)\cdot e_X)\bigr|\geq \tau,
$$
for   some $k\in \N$ with $|k|\leq N$, arithmetic progression $P\subset [N]$,    $ f \in\CM$, and  function
  $\Phi\colon X'_\alpha\to \C$ with $\norm\Phi_{\lip(X'_\alpha)}\leq 1$ and $\int_{X'_\alpha} \Phi\ \! dm_{X'_\alpha}=0$.
Then the sequence $(h(n)\cdot e_X)_{n\in[N]}$ is not totally $\sigma$-equidistributed in $X'_\alpha$.
\end{claim}
\begin{remark}
We stress that although the filtration in $G$ is the natural one, the induced filtration in
$G'_\alpha$ is not necessarily the natural one, this is why in proving Theorem~\ref{th:discorrelation}
we treat the more difficult case of  arbitrary filtrations.
\end{remark}

\subsubsection{Parameters related to the factorization theorem: $\omega$ and $M_1$}\label{SS:omegaM1}
We let
\begin{equation}
\label{eq:chose-theta}
\lambda(M):=\frac{\delta^2}{128\,C_1H(M)^2M^2}
\end{equation}
where
$C_1$ is the universal  constant  defined by
Lemma~\ref{lem:Us-intels}.
We define also
\begin{equation}
\label{E:def-sigma} \wt{\sigma}(M)  :=
\sigma\big(M,\frac{C_1\lambda(M)}{4\,H(M)^2}\big)
\end{equation}
where $\sigma$ is the function defined in Claim~\ref{cl:def-sigma} above.

Since  $\rho(M,t)$ defined in~\eqref{it:def-rho} of Section~\ref{subsubsec:defH} decreases  to $0$ as $t\to 0^+$ and $M$ is fixed, there exists a function $\omega\colon\N\to\R_+$ that satisfies
\begin{equation}
\label{eq:def-omega}
\rho(M,\omega(M))\leq \wt\sigma(M)  \ \text{ for every }\ M\in \N.
\end{equation}

For this choice of $\omega$, Theorem~\ref{th:FactoGT} associates to $X$ a positive integer $M_1$.
Note that as $\omega$ depends only on $\delta$, the value of $M_1$ depends only on $\delta$.
We let
\begin{gather*}
N_2:= \max_{1\leq M\leq M_1}\Big\{N_1\Bigl(M,\frac{C_1\lambda(M)}{4\,H(M)^2}\Bigr)\Big\}\ \text{ where } \ N_1(M,\tau)\ \text{ is defined in Claim~\ref{cl:def-sigma}};\\
N_3:=\max_{1\leq M\leq M_1}\Big\{\frac{16\,M^2H(M)^2}\delta\Big\}.
\end{gather*}
We remark that  these numbers too depend only on   $\delta$.
\subsubsection{Defining $\theta_0$ and the eligible range of $N$}\label{SS:theta0}
Let $\delta>0$.   We let
\begin{equation}\label{E:theta1}
 \theta_0:=\min_{1\leq M\leq M_1}\lambda(M),
\end{equation}
where the function $\lambda$ is given by~\eqref{eq:chose-theta} and $M_1$ is defined in Section~\ref{SS:omegaM1}. Note that  $\theta_0$ depends on $\delta$ only.

Let now $\theta$ be such that
\begin{equation}\label{E:theta}
0<\theta\leq\theta_0.
\end{equation}
We define
\begin{equation}\label{Eq:N_4}
N_4(\delta,\theta):=\max\{N_0(\theta), N_2(\delta),N_3(\delta)\},
\end{equation}
where $N_0(\theta)$ is the integer in the statement of Theorem~\ref{th:Decomposition-U2} and $N_2,N_3$ were defined above and depend only on $\delta$.
Henceforth, we assume that
  \begin{equation} \label{E:N4theta}
  \theta \text{ satisfies }   \eqref{E:theta} \text{ and } N\geq N_4 \text{  where } N_4 \text{ satisfies } \eqref{Eq:N_4}.
\end{equation}
\subsection{Our goal restated}
After setting up the stage we are now ready to enter the main body
of the  proof of Proposition~\ref{P:NilCorr}.
We argue by contradiction. For a fixed
 nilmanifold $X:=G/\Gamma$ and $\delta>0$ we
let  $\theta_0$, $\theta$,  $N_4$ satisfy \eqref{E:theta1}, \eqref{E:theta},  \eqref{Eq:N_4}, respectively. Suppose that
 \begin{equation}\label{E:NilCorr'}
\bigl|\E_{n\in[\tN]} f_ {N,\un}(n)\, \Phi(g^n\cdot e_X)\bigr|>\delta,
 \end{equation}
for some integer $N\geq N_4$, $ f \in \CM$, $g\in G$, and function $\Phi$ with  $\norm \Phi_{\lip(X)}\leq 1$. We are going to derive a contradiction.

\subsection{Using the factorization theorem}
\label{subsec:using-factor}
Recall that $G$ is endowed with its natural filtration.
The sequence $(g^n)_{n\in[\wt N]}$ is a polynomial sequence in $G$ and thus there exists an
integer $M^*$ with $$
1\leq M^*\leq M_1,
$$
where $M_1$ is defined in Section~\ref{SS:omegaM1}, such that the sequence admits a factorization
$$
g^n=\epsilon(n)g'(n)\gamma(n), \quad n\in [N],
$$
as in Theorem~\ref{th:FactoGT};  the sequence $\epsilon$ is $(M^*,\wt N)$ smooth, the polynomial sequence $g'$ takes values in $G'$,  $(g'(n)\cdot e_X)_{n\in[\wt N]}$ is totally $\omega(M^*)$-equidistributed in $X'\in \CF(M^*)$, and the sequence $\gamma$ is $M^*$-rational and $(\gamma(n)\cdot e_X)_{n\in [N]}$ has   period at most $M^*$.

\subsection{Eliminating the smooth and periodic components}
From this point on, we work with the value of $M^*$ and the factorization given in the previous subsection. We use the objects, notation, and estimates associated to this factorization in Sections~\ref{subsubsec:defH} and  \ref{susubsec:fact} with $\wt N$ substituted for $N$.

By~\eqref{eq:Phi'lip}
we have
$\norm{\Phi'}_{\lip(X)}\leq H(M^*)$
and
Property~\eqref{eq:LipXXprime} of Section~\ref{subsubsec:defH} gives  that
\begin{equation}
\label{eq:LipPhiprime0}
 \norm{\Phi'|_{X'_\alpha}
 }_{\lip(X'_\alpha)}\leq H(M^*)^2.
\end{equation}
Furthermore, since the sequence $(g'(n)\cdot e_X)_{n\in[\wt N]}$ is totally
$\omega(M^*)$-equidistributed in $X'$, it follows from Property~\eqref{it:def-rho} of
Section~~\ref{subsubsec:defH}  that the sequence $(g_\alpha'(n)\cdot e_X)_{n\in[\wt N]}$ is totally $\rho(M^*,\omega(M^*))$-equidistributed in $X'_\alpha$ where $g_\alpha'\in \poly(G'_{\alpha\bullet})$. By~\eqref{eq:def-omega}, we get
\begin{equation}
\label{eq:equid-g''}
\text{the sequence } (g_\alpha'(n)\cdot e_X)_{n\in[\wt N]}\text{ is totally $\wt\sigma(M^*)$-equidistributed in $X'_\alpha$}.
\end{equation}
Since $\wt N\geq N_4\geq   16\,H(M^*)^2{M^*}^2/\delta$, following the argument in Section~\ref{susubsec:fact} we get that
\begin{equation}
\label{eq:avPsi-on-P}
\bigl|\E_{n\in[\wt N]}\one_{P_1}(n)\,  f_ {N,\un}(n)\, \Phi'(g'_\alpha(n)\cdot e_X)\bigr|\geq
\frac{\delta^2}{64\,H(M^*)^2{M^*}^2}=2C_1\lambda(M^*),
\end{equation}
where $P_1\subset [\wt N]$ is an arithmetic progression and $g'_\alpha$, $\Phi'$ are defined in~\eqref{eq:defg''} and~\eqref{eq:defPhi'}. The main advantage now is that the sequence $g'_\alpha$ is  ``totally equidistributed'' on some sub-nilmanifold of $X$.

\subsection{Reducing to the zero integral case}
\label{SS:zero}
Our goal is to show that upon replacing   $\Phi'$ with $\Phi'-z$, where
 $z$ is some constant, we  get a bound similar to
\eqref{eq:avPsi-on-P}. To  this end, we make crucial use of the fact that
the $U^2$-norm of $ f_ {N,\un}$ is suitably small, in fact, this is
the step that determined our choice of the degree of
$U^2$-uniformity $\theta_0$ of $ f_ {N,\un}$ in Section~\ref{SS:theta0}. Recall  Theorem~\ref{th:Decomposition-U2} gives that $\norm{ f_ {N,\un}}_{U^2(\tZN)}\leq\theta\leq\theta_0$.
 We let
$$
z:=\int_{X_\alpha'}\Phi'\,dm_{X_\alpha'}\ \ \text{ and }\ \ \Phi'_0:=\Phi'-z.
$$
Then of course $\int_{X_\alpha'}\Phi_0'\,dm_{X_\alpha'}=0$.

Combining Lemma~\ref{lem:Us-intels} in the Appendix, Theorem~\ref{th:Decomposition-U2},
  the definition~\eqref{E:theta1} of $\theta_0$, and that $1\leq M^*\leq M_1$,
 we get
$$
 \big|\E_{n\in [\wt N]}
\one_{P_1}(n)\,  z \,  f_ {N,\un}(n) \big|\leq
 C_1\norm{ f_ {N,\un}}_{U^2(\tZN)}\leq C_1\theta_0\leq C_1\lambda(M^*).
$$

From this estimate and  \eqref{eq:avPsi-on-P} we deduce that
\begin{equation}
\label{eq:avPsi-on-P2}
\bigl| \E_{n\in [\wt N]}\one_{P_1}(n)\,  f_ {N,\un}(n)\,
\Phi'_0(g_\alpha'(n)\cdot e_X)\bigr|\geq C_1\lambda(M^*).
\end{equation}
Moreover, the bound~\eqref{eq:LipPhiprime0} remains valid with $\Phi'_0$ substituted for $\Phi'$.


\subsection{End of proof of Theorem~\ref{T:DecompositionI}.}
\label{subsec:convolution}
 We are now very close to completing
the proof of Proposition~\ref{P:NilCorr} and hence of Theorem~\ref{T:DecompositionI}. To this
end, we are going to   combine the correlation estimate~\eqref{eq:avPsi-on-P2}, the equidistribution result~\eqref{eq:equid-g''}, and Claim~\ref{cl:def-sigma}
to deduce a contradiction.

 Recall that  $ f_ {N,\st}= f_ N*\phi$ (the convolution is taken in
$\Z_{\tN}$) where $\phi$ is a kernel in $\tZN$, meaning a
non-negative function with $\E_{n\in\tZN}\phi(n)=1$. Since
$ f_ {N,\un}= f_ N- f_ {N,\st}$, we can write $ f_ {N,\un}= f_ N*\psi$, where the function $\psi$ on $\tZN$ is given by
$$
\psi(n):=\begin{cases} -\phi(n) &\text{ if }n \neq 0\ \bmod \wt N;\\
 \wt N-\phi(0) & \text{ if }n=0\bmod\wt N,
 \end{cases}
 $$
and satisfies $\E_{n\in\tZN}|\psi(n)|\leq 2$.

 We deduce from~\eqref{eq:avPsi-on-P2}   that there
exists an integer $q$ with $0\leq q < \wt N$ such that
\begin{equation}
\label{eq:bound5}
 \bigl| \E_{n\in [\wt N]} \one_{P_1}(n+q\bmod \wt N) \, f_ N(n)\,
\Phi'_0(g_\alpha'(n+q\bmod \wt N)\cdot e_{X})\bigr|\geq \frac {C_1\lambda(M^*)}2,
\end{equation}
where the residue class $n+q \bmod\wt N$ is taken in $[\wt N]$
instead of the more commonly used interval $[0,\wt N)$. It follows that
$$
 \bigl| \E_{n\in [\wt N]} \one_{P_1}(n+k)\,\one_J(n)\,
\one_{[N]}(n)\,  f (n) \Phi'_0(g_\alpha'(n+k)\cdot
e_{X})\bigr|
\geq \frac {C_1\lambda(M^*)}4,
$$
where either $J$ is the interval $[\wt N-q]$ and $k:=q$, or  $J$ is
the interval $(\wt N-q,\wt N]$ and $k:=q-\wt N$.   In either case we have $|k|\leq \wt N$ and  $\one_{P}(n+k)\,\one_J(n)\,\one_{[N]}(n)=\one_{P_2}(n)$
for some arithmetic progression $P_2\subset[N]$. Thus, for some $k\in \N$ with  $|k|\leq \wt N$ we have
\begin{equation}
\label{eq:bound5'}
 \bigl| \E_{n\in [\wt N]} \one_{P_2}(n)\,  f (n)\,
\Phi'_0(g_\alpha'(n+k)\cdot e_{X})\bigr|\geq \frac {C_1\lambda(M^*)}4.
\end{equation}
Recall that  $\int_{X_\alpha'}\Phi_0'\,dm_{X_\alpha'}=0$,  $\norm{\Phi'_0|_{X'_\alpha}
}_{\lip(X'_\alpha)}\leq H(M^*)^2$, and  the nilmanifold
$X'_\alpha$ belongs to the family $\CF'$. By~\eqref{E:N4theta}  and since $g'_\alpha\in \poly(G_{\alpha\bullet}')$
we can use     Claim~\ref{cl:def-sigma}  with
$\tau:=C_1\lambda(M^*)/4\,H(M^*)^2$. We deduce that the sequence  $(g_\alpha'(n)\cdot e_X)_{n\in[\wt N]}$ is not totally $\sigma(M^*,C_1\lambda(M^*)/4\,H(M^*)^2)$-equidistributed in $X'_\alpha$.
But, by definition~\eqref{E:def-sigma} we have  $ \wt{\sigma}(M^*)=$
$\sigma\bigl(M^*,C_1\lambda(M^*)/ 4\,H(M^*)^2\bigr)$ which contradicts~\eqref{eq:equid-g''}.

Hence, our hypothesis \eqref{E:NilCorr'} cannot hold, and as a consequence  Proposition~\ref{P:NilCorr} is verified.
This completes the proof of
 Theorem~\ref{T:DecompositionI} and thus of Theorem~\ref{T:DecompositionSimple}. \qed

\subsection{End of proof  of Theorem~\ref{T:DecompositionII}.}
\label{subsec:proof_strong}

 We are going to  deduce  Theorem~\ref{T:DecompositionII} from Theorem~\ref{T:DecompositionI}
using
an iterative  argument of energy increment. To do this, we will use explicit properties
of the kernels introduced in Section~\ref{subsec:kernels} and used to define the structured part $ f_ {N,\st}$ in Theorem~\ref{T:DecompositionI}. In particular, the following
monotonicity of the Fourier coefficients of the kernels $\phi_{N,\theta}$  is key:
\begin{multline}
\label{eq:phi-increases2}
\text{if }\theta\geq\theta'>0\text{ and }
N\geq\max\{N_0(\theta), N_0(\theta')\},
\\
 \text{then for every }\xi\in\tZN, \ \widehat{\phi_{N,\theta'}}(\xi)\geq
\widehat{\phi_{N,\theta}}(\xi)\geq  0
\end{multline}
where $N_0$ is given by Theorem~\ref{th:Decomposition-U2}.

We fix a function $F\colon \N \times \N \times \R^+\to \R^+$, an
  $\ve>0$, and a probability  measure $\nu$ on the compact space
$\CM$ of multiplicative functions.

We define inductively  a sequence  $(\theta_j)$ of positive reals
and sequences $(N_j)$,  $(Q_j)$, $(R_j)$ of positive integers as
follows.
 We let $\theta_1=N_1=Q_1=R_1=1$. Suppose
 that $j\geq 1$ and that the first $j$ terms of the sequences are defined.
 We apply
Theorem~\ref{T:DecompositionI} with
$$
 \frac
1{F(Q_j,R_{j},\ve)} \ \text{ substituted for } \ \ve.
$$
Theorem~\ref{T:DecompositionI} provides a real  $\theta_0:=\theta_0(j)>0$ and we define
$$
\theta_{j+1}:=\min\{\theta_0,\theta_j\}.
$$
Then Theorem~\ref{T:DecompositionI} with $\theta_{j+1}$ substituted for $\theta$ provides
integers $N_0$, $Q$, $R$, and we let $Q_{j+1}:=Q$, $R_{j+1}:=R$, and  $N_{j+1}:=\max\{N_0,N_j\}$.  For every $N\geq N_{j+1}$ the kernel $\phi_{N,\theta_{j+1}}$ can be defined, and the functions
$$
 f_ {j+1,N,\st}:= f_ N*\phi_{N,\theta_{j+1}}\quad \text{ and } \quad
 f_ {j+1,N,\un}:= f_ N- f_ {j+1,N,\st}
$$
 satisfy  Property~\eqref{it:decomU22} of Theorem~\ref{th:Decomposition-U2} and the conclusion of Theorem~\ref{T:DecompositionI}, that is,
\begin{gather}
\label{eq:decompU32}
 | f_ {j+1,N,\st}(n+Q_{j+1})- f_ {j+1,N,\st}(n)|\leq \frac {R_{j+1}} \tN\quad \text{ for every }\ n\in
\Z_\tN;\\
\label{eq:decompU33} \norm{ f_ {j+1,N,\un}}_{U^s(\tZN)}\leq \frac
1{F(Q_j,R_{j},\ve)}.
\end{gather}
 By construction, the
sequence $(N_j)$  increases and the sequence
$(\theta_j)$  decreases with $j$. Let
$$
J:=1+\lceil 2\ve^{-2}\rceil\ \text{ and } \  N_0^{'}:=N_{J+1}.
$$
 For every
$N\geq N_0^{'}$ we have
\begin{multline*}
\sum_{j=2}^J\int_\CM
\norm{ f_ {j+1,N,\st}- f_ {j,N,\st}}_{L^2(\tZN)}^2\,d\nu( f )=\\
\int_\CM \sum_{\xi\in\tZN} |\widehat{ f_ N}(\xi)|^2 \,\sum_{j=2}^J
|\widehat{\phi_{N,\theta_{j+1}}}(\xi)-\widehat{\phi_{N,\theta_{j}}}(\xi)|^2\,d\nu( f )
\leq\\
 2 \int_\CM \sum_{\xi\in\tZN}|\widehat{ f_ N}(\xi)|^2
\,\sum_{j=2}^J
\bigl(\widehat{\phi_{N,\theta_{j+1}}}(\xi)-\widehat{\phi_{N,\theta_j}}(\xi)\bigr)\,d\nu( f ),
\end{multline*}
where to get the last estimate we used that   $\theta_{j+1} \leq
\theta_{j}$ and thus
$\widehat{\phi_{N,\theta_{j+1}}}(\xi)\geq\widehat{\phi_{N,\theta_j}}(\xi)\geq
0$ for every $\xi\in \Z_{\wt N}$ by~\eqref{eq:phi-increases2}. Since
$|\widehat{\phi_{N,\theta}}(\xi)|\leq 1$, the last quantity in the
estimate is  at most
$$
 2\int_\CM \sum_{\xi\in \tZN} |\widehat{ f_ N}(\xi)|^2 \, d\nu( f )\leq 2.
$$
Therefore,  for every $N\geq N_0^{'}$ there exists $j_0:=j_0(F,N,\ve,
\nu)$ with
\begin{equation}\label{E:j0isbounded}
2\leq j_0\leq J
\end{equation} such that
\begin{equation}
\label{eq:decompU34} \int_\CM
\norm{ f_ {j_0+1,N,\st}- f_ {j_0,N,\st}}_{L^2(\tZN)}^2\,d\nu( f )\leq
\frac 2{J-1}\leq \ve^2.
\end{equation}
 For   $N\geq N_0^{'}$, we  let
\begin{gather*}
\psi_{N,1}:= \phi_{N,\theta_{j_0}};\quad \psi_{N,2}:=\phi_{N,\theta_{j_0+1}};\\
  f_ {N,\st}:= f_ N*\psi_{N,1}= f_ {j_0,N,\st} ;\quad
  f_ {N,\un}:= f_ N- f_ N*\psi_{N,2}=  f_ {j_0+1,N,\un} ;\\
 f_ {N,\er}:= f_ N*(\psi_{N,2}-\psi_{N,1})=  f_ {j_0+1,N,\st}- f_ {j_0,N,\st} ;\\
Q:=Q_{j_0}\ \text{ and }\ R:=R_{j_0}.
\end{gather*}
Then we have the decomposition
$$
 f_ N= f_ {N,\st}+ f_ {N,\un} + f_ {N,\er}.
$$
Furthermore, Property~\eqref{it:decompU32} of Theorem~\ref{T:DecompositionII} follows from~\eqref{eq:decompU32} (applied for $j:=j_0-1$),
 Property~\eqref{it:decomU33} follows from~\eqref{eq:decompU33} (applied for $j:=j_0$),  and Property~\eqref{it:decomU34}
 follows from~\eqref{eq:decompU34} and the Cauchy-Schwarz
 inequality. Lastly, it follows from \eqref{E:j0isbounded} that the integers $N_0^{'},Q,R$  are bounded by a  constant that depends on $F$ and $\ve$ only.
 Thus, all the announced properties are satisfied,
completing the proof of Theorem~\ref{T:DecompositionII}.
\qed

\section{Aperiodic multiplicative functions}
\label{sec:aperiodic}

In this section our goal is to prove the main results regarding aperiodic multiplicative functions, that is,  Theorems~\ref{th:aperiod_uniform} and \ref{th:chowla}.

\subsection{Proof of Proposition~\ref{prop:equiv-aperiodic}}
\label{subsec:proof-aperiodic}
We prove here  the equivalence between the four characterizations of aperiodic multiplicative functions given in Proposition~\ref{prop:equiv-aperiodic}.

The equivalence of the first three properties of Proposition~\ref{prop:equiv-aperiodic}  is easy. The equivalence of~\eqref{it:aperiodic} and~\eqref{it:chi-npq} is simple. Furthermore, \eqref{it:chi-npq} immediately implies that for every periodic function $a$ we have $\E_{n\in[N]}  f (n)a(n)\to 0$ as $N\to+\infty$, and this  in turn implies~\eqref{it:chi-chi}.
 Going  from \eqref{it:chi-chi} to \eqref{it:chi-npq} is also simple and standard: for $p,q\in\N$ we have
$$
\E_{n\in [N]}  f (n)\e(np/q)=\frac{1}{N}\sum_{d|q^r\text{ for some }r} f (d) \sum_{1\leq n\leq \frac{N}{d}, (n,q)=1}  f (n)\e(dnp/q)
$$
and the function $\one_{\{k\colon (k,q)=1\}}(n) \e(dnp/q)$ has period $q$. The implication follows from the fact that  any periodic function with period $q$ that is supported in the set $\{k\colon (k,q)=1\}$
 can be expressed as a finite linear combination of Dirichlet characters with period $q$.

  The equivalence of \eqref{it:chi-chi} and \eqref{it:ChiSeries} follows from
 Theorem~\ref{T:Halasz}.
\qed

\subsection{$U^2$ norm and aperiodic multiplicative functions}

We establish here two preliminary results used in the proof of Theorem~\ref{th:aperiod_uniform}.

 \begin{lemma}
 \label{lem:inverse-mul}
Let  $\ve>0$. There exist $\delta:=\delta(\ve)>0$ and $Q:=Q(\ve)\in \N$
such that the following holds: If $f\in \CM$ is a multiplicative function and  $\limsup_{N\to+\infty}\norm{f}_{U^2[N]}\geq \ve$, then there exists $p\in \N$ with $0\leq p<Q$
such that
$$
\limsup_{N\to+\infty}\Big|\E_{n\in [N]} f(n) \, \e\big(n\frac{p}{Q}\big)\Big|\geq \delta.
$$
\end{lemma}
\begin{remark} Note that the implication fails for arbitrary bounded sequences; consider for example a sequence of the
form $(\e(n\alpha))_{n\in \N}$ where $\alpha$ is irrational in place of $(f(n))_{n\in \N}$.
\end{remark}
\begin{proof}
Let $\ve>0$ and $f\in \CM$ be such that
$\norm f_{U^2[N]}\geq\ve$ for infinitely many values of $N\in \N$.
For these values of $N$ let  $\wt N$ be the smallest prime in the interval $(2N, 4N]$.
Since  $\tN\leq 4 N$,
  by Definition~\ref{def:Us-interv} of the norm $U^2[N]$  and  Lemma~\ref{L:norm-N},
 we deduce that  for  these values of $N$ we have
\begin{equation}
\label{E:U2ve}
\norm{f_N}_{U^2(\Z_{\tilde N})}\geq \frac{\ve}{4}.
\end{equation}
We apply Corollary~\ref{cor:katai}, with  $\ve^2/2^4$ in place of $\theta$,
for $\ell=4$ (see notation in Section~\ref{subsec:decomposition}) and for  $\wt N$  chosen as above.
We get some  positive integers $N_0$, $Q$, $V$ that depend only on $\ve$ such that \eqref{eq:Fourier_chi} holds. Henceforth, we assume that  $N\geq N_0$   is such that \eqref{E:U2ve} holds.

 Combining \eqref{eq:U2Fourier2} and \eqref{E:U2ve} we deduce that  there exists $\xi\in\Z_\tN$
 such that
\begin{equation}
\label{eq:aperiodic1}
\big|\E_{n\in[\tN]}f_N(n)\, \e\big(n \frac{\xi}{\tN}\big)\big|\geq \frac{\ve^2}{2^4}.
\end{equation}
By implication \eqref{eq:Fourier_chi} of  Corollary~\ref{cor:katai},  there exists $p\in \N$ with $0\leq p\leq Q$ such that
\begin{equation}
\label{E:QRN}
\Bigl|\frac \xi \tN-\frac pQ\Bigr|\leq \frac V\tN.
\end{equation}
 We deduce that there exist
  infinitely many $N\in\N$
  for which the above estimate holds for the same value of $p$.
  Henceforth, we further  restrict ourselves to these values of $N$.

Let $N_1:=\lfloor  \ve^2 N/2^4\rfloor$. Since $\tN\geq 2N$, it follows from~\eqref{eq:aperiodic1} that
$$
\big|\E_{n\in[\tN]}\one_{[N_1,N]}(n)\, f(n)\, \e\big(n\frac \xi \tN\big)\big|\geq \frac{\ve^2}{2^5}.
$$
We let
$$
\delta:= \frac{\ve^2}{2^{10}\,\pi V}\ \text{ and } \
L:=\Bigl\lfloor\delta N\Bigr\rfloor
$$
and suppose that $N\in\N$ is sufficiently large so that $L\geq 2$. We partition the interval $[N_1,N]$ into intervals of length between $L$ and $2L$. The number of these intervals is bounded by $N/L$ and thus one of them, say $J$, satisfies
\begin{equation}
\label{eq:aperiodic4}
\big|\E_{n\in[\tN]}\one_J(n)\, f(n)\, \e\Bigl(n\frac \xi \tN\Bigr)\big|\geq \frac LN\, \frac{\ve^2} {2^5}\geq \frac{\delta \ve^2}{2^6}.
\end{equation}
Let $n_0$ be the first term of the interval. For every $n\in J$ we have
$$
\Big|\e\big(n \big(\frac \xi \tN-\frac{p}{Q}\big)\big)-\e\big(n_0\big(\frac\xi \tN-\frac pQ\big)\big)\Big|
\leq(n-n_0)2\pi\Big|\frac\xi \tN-\frac pQ\Big|\leq 4\pi V\delta,
$$
where the last estimate follows from \eqref{E:QRN} and the fact that the length of $J$ is at most $2L$.
Combining this estimate with \eqref{eq:aperiodic4}, and using  again the fact that the length of $J$ is at most $2L$, we get
$$
\big|\E_{n\in[\tN]}\one_J(n)\, f(n)\, \e\bigl(n\frac pQ\bigr)\big|
\geq \frac{\delta \ve^2}{2^6}-\frac{2L}\tN\, 4\pi V\delta \geq \frac{\ve^4}{2^{17}\,\pi V}.
$$
Writing $J=(N_2,N_3]$ where $N_1\leq N_2\leq N_3\leq N$ we have $\one_J=\one_{[N_3]}-\one_{[N_2]}$. For $N_4:=N_2$ or $N_4:=N_3$ we have
$$
\big|\E_{n\in[N_4]}f(n)\, \e\bigl(n\frac pQ\bigr)\big|\geq \big|\E_{n\in[\tN]}\one_{[N_4]}(n)\, f(n)\, \e\bigl(n\frac pQ\bigr)\big|\geq \frac{\ve^4}{2^{18}\,\pi V}.
$$
Since  $N_4\geq N_1\geq \ve^2N/2^{6}$  we have that $N_4\to +\infty$ as $N\to +\infty$ and  we deduce that
$$
\limsup_{N\to+\infty}\big|\E_{n\in[N]}f(n)\, \e\bigl(n\frac pQ\bigr)\big|>0.
$$
This completes the proof.
\end{proof}

\begin{corollary}\label{C:92}
If $f\in\CM$, then $f$ is aperiodic if and only if $\norm f_{U^2[N]}\to 0$ as $N\to+\infty$.
\end{corollary}
\begin{proof}
The necessity of the condition follows immediately from Lemma~\ref{lem:inverse-mul} and the characterization~\eqref{it:chi-npq} of aperiodic functions given in Proposition~\ref{prop:equiv-aperiodic}. The sufficiency follows from Lemma~\ref{lem:NormsUs}, Lemma~\ref{lem:U2ent}, and the same characterization.
\end{proof}

\subsection{Proof of Theorem~\ref{th:aperiod_uniform}}
 We move now to the proof of  Theorem~\ref{th:aperiod_uniform} which makes essential use of
  Theorem~\ref{T:DecompositionI}.
\begin{proof}[Proof of Theorem~\ref{th:aperiod_uniform}]
 For  $N\in \N$ let  $\wt N$ be the smallest prime in the interval $(2N, 4N]$. Suppose that
$\norm{ f }_{U^s[N]}$ does not converge to $0$ as $N\to+\infty$.   Since $\tN\leq 4 N$, by Definition~\ref{def:Us-interv} of the norm $U^s[N]$  and  Lemma~\ref{L:norm-N}, we have  that  $\norm{ f_ N }_{U^s(\Z_{\tN})}$  does not tend to zero.
  As a consequence,
 there exists $\ve>0$ such that
$\norm{ f_ N }_{U^s(\Z_{\tN})}\geq 2\ve$ for infinitely many $N\in\N$ for  which Theorem~\ref{T:DecompositionI} applies (for $\ell=4$). Let  $Q, R\in \N$, $ f_ {N,\st}$ and  $ f_ {N,\un}$ be given by  Theorem~\ref{T:DecompositionI}
for this value of $\varepsilon$ and these values of $N$.
Then Property~\eqref{it:weakUs-3} of Theorem~\ref{T:DecompositionI} implies that $\norm{ f_ {N,\st} }_{U^s(\Z_{\tN})}\geq \ve$.

By Property~\eqref{it:weakUs-2} of Theorem~\ref{T:DecompositionI}, the cardinality of the spectrum of $ f_ {N,\st}$ (that is, the set of $\xi\in\tZN$ such that $\widehat{ f_ {N,\st}}(\xi)\neq 0$) is bounded by a positive real  $S$ that depends only on $\varepsilon$ and $s$. Since  the $U^s(\Z_{\tN})$-norm of each function $\e(n\xi/\wt N)$ is equal to $1$ for every $s\geq 2$, it follows that
 there exist $\xi\in\Z_\tN$ (depending on $N$) such that $|\widehat{ f_ {N,\st}}(\xi)|\geq \ve/S$.
 By Property~\eqref{it:weakUs-1} of Theorem~\ref{T:DecompositionI},
 we have $\widehat{ f_ {N,\st}}(\xi)=\widehat{\phi_{N}}(\xi)
 \widehat{ f_ N}(\xi)$ and since $|\widehat{ f_ {N,\st}}(\xi)|\geq \ve/S$ and
  $|\widehat{\phi_{N}}(\xi)|\leq 1$ it
  follows that $|\widehat{ f_ {N}}(\xi)|\geq \ve/S$. We deduce from  \eqref{eq:U2Fourier2}  that
   $\norm{ f }_{U^2[N]}\geq \norm{ f_ N}_{U^2(\Z_\tN)}\geq \ve/S$.
Hence,  $\norm f_ {U^2[N]}$ does not converge to zero as $N\to+\infty$. Corollary~\ref{C:92}
gives that $ f $ is not aperiodic, completing the proof.
\end{proof}

\subsection{Background on quadratic fields}
Our next goal is to prove Theorem~\ref{th:chowla}.
Here $d$ is a positive integer, and we adopt the notation  and refer the reader to Section~\ref{SS:chowla} for the definition of
 $\tau_d$, $\CN(z)$ and $Q_d$.

We  recall some classical facts about the ring $\Z[\tau_d]$.
For every $N\in\N$ we let
\begin{equation}\label{E:BN}
B_N:=\bigl\{(m,n)\in\Z^2\colon Q_d(m,n)\leq N^2\bigr\}.
\end{equation}
Recall that $Q_d$ is a positive definite quadratic form, and in this case it is not
hard to see that there exist constants $R_d\in \N$ and  $c_d>0$ such that
\begin{gather}
\label{eq:cardBN}
\Bigl[-\frac N{R_d},\frac N{R_d}\Bigr]^2\subset B_N\subset [-R_dN,R_dN]^2;\\
\label{eq:nombre_norme}
\lim_{x\to+\infty}\frac 1x \bigl|\{z\in\Z[\tau_d]\colon \CN(z)\leq x\}\bigr|=c_d.
\end{gather}
The units of $\Z[\tau_d]$ are the elements of norm $1$; their number  is denoted by  $N_1(d)$ and  is equal to $2, 4$ or $6$.
In general, the ring  $\Z[\tau_d]$ is not a principal ideal domain, but it is always a Dedekind domain, and has the property of unique factorization of ideals into prime ideals~(see for example \cite[Theorem 5.3.6]{MuEs}).

We say that $\alpha\in\Z[\tau_d]$ is a \emph{prime element} if $\alpha\neq 0$ and the ideal $(\alpha)$ spanned by $\alpha$ is a prime ideal. Note that prime elements are irreducible, but  the converse is not true in general.
To avoid ambiguities, we  do not abbreviate the expressions ``prime integer'', ``prime ideal'', and ``prime element'' of $\Z[\tau_d]$.
 Any two  prime  elements that generate the same ideal, or equivalently,  that can be obtained from one another by multiplication by a unit, are called \emph{associates}
and we  identify them.
Some prime elements of $\Z[\tau_d]$ are prime integers and some other are not;  we  plan to work
with non-integer prime elements only. We write
\begin{equation}
\label{eq:def-Id}
\CI_d:=\{\text{non-integer prime elements of }\Z[\tau_d]\},
\end{equation}
 and  throughout the argument we take into account the aforementioned identification, that is,
 we assume that no two elements in $\CI_d$ are associates.

In the sequel we use the  following immediate observations. If $\alpha\in\CI_d$, then $\CN(\alpha)=|\alpha|^2$ is a prime integer;  if $z\in\Z[\tau_d]$ is such that $\CN(\alpha)$ divides $\CN(z)$, then $z$ is a multiple of $\alpha$ or of $\overline\alpha$ (or of both).
We have
\begin{equation}
\label{eq:sum-squares}
\sum_{\alpha\in\CI_d}\frac 1{\CN(\alpha)^2}
\leq\sum_{z\in\Z[\tau_d], z\neq 0}\frac 1{\CN(z)^2}<+\infty
\end{equation}
where the convergence of the second series follows from~\eqref{eq:nombre_norme}.
We also need a deeper result.
We have
$$
\sum_{\alpha\text{ prime element of }\Z[\tau_d]}\frac 1{\CN(\alpha)}=N_1(d)
\sum_{\fp \text{ principal prime ideal, } \fp\neq\{0\}}\frac 1{\CN(\fp)}=+\infty.
$$
The divergence of the last series can be deduced from the
 Chebotarev density theorem (see for example \cite[Theorem 13.4]{Ne}), 
 but also a much more elementary proof can be found, for example, on pages 148--149 of~\cite{MuEs}.
On the other hand, writing as usual $\P$ for the set of prime integers, we have
$$
\sum_{p\in\P}\frac 1{\CN(p)}=\sum_{p\in\P}\frac 1{p^2}<+\infty
$$
and thus we have
\begin{equation}
\label{eq:sum_inverse}
\sum_{\alpha\in\CI_d}\frac 1{\CN(\alpha)}=+\infty.
\end{equation}

\subsection{The K\'atai orthogonality criterion for \mbox{$\Z[\tau_d]$}}
\label{subsec:Katai-Ztau}
Next we prove a variant of the orthogonality criterion of K\'atai
(see Lemma~\ref{lem:katai}) that works for the rings $\Z[\tau_d]$.
Given the basic information about the rings $\Z[\tau_d]$ recorded above,
the proof is a straightforward adaptation of the original argument of K\'atai~\cite{K86};  we give it for completeness.
\begin{lemma}[Tur\'an-Kubilius for subsets of \mbox{$\Z[\tau_d]$}]
\label{lem:Tur_Ku}
Let $\CP$ be a finite subset of $\CI_d$ and
for $z\in\Z[\tau_d]$  let
$$
\CA:=\sum_{\alpha\in \CP}\frac{1}{\CN(\alpha)},
\quad \omega(z):=\sum_{\alpha\in \CP,\alpha|z}1.
$$
 Then for every $x\in \N$ we have
$$
\sum_{z\in\Z[\tau_d]\colon \CN(z)\leq x} |\omega(z)-\CA|\ll \sqrt{\CA} \cdot
x+ |\CP|\cdot o(x)
$$
where  the implied constant and the $o(x)$ term depend only on  $d$.
\end{lemma}
\begin{proof}
Using \eqref{eq:nombre_norme} and the Cauchy-Schwarz inequality, we see that  it suffices to show that
\begin{equation}\label{E:wanted}
\sum_{z\in\Z[\tau_d]\colon \CN(z)\leq x} (\omega(z)-\CA)^2\ll \CA \cdot x+
|\CP|^2\cdot  o(x).
\end{equation}
The left hand side is equal to
\begin{equation}\label{E:00}
\sum_{z\in\Z[\tau_d]\colon \CN(z)\leq x} \omega(z)^2+c_d\CA^2\cdot x+ o(\CA^2x)-2\CA
\sum_{z\in \Z[\tau_d]\colon \CN(z)\leq x} \omega(z).
\end{equation}
where the constant $c_d$ is defined in \eqref{eq:nombre_norme}.

As every $z \in \Z[\tau_d]$ that is divisible by some $\alpha\in \CP$ is of the form $\alpha w$ for some $w\in\Z[\tau_d]$,  using \eqref{eq:nombre_norme} we see that
$$
\sum_{z\in\Z[\tau_d]\colon \CN(z)\leq x,\  \alpha|z} 1=
\Big|\Bigl\{w\in\Z[\tau_d]\colon \CN(w)\leq \frac{x}{\CN(\alpha)}\Bigr\}\Big|
=c_d\frac{x}{\CN(\alpha)}+o(x).
$$
Summing over $\alpha\in \CP$ we find that
\begin{equation}\label{E:11}
\sum_{z\in\Z[\tau_d]\colon \CN(z)\leq x}\omega(z)=c_d \CA\cdot  x+|\CP|\cdot o(x).
\end{equation}
In the same way,
\begin{multline*}
\sum_{z\in\Z[\tau_d]\colon \CN(z)\leq x}\omega(z)^2
=\sum_{\alpha\neq\beta\in\CP}\;\sum_{z\in\Z[\tau_d]\colon \CN(z)\leq x,\  \alpha\beta|z} 1
+\sum_{\alpha\in\CP}\sum_{z\in\Z[\tau_d]\colon \CN(z)\leq x,\  \alpha|z} 1\\
=\sum_{\alpha,\beta\in\CP}\Bigl(c_d\frac x{\CN(\alpha)\CN(\beta)}+o(x)\Bigr)
-\sum_{\alpha\in\CP}\Bigl(c_d\frac x{\CN(\alpha)^2}+o(x)\Bigr)
+\sum_{\alpha\in\CP}\Bigl(c_d\frac x{\CN(\alpha)}+o(x)\Bigr)\\
=c_d\CA^2 \cdot x+\CA
\cdot O(x)+|\CP|^2\cdot o(x)
\end{multline*}
by~\eqref{eq:sum-squares}. Combining this formula with~\eqref{E:00} and \eqref{E:11} we get \eqref{E:wanted},  completing the proof.
\end{proof}

\begin{definition}
We say that a function $f\colon\Z[\tau_d]\to\C$ is \emph{multiplicative} if $f(zz')=f(z)f(z')$ whenever $\CN(z)$ and $\CN(z')$ are relatively prime. We denote by $\CM_d$  the family of multiplicative functions $f\colon\Z[\tau_d]\to\C$ with modulus at most $1$. We remark that for $g\in\CM$ and $r\in\N$ the function $f\colon z\mapsto g(\CN(z)^r)$ belongs to $\CM_d$.
\end{definition}

\begin{lemma}[K\'atai estimate for \mbox{$\Z[\tau_d]$}]
\label{lem:KataiZd}
 Let
 $f\in\CM_d$ be a multiplicative function  and $h\colon \Z[\tau_d]\to \C$
 be an arbitrary  function of modulus at most $1$.
For $x\in\N$, let also
\begin{gather*}
S(x):=\sum_{z\in\Z[\tau_d]\colon \CN(z)\leq x}f(z)\, h(z),\\
C(x):=\sum_{\alpha,\beta \in \CP, \alpha\neq\beta}\Big|
\sum_{z\in \Z[\tau_d]\colon \CN(z)\leq
\min\{x/\CN(\alpha),\ x/\CN(\beta)\}}f(\alpha z)\cdot \overline{f(\beta z)}\Big|
\end{gather*}
where
 $\CP$ is a finite subset of $\CI_d$.
 Then we have the estimate
$$
\Big|\frac{S(x)}{x}\Big|^2 \ll \frac{1}{\CA}+\frac 1{\CA^2}+ \frac{|\CP|^2}{\CA^2}\,o(1)
+\frac{1}{ \CA^2} \frac{C(x)}{x},
$$
where $\CA:=\sum_{\alpha\in\CP}\CN(\alpha)\inv$,   the implied constant, and the $o(1)$ term depend only on $d$.
\end{lemma}
\begin{proof}
Let $\omega(z)$ be defined as in Lemma~\ref{lem:Tur_Ku} and
$$
S'(x):=\sum_{w\in\Z[\tau_d]\colon \CN(w)\leq x}f(w)\, h(w)\, \omega(w).
$$
From Lemma~\ref{lem:Tur_Ku} we deduce that
$$
|S'(x)-\CA\cdot S(x)|\ll \sqrt{\CA} \cdot x+ |\CP|\cdot o(x).
$$
The formula defining $S'(x)$ can be rewritten as
$$
S'(x)=\sum_{w\in\Z[\tau_d], \alpha\in\CP\colon \alpha| w\text{ and } \CN(w)\leq x}f(w)\, h(w)
=
\sum_{z\in\Z[\tau_d],\ \alpha\in\CP\colon \CN(\alpha z)\leq x}f(\alpha z)\, h(\alpha z).
$$
In this sum, the term corresponding to a pair $(z,\alpha)$ is equal to $f(\alpha)f(z)h(\alpha z)$ except if $\CN(\alpha)$  and $\CN(z)$ are not relatively prime. Since $\CN(\alpha)$ is a prime integer, this holds only if $\CN(\alpha)$ divides $\CN(z)$, that is, if $\alpha$ or $\overline\alpha$ divides $z$.  We let
$$
S''(x):=\sum_{z\in\Z[\tau_d],\  \alpha\in\CP\colon \CN(\alpha z)\leq x} f(\alpha)\, f(z)\, h(\alpha z).
$$
Since  $|f|\leq 1$ and $|h|\leq 1$, it follows that
\begin{multline*}
|S(x)-S''(x)|\leq
2\, \bigl|\{ \bigl(z,\alpha)\in \Z[\tau_d]\times\CP\colon \alpha\text{ or }\overline\alpha\text{ divides } z\text{ and }\CN(\alpha z)\leq x
\bigr\}\bigr|\\
\leq 4\,\sum_{\alpha\in\CP} \bigl|\bigl\{ z\in\Z[\tau_d]\colon \CN(z)\leq x\CN(\alpha)^{-2}\bigr\}\bigr|
\leq 4x\sum_{\alpha\in\CP}\CN(\alpha)^{-2}\ll x
\end{multline*}
by~\eqref{eq:sum-squares}, where the implied constant depends only on $d$.

 We rewrite $S''(x)$ as
$$
S''(x)=\sum_{z\in\Z[\tau_d]\colon |\CN(z)|\leq x}\ f(z) \sum_{\alpha\in \CP\colon
\CN(\alpha)\leq x/\CN(z)}f(\alpha)\, h(\alpha z).
$$
Using \eqref{eq:nombre_norme} and the Cauchy-Schwarz inequality we deduce that
$$
|S''(x)|^2\ll  x
\sum_{z\in\Z[\tau_d]\colon |\CN(z)|\leq x}\Big|\sum_{\alpha\in \CP\colon\CN(\alpha)\leq x/|\CN(z)|}
f(\alpha)\, h(\alpha z)\Big|^2.
$$
Expanding the square, we get that this last expression is equal to
$$
x\, \sum_{\alpha,\beta\in\CP}\quad  \sum_{z\in\Z[\tau_d]\colon \CN(z)\leq x/\CN(\alpha),\ \CN(z)\leq x/\CN(\beta)}
f(\alpha)\, \overline{f(\beta)} \, h(\alpha z)\, \overline{h(\beta z)}.
$$
By~\eqref{eq:nombre_norme}, the contribution of the diagonal terms with $\alpha=\beta$
is at most
$$
x\, \sum_{z\in\Z[\tau_d],\  \alpha\in \CP\colon \CN(z)\leq x/\CN(\alpha)}1\ll \CA \cdot x^2
$$
and the contribution of the off diagonal terms is bounded by  $x\,C(x)$, where $C(x)$ was defined in the statement.

Combining the previous estimates we get that
\begin{multline*}
|\CA\cdot S(x)|^2\ll |S'(x)-\CA\cdot S(x)|^2+|S'(x)-S''(x)|^2+|S''(x)|^2
\\ \ll
 \CA \cdot x^2 +  |\CP|^2\,o(x^2) + x^2 + \CA\,x^2 + x\,C(x).
\end{multline*}
The asserted estimates follows upon dividing by $\CA^2\cdot x^2$.
\end{proof}
Using the previous estimate and \eqref{eq:nombre_norme} we get as an immediate corollary the following orthogonality criterion:

\begin{proposition}[K\'atai orthogonality criterion for \mbox{$\Z[\tau_d]$}]
\label{prop:KataiZd}
Let $d\in\N$, $\tau_d$ and $\CI_d$ be as above, and  let $\CP$ be a subset of $\CI_d$ such that
$$
\sum_{\alpha\in \CP}\frac 1{\CN(\alpha)}=+\infty.
$$
Let $h_x\colon \Z[\tau_d]\to \C$, $x\in \N$,  be  arbitrary functions of modulus  at most  $1$ such that
$$
\lim_{x\to +\infty}\E_{z\in\Z[\tau_d]\colon \CN(z)\leq  x}\;h_x(\alpha z)\cdot
\overline{h_x(\beta z)}=0
$$
for every $\alpha,\beta \in \CP$ with $\alpha\neq\beta$.
 Then
$$
 \lim_{x\to +\infty}\sup_{f\in \CM_d} \big|
 \E_{z\in\Z[\tau_d]\colon \CN(z)\leq x}\;
 f(z)\, h_x(z)\big|=0.
 $$
\end{proposition}

\begin{remark}
 By \eqref{eq:sum_inverse} the assumption $\sum_{\alpha\in \CP}\CN(\alpha)\inv=+\infty$
is satisfied for $\CP:=\CI_d$ and also for any set $\CP$ obtained by removing finitely many elements from $\CI_d$.
\end{remark}

\subsection{Estimates involving Gowers norms}
The following lemma will be crucial in the proof of Theorem~\ref{th:chowla}.
The method of proof is classical, see for example~\cite[Proof of Proposition 7.1]{GT10b}, we summarize it for completeness.
 Lemma~\ref{L:UnifromityEstimates2} below is proved in a similar fashion.
\begin{lemma}
\label{P:semest}
Let  $s\in \N$ and $L_j(m,n)$, $j=1,\ldots, s$,  be  linear forms with integer coefficients
 and suppose that either $s=1$ or $s>1$ and the linear forms $L_1,L_j$ are linearly independent for $j=2,\ldots,  s$.
 For $j=1,\ldots, s$ let $h_j\colon \Z\to \C$ be  bounded 
 functions.
 Suppose that  $h_1$ is an even function and  $\norm{h_1}_{U^{s'}[N]}\to 0$ as $N\to+\infty$ where  $s':=\max\{s-1,2\}$. Then
\begin{equation}\label{E:fjLj}
\lim_{N\to +\infty} \E_{1\leq m,n\leq N} {\bf 1}_{K_N}(m,n) \prod_{j=1}^sh_j(L_j(m,n))=0
\end{equation}
where   $K_N$, $N\in \N$, are arbitrary  convex subsets of $[-N,N]^2$.
\end{lemma}
\begin{proof}
Without loss of generality we can assume that $|h_j|\leq 1$ for $j=1,\ldots, s$.
Furthermore, we can assume that the linear forms $L_j$ are pairwise independent. Indeed, if for some distinct $i,j$  the linear forms $L_i$ and $L_j$ are integer multiples of the same linear form $L$, then by hypothesis, $i$ and $j$ must both be greater
 than $1$, and we can write $h_i(L_i(m,n))h_j(L_j(m,n))=h(L(m,n))$ for some bounded 
 function $h$.

Next, as we want to have Fourier analysis tools available,  we reduce matters  to  averages on a cyclic group.
For $j=1,\dots,s$, we write  $L_j(m,n)=\kappa_jm+\lambda_jn$ where $\kappa_j,\lambda_j\in\Z$.
For every  $N\in\N$, let
 $\tN$ be the smallest   prime such that
$\wt N> 2N$,   $\wt N>2|L_j(m,n)|$ for $m,n\in[N]$ and $j=1,\ldots, s$,     and $\wt N>|\kappa_i\lambda_j-\lambda_i\kappa_j|$ for $i,j=1,\ldots,s$.
  Then $\wt{N}/N$ is bounded by a constant that depends on the linear forms $L_1,\ldots, L_s$ only.

 For $j=1,\ldots, s$, let $\wt h_j\colon \N \to \C$ be periodic of period $\wt N$ and
   equal to  $h_j$ on the interval
 $\bigl[-\lceil \wt N/2\rceil,   \lfloor \wt N/2\rfloor\bigr)$.
For $(m,n)\in K_N$, since $|L_j(m,n)|<\wt N/2$,  we have $h_j(L_j(m,n))=\wt h_j(L_i(m,n))$. Hence,
\begin{equation}\label{E:fjLj'}
 \E_{ m,n\in[N]} {\bf 1}_{K_N}(m,n) \prod_{j=1}^sh_j(L_j(m,n))
 =\big(\frac{\wt N}{N}\big)^2
 \E_{ m,n\in \Z_{\tN}} {\bf 1}_{K_N}(m,n) \prod_{j=1}^s\wt h_j(L_j(m,n)).
\end{equation}
Henceforth, we work with the right hand side and assume that the linear forms $L_j$ and the functions $\wt h_j$ are defined on $\Z_\tN$. We first prove that the right hand side of \eqref{E:fjLj'} converges to $0$ as $N\to +\infty$ under the assumption that \begin{equation}\label{E:ToZero}
\lim_{N\to+\infty}\norm{\wt h_1}_{U^{s'}(\Z_\tN)}=0
\end{equation}
and then verify that this assumption is satisfied (the assumption that $h_1$ is even is only needed here).

Our next goal is to remove the cutoff ${\bf 1}_{K_N}(m,n)$.
 One can follow the exact same method as in~\cite[Proposition~7.1]{GT10b}; we  skip the details and only give a sketch.  The idea is to imbedded $\Z_\tN^2$ in the torus $\T^2$ in the natural way and to represent $K_N$ as the intersection of $\Z_\tN^2$ with a ``convex'' subset of $\T^2$. This convex set is then   approximated by a sufficiently regular function on $\T^2$ which is in turn approximated by a trigonometric polynomial  with bounded coefficients and
 spectrum of bounded cardinality (with respect to $N$). After these reductions, we are left with showing that
$$
\lim_{N\to+\infty}\max_{\eta,\xi\in\Z_\tN}
\Bigl|\E_{ m,n\in \Z_\tN} \e\big(m\frac{\eta}{\wt N} +n\frac{\xi}{\wt  N}\big) \prod_{j=1}^s\wt h_j(L_j(m,n)) \Bigr|
=0.
$$

We now  show that we can restrict ourselves to  the case where  $\xi=\eta=0$. Indeed, if the linear form $m\eta+n\xi$ does not belong to the linear span of the forms $L_j$, then the average vanishes.
On the other hand, if the linear form $m\eta+n\xi$ belongs to the linear span of the forms $L_j$, then we can remove the exponential term by multiplying each function $h_j$ by a complex exponential of the form $\e(\theta_j n/\wt N)$ for some $\theta_j\in\Z_\tN$, and since $s'\geq 2$, this modification does not change the $U^{s'}(\Z_\tN)$-norm of the functions. Note that now $h_1$ is not necessarily even, but we  no longer need this assumption. Therefore, we are reduced to proving that
\begin{equation}\label{E:fjLj''}
\lim_{N\to+\infty}\,\E_{ m,n\in \Z_\tN} \prod_{j=1}^s\wt h_j(L_j(m,n)) =0.
\end{equation}

The pairwise independence of the linear forms $L_j$,  combined with the
 last condition on $\wt N$, imply that the forms $L_j$ on $\Z_\tN\times\Z_\tN$ are pairwise linearly independent over $\Z_\tN$.
Using this and an    iteration of the Cauchy-Schwarz inequality (see for example~\cite[Theorem~3.1]{T06}) we get
$$
\big|\E_{ m,n\in \Z_{\tN}}
\prod_{j=1}^s\wt h_j(L_j(m,n))\big|\leq
\norm{\wt h_1}_{U^{s'}(\Z_{\tN})}.
 $$
From this  estimate and \eqref{E:ToZero} it follows that \eqref{E:fjLj''} holds.

It thus remains  to verify that \eqref{E:ToZero} holds.
We write
$$
\Z_\tN=I\cup J\cup\{0\}\ \text{ where }\ I:= [1,\lfloor \wt N/2\rfloor)\ \text{ and }\
J:= [\lfloor N/2\rfloor, N).
$$
 Since $\norm{\one_{\{0\}}\cdot \wt h_1}_{U^{s'}(\Z_\tN)}\to 0$ when $N\to+\infty$,
 it suffices to show that $\norm{\one_I\cdot h_1}_{U^{s'}(\Z_\tN)}\to 0$ and
$\norm{\one_J\cdot h_1}_{U^{s'}(\Z_\tN)}\to 0$.
By hypothesis, $\norm{h_1}_{U^{s'}[\wt N]}\to 0$ and by  Lemma~\ref{lem:NormsUs} in the Appendix, $\norm{h_1}_{U^{s'}(\Z_\tN)}\to 0$.
By Lemma~\ref{lem:UsLN} in the Appendix, $\norm{\one_I\cdot h_1}_{U^{s'}(\Z_\tN)}\to 0$. Since $\wt h_1$ and $h_1$ coincide on $I$, we have $\norm{\one_I\cdot \wt h_1}_{U^{s'}(\Z_\tN)}\to 0$.
 By assumption, $h_1$ is an even function, hence,  for $n\in J$ we have $\wt h_1(n)=\wt h_1(n-N)=h_1(n-N)=h_1(N-n)$. The map $n\mapsto N-n$ maps the interval $J$ onto the interval $J':=[1,\lceil \wt N/2\rceil]$ and thus $\norm{\one_J\cdot \wt h_1}_{U^{s'}(\Z_\tN)}=
 \norm{\one_{J'}\cdot  h_1}_{U^{s'}(\Z_\tN)}$. The last  quantity tends to $0$ as $N\to+\infty$ by the same argument as above. This completes the proof.
\end{proof}

\subsection{Proof of Theorem~\ref{th:chowla}}
\label{subsec:chowla2}
   Theorem~\ref{th:chowla} follows from the following stronger result:    \begin{theorem}
\label{th:chowla2}
 For $s\in \N$
 let the linear forms $L_1,\ldots, L_s$ and the quadratic form $Q$ be as in the statement of Theorem~\ref{th:chowla}.  Furthermore, let $g\in\CM$  be  arbitrary,  $f_1\in\CM$ be  aperiodic, and suppose that both multiplicative functions are extended to  even functions on $\Z$.
If $s\geq 2$, let also $f_2,\ldots, f_s\colon \Z\to \C$  be arbitrary bounded  functions.
 Then
$$
\lim_{N\to +\infty} \,\E_{1\leq m,n\leq N} \, {\bf 1}_{K_N}(m,n) \,
g(Q(m,n))\,
\prod_{j=1}^sf_j(L_j(m,n))=0,
$$
where   $K_N$, $N\in \N$, are arbitrary  convex subsets of $[-N,N]^2$.
\end{theorem}
\begin{remark}
We deduce that if $Q,L_1,\ldots, L_s$ are as above, $r\in \N$, $R_j\in \Z[t]$ are arbitrary polynomials,
 and $f\in \CM$ is an aperiodic  completely multiplicative function, then for
$P(m,n)=Q(m,n)^r L_1(m,n) \prod_{j=2}^sR_j(L_j(m,n))$ we have
$$
\lim_{N\to +\infty} \,\E_{1\leq m,n\leq N}f(P(m,n))=0.
$$
\end{remark}
\begin{proof}
The proof proceeds in several steps.

\subsubsection{Reduction to the quadratic form $Q_d$}
By assumption, the bilinear form $Q$ can be written as $Q(m,n)=Q_d(F(m,n))$ where $d\in\N$ and $F\colon\Z^2\to\Z^2$ is given by a $2\times 2$ matrix with integer entries and determinant $\pm 1$.

For $ j=1,\ldots,s$ we let  $L'_j(m,n):=L_j(F\inv(m,n))$. For every $N\in\N$, let $K'_N:=F([N]\times[N])$ and $K''_N:=F(K_N)$. There exists  $R\in \N$, which depends only on  the linear map $F$, such that $K'_N\subset[-RN,RN]^2$ for every $N\in\N$, and thus by~\eqref{eq:cardBN} we have
$K''_N\subset K'_N\subset B_{RR_d N}$.
The average in the statement can be rewritten as
$$
\frac{|B_{RR_dN}|}{N^2}\E_{(m,n)\in B_{RR_dN}} \, \one_{K''_N}(m,n) \,
g(Q_d(m,n))\,
 \prod_{j=1}^sf_j(L'_j(m,n))
$$
where $B_N$ is defined in \eqref{E:BN}.
By~\eqref{eq:cardBN} we have  $|B_{RR_dN}|=O(N^2)$ and substituting $L'_j$ for $L_j$ for $j=1,\ldots,s$ and $K''_N$ for $K_N$  we have  reduced matters to showing that
\begin{equation}
\label{eq:av-cho1}
\lim_{N\to+\infty}\E_{(m,n)\in B_N}
\, \one_{K_N}(m,n) \,
g(Q_d(m,n))\,
 \prod_{j=1}^sf_j(L_j(m,n))=0
\end{equation}
where $K_N$ are  convex subsets of $B_N$ for  $N\in\N$.

\subsubsection{Applying the K\'atai orthogonality criterion}
In the rest of the proof, we identify $\Z^2$ with $\Z[\tau_d]$, by mapping $(m,n)\in\Z^2$ to $m+n\tau_d$.

For $j=1,\ldots,s$,  we write $L_j(m,n)=\kappa_jm+\lambda_jn$ and let
$$
\zeta_j:=\begin{cases}
2(\lambda_j-\kappa_j+\kappa_j\tau_d) & \text{if } d= 1 \bmod 4;\\
\lambda_j+\kappa_j\tau_d&\text{otherwise.}
\end{cases}
$$
For every $(m,n)\in\Z^2$ we have
\begin{equation}
\label{eq:phi-imag}
L_j(m+n\tau_d)= \frac 1{\sqrt{-d}}\Im(\zeta_j(m+n\tau_d))
\end{equation}
and thus we are reduced to proving that
\begin{equation}
\label{eq:av-ch2}
\lim_{N\to+\infty}\E_{z\in\Z[\tau_d]\colon \CN(z)\leq N^2}\; g(\CN(z))\,\one_{K_N}(z)\,
\prod_{i=1}^s f_i\bigl(\frac 1{\sqrt{-d}}\,\Im( \zeta_jz)\bigr)=0.
\end{equation}

We  use Proposition~\ref{prop:KataiZd} with
\begin{equation}\label{E:CP}
\CP\ :=\bigl\{\alpha\in\CI_d,\
\alpha\text{ and }\overline\alpha\text{ are non-associates and do not divide } \zeta_j\text{ for } j=1,\ldots, s.\bigr\}
\end{equation}
There are clearly only finitely many $\alpha\in\CI_d$ that do not satisfy the second condition. Moreover,  since we work on a quadratic number field,
 $\alpha$ and $\overline\alpha$ are associates only if the ideal $(\alpha)$ ramifies, and this  can happen for finitely many $ \alpha\in\CI_d$ by Dedekind's Theorem (see for example \cite[Section 5.4]{MuEs} or \cite[Proposition 8.4]{Ne}). By~\eqref{eq:sum_inverse}, we have that
\begin{equation}
\label{eq:div-inv-P}
\sum_{\alpha\in\CP}\frac 1{\CN(\alpha)}=+\infty.
\end{equation}

 Note that  $z\mapsto g(\CN(z))$ defines a   multiplicative function on $\Z[\tau_d]$. We apply
  Proposition~\ref{prop:KataiZd} and we are left with showing that
   \begin{claim}
  \label{cl:Kat-alpha-beta}
  If $\alpha$ and $\beta$ are distinct elements of $\CP$, then
\begin{equation}
\label{eq:av-ch3}
\lim_{N\to+\infty}
\E_{z\in\Z[\tau_d],\ \CN(z)\leq N^2}\,
\one_{K_N^*}(z)\,
\prod_{j=1}^s
f_j\bigl(\frac 1{\sqrt{-d}}\,\Im( \zeta_j\alpha z)\bigr)\,
\overline f_j\bigl(\frac 1{\sqrt{-d}}\,\Im( \zeta_j\beta z)\bigr)=0
\end{equation}
where
\begin{equation}
\label{E:KN}
K_N^*:=\alpha\inv K_N\cap \beta\inv K_N\subset [-N,N]^2.
\end{equation}
\end{claim}

\subsubsection{Independence of the linear forms}\label{cl:independent}
For $ j=1,\ldots,s$, we let
$$
\Lambda_j(z):= \frac 1{\sqrt{-d}}\Im(\zeta_j\alpha z)\ \text{ and } \
\Lambda'_j(z):= \frac 1{\sqrt{-d}}\Im(\zeta_j\beta z).
$$
Identifying $\Z[\tau_d]$ with $\Z^2$, these $2s$ maps can be thought of as linear forms with integer coefficients.
\begin{claim}
\label{cl:indep-linear}
The linear form $\Lambda_1$ is linearly independent of each of the forms $\Lambda_j$ for $j=2,\ldots,s$ and of each of the forms $\Lambda'_j$ for $j=1,\ldots,s$.
\end{claim}

\begin{proof}[Proof of Claim~\ref{cl:indep-linear}] Suppose  that for some non-zero  $a,b\in \Z$, some $j\in\{1,\ldots,s\}$, and some  $\alpha'\in\{\alpha,\beta\}$ with $(1,\alpha)\neq (j,\alpha')$ we have that
$a\Im(\zeta_1\alpha z)=b\Im(\zeta_j\alpha'z)$ for every $z\in\Z[\tau_d]$. Using this relation with $z:=1$ and $z:=\tau_d$ we get that
\begin{equation}
\label{eq:dependent}
a\, \zeta_1\,  \alpha= b\, \zeta_j\, \alpha'.
\end{equation}
We consider three cases:

(i) Suppose that $j=1$ and $\alpha'=\beta$. In this case, equation~\eqref{eq:dependent} gives that $\alpha/\beta$ is a rational which is impossible because $\alpha$ and  $\beta$ are distinct non-integer prime elements of $\Z[\tau_d]$ (recall that  no two elements of $\CI_d$  are associates).

(ii) Suppose that $j>1$ and $\alpha'=\alpha$. In this case we have that $\zeta_1/\zeta_j\in\Q$, and by~\eqref{eq:phi-imag} the linear forms $\phi_1$ and $\phi_j$ are linearly dependent, contradicting the hypothesis.

(iii) It remains to consider the case where  $j>1$ and $\alpha'=\beta$. Let $r_a$ be the exponent of $(\alpha)$ in the factorization of the ideal $(a)$ into  prime ideals of $\Z[\tau_d]$. Since $\alpha$ is a non-integer prime element of $\Z[\tau_d]$, $\overline\alpha$ is  also a non-integer prime element and, since $a$ is real,  the exponent of $(\overline\alpha)$ in the factorization of $(a)$ is also equal to $r_a$.  Since by hypothesis
$\alpha$ and $\overline\alpha$ are non-associate prime elements, it follows that $a$ can be written as $a=\alpha^{r_a}\overline\alpha^{r_a}c$ for some $c\in\Z[\tau_d]$ not divisible by $\alpha$ or $\overline\alpha$.
 In the same way, $b=
\alpha^{r_b}\overline\alpha^{r_b}d$ for
some non-negative integer $r_b$ and some  $d\in\Z[\tau_d]$  not divisible by $\alpha$ or $\overline\alpha$.
 Equation~\eqref{eq:dependent} gives
$$
\alpha^{r_a+1}\, \overline \alpha^{r_a} \, c\, \zeta_1= \alpha^{r_b}\, \overline\alpha^{r_b} \, d \, \zeta_j\, \beta.
$$
 By hypothesis, $\alpha$ and $\overline\alpha$ are non-associate prime elements, $\beta$ is a prime element non-associate to $\alpha$,  and $d$ and $\zeta_j$ are not divisible by $\alpha$. It follows that $r_b\geq r_a+1$. Similarly, $c$ and $\zeta_1$ are not divisible by $\overline\alpha$, and it follows that $r_a\geq r_b$, a contradiction.
 This completes the proof of Claim~\ref{cl:indep-linear}.
\end{proof}

\subsubsection{End of the proof}
In order to prove Claim~\ref{cl:Kat-alpha-beta}
we return to the coordinates $(m,n)$ of a point of $\Z^2$ which is identified with $\Z[\tau_d]$. Denoting the linear forms defined above by $\Lambda_j(m,n)$ and $\Lambda'_j(m,n)$, it remains to show  that
$$
\lim_{N\to+\infty}\E_{(m,n)\in B_N}\,\one_{K_N^*}(m,n)\,
\prod_{j=1}^s f_j(\Lambda_j(m,n))\,\overline f_j(\Lambda'_j(m,n))=0
$$
where the convex sets $K_N^*\subset[-N, N]^2$ were defined in \eqref{E:KN}.
By~\eqref{eq:cardBN}, it suffices to show that
\begin{equation}
\label{eq:av-ch4}
\lim_{N\to+\infty}\,\E_{-N\leq m,n\leq N}\,\one_{K_N^*}(m,n)\,
\prod_{j=1}^s f_j(\Lambda_j(m,n))\,\overline f_j(\Lambda'_j(m,n))=0.
\end{equation}
It follows from Claim~\ref{cl:indep-linear}  that  the form $\Lambda_1$ is linearly independent of each of the  forms $\Lambda_j$ for $j=2,\ldots, s$ and of each of the  forms $\Lambda'_j$ for $j=1,\ldots, s$.
By hypothesis, the function $f_1$ is an even aperiodic multiplicative function and thus $\norm{f_1}_{U^{2s-1}[N]}\to 0$ as $N\to+\infty$ by Theorem~\ref{th:aperiod_uniform}. All the hypotheses of
Lemma~\ref{P:semest}  are satisfied and \eqref{eq:av-ch4} follows. This completes the proof of Theorem~\ref{th:chowla2} and hence of Theorem~\ref{th:chowla}.
\end{proof}

\section{Partition regularity results}\label{S:Recurrence}

The goal of this  section is to prove Theorems~\ref{th:partition-regular2} (which implies Theorem~\ref{th:partition-regular1}) and \ref{th:partition-regular3}. For notational convenience we prove Theorem~\ref{th:partition-regular2} and indicate at the end of this section the  modifications needed to prove the more general Theorem~\ref{th:partition-regular3}.



Recall that our goal is to show that given admissible integers $\ell_0,\ldots, \ell_4$ (see Section~\ref{SS:parametric}), on every partition of $\N$ into finitely many cells there exist $k,m,n\in \Z$ such that the integers
\begin{equation}\label{E:xy}
x:=k\ell_0(m+\ell_1n)(m+\ell_2n), \qquad y:=k\ell_0(m+\ell_3n)(m+\ell_4n),
\end{equation}
are positive, distinct, and belong to the same cell.
We start with some  successive reformulations of the problem that culminate in the analytic
statement of Proposition~\ref{th:ergo2b}. We then prove this result using the structural  result of  Theorem~\ref{T:DecompositionII}.
\subsection{Reduction to a density regularity result}\label{SS:dildens}
We first recast Theorem~\ref{th:partition-regular2} as a density regularity
statement for dilation invariant densities on the integers.

We write $\Q^+$ for the multiplicative group of positive rationals.
Let $p_1,p_2,\ldots$ be the  sequence   of primes.
Then the sequence $(\Phi_M)_{M\in\N}$ of finite subsets of $\N$ defined by
$$
\Phi_M:=
\bigl\{ n\colon n\,|\,(p_1p_2\dots p_M)^M\bigr\}
=\{p_1^{k_1}\cdots p_M^{k_M}\colon 0\leq k_1,\ldots, k_M\leq M
\}
$$
is a \emph{multiplicative F\o lner sequence}.  This means that, for every $r\in\Q^+$, we have
\begin{equation}
\label{eq:Folner}
\lim_{M\to+\infty} \frac 1{|\Phi_M|}{|r\inv\Phi_M\triangle  \Phi_M|}=0,
\end{equation}
where for every subset $A$ of $\N$ and for every $r\in\Q^+$, we write
$$
r\inv A:=\{x\in\N\colon rx\in A\}=  \{r\inv y\colon y\in A\}\cap\N.
$$
To this
 multiplicative F\o lner sequence we
associate a notion of multiplicative density as follows:
\begin{definition}[Multiplicative density]
\label{def:multdensity}
The \emph{(upper) multiplicative density} $\dmult(E)$ of a subset $E$ of
$\N$ is defined as
$$
 \dmult(E):= \limsup_{M\to+\infty}
\frac{|E\cap\Phi_M|}{|\Phi_M|}.
$$
\end{definition}

We remark that the multiplicative density and the additive
density are non-comparable measures of largeness.
For
instance,   the set of odd numbers has zero multiplicative density with respect to any multiplicative  F\o lner sequence,
as has any set that omits all multiples of some positive
integer. On the other hand, it is not hard to construct sets with
multiplicative density $1$ that have additive density $0$ (see for
instance \cite{Be05}).

An important property of the multiplicative density, and the reason
we work with this notion of largeness,  is its invariance under dilations. Indeed,  for every $E\subset
\N$ and every $r\in\Q^+$, it follows from~\eqref{eq:Folner} that
$$
\dmult(E)=\dmult(r\inv E).
$$

Since any  multiplicative density is clearly subadditive,
any finite partition of $\N$ has at least one cell with positive
multiplicative density. Hence,  Theorem~\ref{th:partition-regular2} follows
from the following stronger result:
\begin{theorem}[Density regularity]
\label{th:density-regular}
 Let $\ell_0,\ldots, \ell_4\in \Z$ be admissible. Then  every  set $E\subset\N$ with $\dmult(E)>0$
 contains  distinct $x,y\in \N$ of the
form \eqref{E:xy}.
\end{theorem}
 Since
 $$
 \dmult(x\inv E\cap y\inv E)=
\dmult\big(\bigl\{k\in \N\colon \{kx,ky\}\subset  E\bigr\}\big).
$$
 in order to prove Theorem~\ref{th:density-regular} it suffices to prove the following result:
\begin{proposition}
\label{fo:density-regular2}
Let $\ell_0,\ldots, \ell_4\in \Z$ be admissible.  Then  every   $E\subset\N$ with $\dmult(E)>0$
 contains  distinct $x,y\in \N$ of the
form \eqref{E:xy} such that
\begin{equation}
\label{eq:demultk}
\dmult(x\inv  E\cap y\inv E)>0.
\end{equation}
\end{proposition}
In the next two sections  we are going to reinterpret Proposition~\ref{fo:density-regular2}
as a more convenient to prove analytic statement.
\subsection{Integral formulation}\label{subsec:Bochner} We first  reformulate  Proposition~\ref{fo:density-regular2}
using an integral representation result of positive definite sequences on $\Q^+$.
Recall that a function $ f \colon\N\to\C$ is completely multiplicative if $ f (x y)= f (x) f (y)$ for every $x,y\in\N$.
\begin{definition}
 We denote by $\CCM$  the set of completely multiplicative functions of modulus exactly $1$.
\end{definition}
 A completely multiplicative function is uniquely determined by its values on the primes.
Every $ f \in\CCM$ can be extended to a multiplicative function on $\Q^+$, also denoted by $ f $, by
letting for every $x,y\in \N$
\begin{equation}
\label{eq:duality}
 f (xy\inv):= f (x)\overline f (y).
\end{equation}
Endowed with the pointwise multiplication and the topology of pointwise convergence, the family $\CCM$ of completely multiplicative functions is a compact (metrizable) Abelian group, with unit element the constant function $\one$. This group is the dual group of $\Q^+$, the duality being given by~\eqref{eq:duality}.

Let $E\subset\N$ be a set with  $\dmult(E)>0$.
There exists  a sequence $(M_j)$ of integers, tending to infinity, such that
\begin{gather}
\label{eq:M_j1}
\lim_{j\to+\infty}
\frac{|E\cap\Phi_{M_j}|}{|\Phi_{M_j}|}=\dmult(E); \\
\label{eq:M_j2}
\text{ and }\  \rho(r):=\lim_{j\to+\infty}\frac {|E\cap (r\inv E)\cap\Phi_{M_j}|}{|\Phi_{M_j}|}
\ \text{ exists for every }\ r\in \Q^+.
\end{gather}
 Then the function $\rho\colon\Q^+\to\C$ is positive definite, that is, for every $n\in\N$, all $r_1,\dots,r_n\in\Q^+$ and all $\lambda_1,\dots,\lambda_n\in\C$ we have
$$
\sum_{i,j=1}^n\lambda_i\overline{\lambda_j}\,\rho(r_i\,r_j\inv)\geq 0.
$$
By Bochner's theorem, there exists a unique  positive finite measure  $\nu$ on the compact Abelian group $\CCM$ with a  Fourier-Stieltjes transform $\wh \nu$ equal to the function $\rho$. 
This means that
\begin{equation}
\label{eq:fourier-nu}
\text{for every }r\in\Q^+,\
\int_{\CCM} f (r)\,d\nu( f )=\wh\nu(r)
=\rho(r)
=\lim_{j\to+\infty}
\frac{|E\cap (r\inv E)\cap\Phi_{M_j}|}{|\Phi_{M_j}|}.
\end{equation}
We collect the properties of the measure $\nu$  used in the sequel.
\begin{claim*} Let the set $E$ and the measure $\nu$ be as before and $\delta:=\dmult(E)$. Then
\label{cl:properties-nu}
\begin{gather}
\label{eq:defpos-nuf}
\int_\CCM  f (x)\, \overline f (y)\,d\nu( f )\geq 0\ \text{ for every }x,y\in\N;\\
\label{eq:nu-1f}
\nu(\{\one\})\geq\delta^2.
\end{gather}
\end{claim*}
\begin{proof}[Proof of the Claim]
Property~\eqref{eq:defpos-nuf} follows from~\eqref{eq:fourier-nu} with $r:=xy\inv$.
The proof of \eqref{eq:nu-1f} is classical but we give it for completeness.

For $x\in\N$ let $f(x):=\one_E(x)-\delta$. The averages on $\Phi_{M_j}$ of the function $f$  tend to $0$ as $j\to+\infty$, and it follows from~\eqref{eq:Folner} that
$$
\text{for every }r\in\Q^+,\ \lim_{j\to+\infty} \E_{x\in\Phi_{M_j}}f(rx)=0.
$$
Thus,
$$
\wh\nu(r)=\lim_{j\to+\infty}\frac {|E\cap r\inv E\cap\Phi_{M_j}|}{|\Phi_{M_j}|}=
\lim_{j\to+\infty}\E_{x\in\Phi_{M_j}}\one_E(x)\one_E(rx)
=\delta^2+\psi(r),
$$
where
$$
\psi(r):= \lim_{j\to+\infty}\E_{x\in\Phi_{M_j}}f(x)f(rx)
  $$
  and the limit exists for every $r\in \Q^+$ by $\eqref{eq:M_j2}$.
  The function $\psi\colon\Q^+\to\C$ is positive definite and by Bochner's theorem $\psi=\wh\sigma$ for some positive finite measure $\sigma$ on $\CCM$.  Since $\one$ is the unit element of the group $\CCM$ and  $(\Phi_N)$ is a F\o lner sequence in $\Q^+$, the averages of $\wh\nu(r)$ on $(\Phi_N)$ converge to $\nu(\{\one\})$ and the averages of $\wh\sigma(r)$ to  $\sigma(\{\one\})$. Therefore
$$
\nu(\{\one\})= \lim_{M\to+\infty}\E_{r\in\Phi_M}\wh\nu(r)=
\delta^2+  \lim_{M\to+\infty}\E_{r\in\Phi_M}\wh\sigma(r)=\delta^2+\sigma(\{\one\})\geq\delta^2.
\qed
$$
\renewcommand{\qed}{}
\end{proof}

In order to show Proposition~\ref{fo:density-regular2} it suffices to prove the following:
\begin{proposition}[Analytic formulation]
\label{fo:density-regular4}
Let  $\ell_1, \ldots, \ell_4$  be distinct integers and suppose that  $\min\{\ell_1,\dots,\ell_4\}=0$.
Let $\nu$ be a  probability measure on $\CCM$ that satisfies
Properties
 \eqref{eq:defpos-nuf} and~\eqref{eq:nu-1f}.
Then there exist $m,n\in \Z$ such that    $(m+\ell_1n)(m+\ell_2n)$ and
$(m+\ell_3n)(m+\ell_4n)$ are positive, distinct integers,   and we have
\begin{equation}\label{E:1to4}
\int_\CCM  f (m+\ell_1n)\cdot  f (m+\ell_2n)\cdot \overline{ f }(m+\ell_3n)\cdot \overline f (m+\ell_4n)\,d\nu( f )>0.
\end{equation}
\end{proposition}
We show that Proposition~\ref{fo:density-regular4} implies  Proposition~\ref{fo:density-regular2}.
 Without loss of generality we  can assume that the measure $\nu$ defined in \eqref{eq:fourier-nu} is a probability measure. Let  $m,n\in \Z$ satisfy \eqref{E:1to4}. Letting  $x:=(m+\ell_1n)(m+\ell_2n)$ and $y:=(x+\ell_3n)(x+\ell_4n)$ and using ~\eqref{eq:fourier-nu} we get
\begin{multline*}
\dmult(x\inv E\cap y\inv E)=\limsup_{M\to+\infty} \frac{|x\inv yE\cap E\cap \Phi_M|}{|\Phi_M|}\geq\lim_{j\to+\infty}\frac{|x\inv yE\cap E\cap \Phi_{M_j}|}{|\Phi_{M_j}|}\\
=\wh\nu(xy\inv)
=\int_\CCM  f (xy\inv)\,d\nu( f )
=\int_\CCM  f (x)\, \overline{ f }(y)\,d\nu( f )>0.
\end{multline*}

This proves Proposition~\ref{fo:density-regular2} in the case where the integers $\ell_1,\ldots, \ell_4$ are  distinct
 and   $\min\{\ell_1,\ldots,\ell_4\}=0$.

   Let $\underline \ell:=\min\{\ell_1,\dots,\ell_4\}\neq 0$, by replacing $\ell_j$ with $\ell_j-\underline \ell$ for $j=1,\dots,4$ , and making the change of variables
$m\mapsto m-\underline \ell n$, we reduce matters to the case that $\underline\ell=0$.

 It remains to consider
the degenerate cases  where $\ell_1$ or $\ell_2$ is equal to $\ell_3$ or $\ell_4$. Suppose that $\ell_1=\ell_3$, the other cases are similar.   Since $\ell_0,\dots,\ell_4$ are admissible, $\ell_2\neq \ell_4$. We can assume that
 $\ell_2<\ell_4$, the other case is similar. As before,
 we see that it suffices to show that
 there exist $m,n\in \Z$ such that the integers $m+\ell_2n$ and $m+\ell_4n$ are positive and satisfy
 $$
\int_\CCM  f (m+\ell_2n) \cdot \overline f (m+\ell_4n)\,d\nu( f )>0.
$$
 After making the change of variables $m\mapsto m-\ell_2n$  we see that it suffices to
  show that
 there exist $m,n\in \N$ such that
 $$
\int_\CCM  f (m) \cdot \overline f (m+(\ell_4-\ell_2)n)\,d\nu( f )>0.
$$
   Since the averages of $\widehat\nu$ on the  F\o lner sequence $(\Phi_M)$ converge as $M\to+\infty$  to $\nu(\{\one\})$
  which is  positive by \eqref{eq:nu-1f}, there exists $n_0\in\N$ such that
 $\widehat\nu(n_0+1)>0$. Taking $m:=\ell_4-\ell_2$ and $n:=n_0$ we have
$$
\int_\CCM  f (m) \cdot \overline f (m+(\ell_4-\ell_2)n)\,d\nu( f )=
\widehat\nu(n_0+1)>0.
$$

\begin{convention} In the rest of the proof we assume that $\ell_1,\ldots, \ell_4$ are  distinct integers with $\min\{\ell_1,\dots,\ell_4\}=0$.
 We let
$$
\ell:=\ell_1+\ell_2+\ell_3+\ell_4.
$$
For every $N\in\N$, we denote by $\wt N$ the smallest  prime in the interval $[2\ell N, 4\ell N]$.
 As usual, for every function $\phi$ on $\N$, we denote by  $\phi_N$  the function $\one_{[N]}\,\phi$, considered as a function on $\Z_\tN$.
\end{convention}

\subsection{Final analytic formulation}
In order to prove Proposition~\ref{fo:density-regular4},  it suffices to establish the stronger fact that there are
 ``many'' $m,n\in \N$  such that the integral in this statement is positive.
\begin{proposition}[Averaged analytic formulation]
\label{fo:density-regular6}
Let $\ell_1, \ldots, \ell_4\in \Z$ be distinct  with $\min\{\ell_1,\ldots, \ell_4\}=0$
and let $\nu$ be a probability measure on $\CCM$
  that satisfies
Properties
\eqref{eq:defpos-nuf} and~\eqref{eq:nu-1f}. Then
\begin{equation}
\label{eq:average-integral1}
\liminf_{N\to+\infty}  \int_\CCM\E_{(m,n)\in\Theta_N}
 f (m+\ell_1n)\cdot f (m+\ell_2n)\cdot\overline f (m+\ell_3n)\cdot\overline f (m+\ell_4n)\ d\nu( f )>0
\end{equation}
where $\Theta_N:=\{(m,n)\in[N]\times[N]\colon 1\leq m+\ell_in\leq N\text{ for }i=1,2,3, 4
\}$.
\end{proposition}
In order to show that Proposition~\ref{fo:density-regular6} implies Proposition~\ref{fo:density-regular4}
we  remark   that for $N\in\N$ sufficiently large we have $|\Theta_N|\geq c_1N^2$ and  the cardinality of the set of pairs $(m,n)\in\Theta_N$ that
satisfy $(m+\ell_1n)(m+\ell_2n)=(m+\ell_3n)(m+\ell_4n)$ is bounded by $C_1N$ for some
constants $c_1$ and $C_1$ that depend only on  $\ell$. Therefore,
Property~\eqref{eq:average-integral1} implies that there exist $m,n\in\N$ such that
$(m+\ell_1n)(m+\ell_2n)\neq(m+\ell_3n)(m+\ell_4n)$ and  $\int_\CCM
 f (m+\ell_1n) f (m+\ell_2n)\,\overline f (m+\ell_3n)\,\overline f (m+\ell_4n)\,d\nu( f )>0$.
Hence, the conclusion of  Proposition~\ref{fo:density-regular4} holds.

\begin{remark}
An alternate (and arguably more natural) way to proceed
 is to replace the additive averages
in Proposition~\ref{fo:density-regular6} with multiplicative ones. Upon doing this,
one is required to analyze averages of the form
$$
 \E_{m,n\in
\Psi_N}
 f \big((m+\ell_1n) (m+\ell_2n)\big)\,\overline f \big((m+\ell_3n)(m+\ell_4n)\big),
$$
where $(\Psi_N)_{N\in\N}$ is a multiplicative F\o lner sequence in $\N$ and $ f  \in \CCM$. Unfortunately, we were not able to
prove anything useful for these multiplicative averages,  although one suspects that a positivity property similar to the one in  \eqref{eq:average-integral1} may hold.
\end{remark}
Next, for technical reasons, we   recast the previous proposition  as a positivity property
involving averages over the cyclic groups $\Z_{\tN}$. This is going to be the final form
of the analytic statement that we aim to prove.
\begin{proposition}[Final analytic formulation]
\label{th:ergo2b}
Let  $\ell_1, \ldots, \ell_4\in \Z$ be distinct  and suppose that  $\min\{\ell_1,\dots,\ell_4\}=0$. Let $\delta>0$  and  $\nu$ be a  probability measure on $\CCM$, such that
\begin{enumerate}
\item
\label{it:nu-1}
$\nu(\{\one\})\geq\delta^2$;
\item
\label{it:defpos-nu}
$\displaystyle\int_\CCM  f (x)\, \overline f (y)\,d\nu( f )\geq 0$ \ for every  $x,y\in\N$.
\end{enumerate}
Then we have
\begin{equation}
\label{eq:average-integral1b}
\liminf_{N\to+\infty}  \int_\CCM\E_{m,n\in\Z_\tN}
\one_{[N]}(n)
 f_ N(m+\ell_1n) f_ N(m+\ell_2n)\overline f_ N(m+\ell_3n)\overline f_ N(m+\ell_4n)\,d\nu( f )>0,
\end{equation}
where in the above average  the expressions $m+\ell_in$ can be considered as elements of  $\Z$ or $\Z_\tN$ without affecting the value of the average.
\end{proposition}

We verify that Proposition~\ref{th:ergo2b} implies Proposition~\ref{fo:density-regular6}.
Using the definition of the set $\Theta_N$ given in Proposition~\ref{fo:density-regular6}, we can rewrite the averages that appear in the statement of Proposition~\ref{fo:density-regular6} as follows:
\begin{multline}\label{E:identity}
\E_{(m,n)\in\Theta_N}
 f (m+\ell_1n)\cdot  f (m+\ell_2n)\cdot \overline f (m+\ell_3n)\cdot\overline f (m+\ell_4n)=\\
 \frac{\wt N^2}{|\Theta_N|}\;\E_{m,n\in[\wt N]} \one_{[N]}(n)\cdot
 f_ N(m+\ell_1 n)\cdot f_ N(m+\ell_2n)\cdot\overline f_ N(m+\ell_3n)\cdot\overline f_ N(m+\ell_4n).
\end{multline}
To prove this equality, we remark  that since $\min\{\ell_1,\ldots, \ell_4\}=0$,  if $m,n$ are such that $m\in[\tN]$, $n\in[N]$, and
$m+\ell_jn \bmod \tN\in[N]$ for $j=1,\dots,4$, then
$m\in[N]$. Thus  $1\leq m+\ell_jn\leq (\ell+1)N<\tN$, hence $m+\ell_jn=m+\ell_jn \bmod \tN\in[N]$ for  $j=1,2,3,4$ and
every $(m,n)\in\Theta_N$. The sets of pairs $(m,n)$ taken in account in the two averages are identical, and
the value of the last  expression remains unchanged if we replace each term $m+\ell_in$ by $m+\ell_in \bmod{\wt N}$.

Using identity \eqref{E:identity} and the estimate $cN^2\leq |\Theta_N|\leq N^2$ which holds for some positive constant $c$ that depends only on $\ell$, we get the asserted implication.

\subsection{A positivity property}
We  derive now a positivity property that will be used in the
proof of Proposition~\ref{th:ergo2b} in the next subsection. Here we make
essential use of the positivity Property~\eqref{it:defpos-nu} of the measure $\nu$ given in
Proposition~\ref{th:ergo2b}.

\begin{lemma}[Hidden non-negativity]\label{L:HiddenPositivity}
Let $\nu$ be a positive finite measure on  $\CCM$ that  satisfies Property~\eqref{it:defpos-nu} of Proposition~\ref{th:ergo2b}.
 Let $\psi$ be a
non-negative function defined on $\Z_\tN$. Then
 $$
\int_\CCM ( f_ N*\psi)(n_1)\cdot ( f_ N*\psi)(n_2)\cdot
 (\overline f_ N*\psi)(n_3)\cdot (\overline f_ N*\psi)(n_4) \ d \nu( f )\geq 0
$$
for every $n_1,n_2,n_3,n_4\in\Z_\tN $, where the convolution product is taken on $\Z_\tN$.
\end{lemma}
\begin{proof}
The  convolution product $ f_ N*\psi$ is defined on the
group $\Z_\tN$ by the formula
$$
( f_ N*\psi)(n)=\E_{k\in\Z_\tN}\psi(n-k)\cdot f_ N(k).
$$
It follows that for every $n\in[\wt N]$ there exists a sequence
$(a_n(k))_{k\in\Z_\tN}$ of non-negative numbers that are  independent of $ f $,
such that for every $ f \in\CCM$ we have
$$
( f_ N*\psi)(n) =\sum_{k\in\Z_\tN}a_n(k)\,  f (k).
$$
The left hand side of the expression in the statement is thus equal to
\begin{multline*}
\sum_{k_1,k_2,k_3,k_4\in\Z_\tN}\prod_{i=1}^4 a_{n_i}(k_i) \int_\CCM
 f (k_1)\cdot  f (k_2)\cdot
 \overline f (k_3)\cdot \overline f (k_4) \, d \nu( f )=\\
 \sum_{k_1,k_2,k_3,k_4\in \Z_\tN}\prod_{i=1}^4 a_{n_i}(k_i)
\int_\CCM
 f (k_1k_2)\,\overline f (k_3k_4) \, d \nu( f )\geq 0
\end{multline*}
by Property~\eqref{it:defpos-nu}  of Proposition~\ref{th:ergo2b}.
\end{proof}

\subsection{Estimates involving Gowers norms}
Next we establish
an elementary estimate that will be  crucial in the sequel.

\begin{lemma}[Uniformity estimates]
\label{L:UnifromityEstimates2}
Let $ s\geq 3$, $\ell_1,\dots,\ell_{s}\in \Z$ be distinct, and let $\ell:=|\ell_1|+\dots+|\ell_s|$.    Then
there exists $C:=C(\ell)$   such for every $N\in\N$ and  all functions $a_j\colon\Z_\tN\to\C$, $j=1,\dots,s$,  with
$|a_j|\leq 1$, we have
 $$
\big|\E_{m,n\in\Z_\tN}\one_{[N]}(n)\cdot \prod_{j=1}^sa_j(m+\ell_jn)
\big| \leq C
\min_{j=1,\ldots,  s}(\norm{a_j}_{U^{s-1}(\Z_\tN)})^{1/2}+\frac{2}{\tN}
$$
where $\tN$ is the smallest prime that is  greater than $2\ell N$.
\end{lemma}
\begin{proof}
We first  reduce matters to estimating a similar average that does not contain
 the term $\one_{[N]}(n)$.
Let $r$ be an integer that will be specified later and satisfies
$0<r< N/2$. We define the ``trapezoid function'' $\phi$ on $\Z_N$ so
that $\phi(0)=0$, $\phi$ increases linearly from $0$ to $1$ on the
interval $[0,r]$, $\phi(n)=1$ for $r\leq n\leq N-r$, $\phi$
decreases linearly from
 $1$ to $0$ on $[N-r,N]$, and $\phi(n)=0$ for $N<n<\wt N$.

The absolute value of the difference between the average in the statement and
$$
\E_{m,n\in\Z_{\tN}}\phi(n)\cdot \prod_{j=1}^s a_j(m+\ell_jn)
$$
is bounded by $2r/\wt N$.
Moreover, it is classical that (the argument is sketched in the proof of Lemma~\ref{lem:UsLN} in the Appendix)
$$
\norm{\wh\phi}_{l^1(\Z_{\tN})}\leq \frac{2N}r\leq \frac{\wt{N}}r
$$
 and thus
$$
\Bigl|\E_{m,n\in\Z_\tN}\phi(n)\cdot
 \prod_{j=1}^s a_j(m+\ell_jn)\Bigr|\leq
 \frac{\wt{N}}r\,\max_{\xi\in\Z_\tN}
\Bigl|\E_{m,n\in\Z_\tN}\e(n\xi/\tN)\cdot
\prod_{j=1}^s a_j(m+\ell_jn)\Bigr|.
$$

 Since $\ell_1\neq\ell_2$ and $\tN>\ell$ we have $\ell_1-\ell_2\neq 0 \bmod\tN$ and there exist $\ell^*\in\Z_\tN$ such that $\ell^*(\ell_1-\ell_2)=1 \bmod\tN$.
 Upon replacing $a_1(n)$ with
$a_1(n)\e(-\ell^*n\xi/\tN)$ and $a_2(n)$ with
$a_2(n)\e(\ell^*n\xi/\tN)$,  the $U^{s-1}$-norm of all sequences remains unchanged,  and the term $\e(n\xi/\tN)$ disappears. We are thus left with estimating
the average
$$
\big|\E_{m,n\in\Z_\tN} \prod_{j=1}^s a_j(m+\ell_jn)\big|.
$$

  Since $\tN>2\ell$ the numbers  $\ell_1,\ldots, \ell_s$ are  distinct
   as elements of $\Z_\tN$. Using this and the fact that  $\tN$ is a prime, it is possible to show
    by an iterative use of the Cauchy-Schwarz inequality  (see for example \cite[Theorem 3.1]{T06})   that   the last average is bounded
 by
$$ U:=\min_{1\leq j\leq s}\norm{a_j}_{U^{s-1}(\Z_\tN)}.
$$

Combining the preceding estimates, we get that  the average in the
statement is bounded by
$$
\frac {2r}{\wt N} +
\frac{2 \wt{N}}r U.
$$
Assuming that $U\neq 0$ and choosing   $r:                                                                   = \lfloor \sqrt{U}\tN/(8\ell)\rfloor+1$ (then $r\leq
\tN/(8\ell)\leq N/2$)
gives  the announced bound.
\end{proof}

\subsection{Proof of Proposition~\ref{th:ergo2b}}
\label{SS:assuming}
  We start with a brief sketch of our
 proof strategy. Roughly speaking, Theorem~\ref{T:DecompositionII}
enables us to decompose  the restriction of an arbitrary
multiplicative function on a finite interval  into three terms, a
close to  periodic term, a ``very uniform'' term, and an error term.
In the course of the proof of Proposition~\ref{th:ergo2b} we study these
three terms separately. The order of the different steps is
important as well as the precise properties of the decomposition.

 First, we show that the uniform term has a negligible contribution
in evaluating the averages in \eqref{eq:average-integral1b}. To do
this we use  the  uniformity estimates established in
Proposition~\ref{P:semest}. It is for this part of the proof
that it is very important to work with patterns that factor into
products of linear forms in two variables, otherwise we have no
 way of controlling the corresponding averages
 by Gowers uniformity norms. At this
point, the error term is shown to have negligible contribution, and
thus can be ignored.
   Lastly, the
structured term $ f_ \st$  is dealt by restricting the variable $n$
to a suitable sub-progression where each function $ f_ \st$ gives
approximately the same value to all four linear
forms;\footnote{\label{foot3}This  coincidence of values is very
important, not having it is a key technical obstruction for
 handling equations like $x^2+y^2=n^2$. Restricting the range
of both variables $m$ and $n$ does not seem to help either, as this
creates problems with  controlling the error term in the
decomposition.}
 it  then becomes possible to  establish the asserted positivity.
 In fact, the step where we restrict to a
sub-progression  is  rather delicate, as it has to take place
before the component $ f_ \er$ is eliminated (this explains also why
we do not restrict both variables $m$ and $n$ to a sub-progression), and in addition one
has to guarantee that the terms left out are non-negative, a
property that follows from
Lemma~\ref{L:HiddenPositivity}.

 We  now enter the main body of the proof.
Recall that $\ell_1,\ldots, \ell_4\in \Z$ are fixed and distinct and that $\ell=|\ell_1|+\cdots+|\ell_4|$. We stress also  that
in  this proof the quantities $m+\ell_in$ are computed in $\Z_\tN$,
that is, modulo $\wt N$.

Let $\nu$ be a positive finite measure on $\CCM$ satisfying the
Properties~\eqref{it:nu-1} and~\eqref{it:defpos-nu} of Proposition~\ref{th:ergo2b} and let $\delta>0$ be as in~\eqref{it:nu-1}.
We let
$$
\ve:=c_1\delta^2\ \quad \text{ and }\quad
 F(x,y,z):=C_1^2\,
\frac{x^2y^2}{z^4},
$$
where $c_1$ and $C_1$ are positive constants that will be specified
later, what is important is that  they depend only on $\ell$. Our goal is
for all large values of $N\in\N$ (how large will depend only on $\delta$) to bound from below the average
$$
 A(N):= \int_\CCM\E_{m,n\in\Z_\tN}\one_{[N]}(n)\,   f_ N(m+\ell_1n)\,  f_ N(m+\ell_2n) \,
\overline f_ N(m+\ell_3n) \, \overline f_ N(m+\ell_4n)\ d\nu( f ).
$$

We start by applying
Theorem~\ref{T:DecompositionII} for the $U^3$-norms, taking as input the
measure $\nu$, the number $\ve$, and the function $F$ defined
above. Let
$$
 Q:=Q(F,N, \ve, \nu)=Q(N, \delta, \nu),\quad  R:=R(F,N, \ve,  \nu)=R(N,\delta,\nu)
$$ be
the numbers provided by Theorem~\ref{T:DecompositionII}. We
recall  that $Q$ and $R$ are bounded by a constant that depends only
on $\delta$. From this point on we assume that $N\in\N$ is sufficiently
large, depending only on $\delta$,  so that the conclusions of
Theorem~\ref{T:DecompositionII} hold. 
For
 $ f  \in \CCM$, we have a decomposition
$$
 f_ N(n)= f_ {N,\st}(n)+ f_ {N,\un}(n)+ f_ {N,\er}(n), \quad n\in \Z_\tN,
$$
for the decomposition that satisfies Properties~\eqref{it:decomU31}--\eqref{it:decomU34} of Theorem~\ref{T:DecompositionII}.

Next, we  use the uniformity estimates of
Lemma~\ref{L:UnifromityEstimates2} for $s=4$  in order to eliminate the uniform
component $ f_ \un$ from the average $A(N)$. We let
$$
 f_ {s,e}:= f_ {N,\st}+ f_ {N,\er}
$$ and
 $$
A_1(N):=\int_\CCM  \E_{m,n\in\Z_\tN} \one_{[N]}(n) f_ {s,e}(m+\ell_1n)
 f_ {s,e}(m+\ell_2n)  \overline f_ {s,e}(m+\ell_3n)
\overline f_ {s,e}(m+\ell_4n)\ d\nu( f ). $$
 Using
Lemma~\ref{L:UnifromityEstimates2}, Property~\eqref{it:decomU33} of
Theorem~\ref{T:DecompositionII}, and the estimates  $| f_ N(n)|\leq 1$, $| f_ {s,e}(n)|\leq 1$ for every $n\in\Z_\tN$,
we get that
\begin{equation}\label{E:a1}
|A(N)-A_1(N)|\leq \frac{4\, C_2}{F(Q,R,\ve)^{\frac{1}{2}}}+\frac{8}{\tN}
\end{equation}
where $C_2$ is the constant provided by
Lemma~\ref{L:UnifromityEstimates2} and depends only on $\ell$.

 Next, we  eliminate the error term $ f_ \er$. But before doing this, it is
  important to first restrict the range of $n$ to a suitable sub-progression; the utility of
  this  maneuver will be clear on our next step when we estimate the contribution of the leftover term
  $ f_ \st$. We stress that we cannot postpone this restriction on the range of $n$ until after
  the term $ f_ \er$ is eliminated, if we did this  the contribution of the term $ f_ \er$ would
   swamp the positive lower bound we get from the term $ f_ \st$.
 We let
 \begin{equation}
\label{eq:def-epsilon}
\eta:=\frac{\ve}{QR}.
\end{equation}
 By Property~\eqref{it:decomU31} of Theorem~\ref{T:DecompositionII}
and Lemma~\ref{L:HiddenPositivity},  we have
the positivity property
\begin{equation}\label{E:positivea}
\int_\CCM  f_{s,e}(n_1)\cdot  f_ {s,e}(n_2)\cdot
  \overline{f}_{s,e}(n_3)\cdot  \overline{f}_{s,e}(n_4) \ d \nu( f )\geq 0
\end{equation}
for every $n_1,n_2,n_3,n_4\in\Z_\tN $,

 Note that
the integers $Qk$, $1\leq k\leq \eta N$, are distinct elements of
the interval $[N]$.  It follows from \eqref{E:positivea} that
\begin{multline*}
\sum_{m,n\in\Z_\tN}\int_\CCM
\one_{[N]}(n)\, f_ {s,e}(m)\, f_ {s,e}(m+\ell_1n)\,
\overline{ f }_{s,e}(m+\ell_2n)\, \overline{ f }_{s,e}(m+\ell_3n)\,d\nu( f )\geq
\\
\sum_{m\in\Z_\tN}\sum_{k=1}^{\lfloor \eta N\rfloor}
\int_\CCM
 f_ {s,e}(m+\ell_1Qk) \, f_ {s,e}(m+\ell_2Qk)\,
\overline{ f }_{s,e}(m+\ell_3Qk)\, \overline{ f }_{s,e}(m+\ell_4Qk)\,d\nu( f ).
\end{multline*}
Therefore, we have
\begin{equation}
\label{eq:A2} A_1(N)\geq \frac{\lfloor\eta N\rfloor}{\tN}A_2(N)\geq
\frac\eta{40\,\ell}A_2(N) = \ve\, \frac{1}{40\,\ell QR}\,A_2(N)
\end{equation}
where
\begin{multline*}
A_2(N):= \\
\int_\CM  \E_{m\in\Z_\tN}\,\E_{k\in [\lfloor \eta N\rfloor]}\,
 f_ {s,e}(m+\ell_1Qk)\,
 f_ {s,e}(m+\ell_2Qk) \,  \overline f_ {s,e}(m+\ell_3Qk)\,
\overline f_ {s,e}(m+\ell_4Qk)\ d\nu( f ).
\end{multline*}
We let
\begin{multline*}
\label{eq:defA3}
A_3(N):= \\ \int_\CCM\E_{m\in\Z_\tN}\E_{k\in[\lfloor \eta N\rfloor]}
 f_ {N,\st}(m+\ell_1Qk)  f_ {N,\st}(m+\ell_2Qk)
\overline{ f }_{N,\st}(m+\ell_3Qk) \overline{ f }_{N,\st}(m+\ell_4Qk)d\nu( f ).
\end{multline*}
Since for every $n\in\Z_\tN$ we have $| f_ {N,\st}(n)|\leq 1$, and since
$| f_ {s,e}(n)|=| f_ {N,\st}(n)+ f_ {N,\er}(n)|\leq 1$, by
Property~\eqref{it:decomU34} of
 Theorem~\ref{T:DecompositionII} we deduce that
\begin{equation}
\label{eq:A3}
|A_2(N)-A_3(N)|\leq 4\,\int_\CCM\E_{m\in\Z_\tN} | f_ {N,\er}(m)|\,d\nu( f )<4\ve.
\end{equation}

Next, we study the term $A_3(N)$.  We utilize Property~\eqref{it:decompU32} of
Theorem~\ref{T:DecompositionII}, namely
$$
| f_ {N,\st}(n+Q)- f_ {N,\st}(n)|\leq \frac{R}{\tN} \quad \text{ for  every } \ n\in\Z_\tN.
$$
We get for $m\in \Z_\tN$, $1\leq k\leq\eta N$,  and for $i=1,2,3,4$, that
$$
| f_ {N,\st}(m+\ell_i Qk)- f_ {N,\st}(m)|\leq\ \ell_ik\,\frac {R}\tN\leq \ell
\eta N\,\frac {R}\tN \leq \frac \ve Q,
$$
where the last estimate follows from \eqref{eq:def-epsilon} and the estimate $\tN\geq \ell N$.
Using this estimate in conjunction with the definition~ of $A_3(N)$, we get
$$
A_3(N)\geq \int_\CCM\E_{m\in\Z_\tN}| f_ {N,\st}(m)|^4\,d\nu( f )
-\frac{3\ve}Q.
$$

Recall that  $\one$ denotes  the multiplicative function that is identically
 equal  to $1$.
By Property~\eqref{it:nu-1} of Proposition~\ref{th:ergo2b} we have
$\nu(\{\one\})\geq\delta^2$.
 Using this,
we deduce that
$$
\int_\CCM\E_{m\in\Z_\tN}| f_ {N,\st}(m)|^4\,d\nu( f )\geq
\nu(\{\one\}) \cdot\E_{m\in\Z_\tN} |\one_{N,\st}(m)|^4
\geq \delta^2\, \bigl|\E_{m\in\Z_\tN} \one_{N,\st}(m)|^4.
$$
Since $\one_{N,\st}=\one_N*\psi$ for some kernel $\psi$  on $\Z_\tN$  and $\tN\leq 4\ell N$, we have
$$\E_{m\in\Z_\tN}\one_{N,\st}(m)=
\E_{m\in\Z_\tN}\E_{k\in\Z_\tN}\one_N(k)\psi(m-k)=
\E_{k\in\Z_\tN}\one_N(k) =
\frac N{\tN}\geq \frac 1{4\,\ell}.
$$
Combining the above we get
\begin{equation}
\label{eq:A3b} A_3(N)\geq  \frac
{\delta^2}{4^4\,\ell^4}-\frac{3\ve }Q.
\end{equation}

Putting~\eqref{E:a1},  \eqref{eq:A2}, \eqref{eq:A3}, and
\eqref{eq:A3b} together,   we get
$$A(N)
\geq\ve\, \frac{1}{40\,\ell QR}\,\Bigl(\frac
{\delta^2}{4^4\,\ell^4}-7\ve \Bigr)
-\frac{4C_2}{F(Q,R,\ve)^{\frac{1}{2}}}-\frac{8}{\tN}.
$$
Recall that $\ve=c_1\delta^2$, for some positive constant $c_1$ that we
left unspecified until now.  We choose $c_1<1$, depending only on
$\ell$, so that
$$
\frac 1{40\,\ell} \Bigl(\frac
{\delta^2}{4^4\,\ell^4}-7\ve\Bigr) \geq c_2 \delta^2
$$
for some positive constant $c_2$ that depends only on
$\ell$. Then we have
$$
A(N)\geq\delta^2\,\frac{c_2 \ve }{QR}-\frac{4C_2}{F(Q,R,\ve)^{\frac{1}{2}}}-\frac{8}{\tN}.
$$
Recall that $$ F(Q,R,\ve)= C_1^2\frac{Q^2R^2}{\ve^4}
$$
where $C_1$ was not determined until this point. We choose
$$
C_1:=\frac{8c_1C_2 }{c_2 }
$$
and upon recalling that $\ve=c_1 \delta^2$    we get
$$
A(N)+ \frac{8}{\tN}\geq\delta^2\,\frac{c_2\ve }{QR}-C_2\, \frac{4\ve^2}{C_1 QR}=\frac {c_2\delta^2\ve}{2\,QR}=
\frac {c_1 c_2 \delta^4}{2\,QR}>0.
$$
Recall that $Q$ and $R$ are bounded by a constant that depends only
on $\delta$. Hence, $A(N)$ is greater than a positive constant that
depends only on $\delta$, and in particular is independent of $N$,
provided that $N$ is sufficiently large, depending only on $\delta$,
as indicated above. This completes the proof of
Proposition~\ref{th:ergo2b}. \qed

\subsection{Proof of  Theorem~\ref{th:partition-regular3}}
The proof of Theorem~\ref{th:partition-regular3} goes along the lines of
Theorem~\ref{th:partition-regular2} with small changes only.

As a first step we reduce matters to the case where the coefficient of $m$ in all linear forms is $1$.
For $i=1,\ldots, s$, let the linear forms be given by
$L_{1,i}(m,n):=\kappa_im+\lambda_in$, $L_{2,i}(m,n):=\kappa_i'm+\lambda_i'n$,  where $\kappa_i,\kappa_i',\lambda_i,\lambda_i'\in \Z$.
Let
$
\ell_0:=\prod_{i=1}^s\kappa_i=\prod_{i=1}^s\kappa_i'
$
where the second equality follows by our assumption. We also have $\ell_0\neq 0$ by assumption and
we can   assume that $\ell_0>0$, the other case can be treated similarly. Inserting $\ell_0n$ in place of
$n$ and factoring out the coefficients of $m$ we reduce    to the case where $\kappa_i=\kappa'_i=1$ for $i=1,\ldots,s$. Our assumption gives that
$\{\lambda_1,\dots,\lambda_s\}\neq \{\lambda'_1,\dots,\lambda'_s\}$.

Theorem~\ref{th:partition-regular3} can be deduced from an analytic statement completely similar
to Proposition~\ref{th:ergo2b}. By an induction on $s$ we can reduce to the case that the integers
$\lambda_1,\dots,\lambda_s$ and $\lambda'_1,\dots,\lambda'_s$ are distinct and furthermore we can assume that the smallest one is equal to $0$.
 The rest of the argument is identical to the proof of Proposition~\ref{th:ergo2b}  given in this section;
the only difference is that in place of Theorem~\ref{T:DecompositionII} for the $U^3$-norm we use
 the same result for the $U^{2s-1}$-norm.
\appendix

\section{Elementary facts about Gowers norms}
In this section we gather some elementary facts about the $U^s$-norms that
we use throughout the main body of the article.
\subsection{Gowers norms and restriction to subintervals}
Our first result shows  that if the $U^s(\Z_N)$-norm of a function is sufficiently small,
then its restriction to an arbitrary subinterval of $[N]$ is small.

\begin{lemma}
\label{lem:UsLN}
Let  $s\geq 2$ be an integer and  $\ve>0$. There exists $\delta:=\delta(s,\ve)>0$ and $N_0:=N_0(s,\ve)>0$
such that for every integer $N\geq N_0$,   interval $J\subset [N]$, and $a\colon \Z_N\to \C$ with $|a|\leq 1$, the following implication holds:
$$
\text{if }\ \norm a_{U^s(\Z_N)} \leq\delta, \ \text{ then }\ \norm{\one_{J}\cdot a}_{U^s(\Z_N)}\leq \ve.
$$
\end{lemma}
\begin{proof}
Without loss, we can assume that $0<\norm a_{U^s(\Z_N)}<\frac{1}{4}$ and that the  length of
$J$ is an even number, say $2L$.  Furthermore, since the  $U^s(\Z_N)$-norm is invariant under translations we can assume that $J=[2L]$.

Let $l$ be an integer with $0<l< L$  that will be defined later.
Let $\phi:=\phi(l,L)$ be a ``trapezoid function'' on $\Z_N$ that increases linearly from $0$ to $1$ on the interval $[l]$, is  equal to $1$ between $l$ and $2L-l$, decreases linearly from $1$ to $0$ between $2L-l$ and $2L$, and is equal to $0$ between $2L$ and $N$.  This function is a variant of the function used in de la Vall\'ee-Poussin sums. Indeed, let
$\phi_1$ and $\phi_2$ be  the ``triangle functions'' of  height $1$ and of base $[0,2L]$  and $[l,2L-l]$, respectively. These functions are images under some translation of classical Fejer kernels on $\Z_N$ and thus $\norm{\wh\phi_1}_{l^1(\Z_N)}=\norm{\wh \phi_2}_{l^1(\Z_N)}=1$. Furthermore, for $n\in\Z_N$ we have
$$
\phi(n)=\frac L l\phi_1(n)-\frac {L-l}l\phi_2(n)
$$
and thus
\begin{equation}\label{E:2Ll}
\norm{\wh \phi}_{l^1(\Z_N)}\leq \frac{2L}{l}.
\end{equation}

Since the $U^s(\Z_N)$-norm is invariant under multiplication by $ \e(n\xi/N)$ for $s\geq 2$, using the triangle inequality
for the $U^s(\Z_N)$-norm and \eqref{E:2Ll} we get
\begin{equation}\label{E:phif1}
\norm{\phi\cdot a}_{U^s(\Z_N)}\leq \frac{2L}{l} \, \norm a_{U^s(\Z_N)}.
\end{equation}
Furthermore, since $\one_{[2L]}-\phi$ is supported on an interval of length $2l$ and is bounded by $1$, it follows that
\begin{equation}\label{E:phif2}
\norm{\one_{[2L]}\cdot a-\phi\cdot a}_{U^s(\Z_N)}
\leq
\Big(\frac{2l}{ N} \Big)^{2^{-s}}.
\end{equation}
Using \eqref{E:phif1}, \eqref{E:phif2}, and the triangle inequality for the $U^s(\Z_N)$-norm, we get
\begin{equation}\label{E:phif3}
\norm{\one_{[2L]}\cdot a}_{U^s(\Z_N)}\leq \frac{2L}{l}\, \norm a_{U^s(\Z_N)}+
\Big(\frac{2l}{ N}\Big)^{2^{-s}}.
\end{equation}
We choose
$l:=\lfloor 2L\cdot \norm a_{U^s(\Z_N)}^{2^s/(2^s+1)}\rfloor +1$. Since
$\norm a_{U^s(\Z_N)}<\frac{1}{4}$,  we get  $1\leq l\leq L$,
and \eqref{E:phif3}    together with the estimate $l\leq N \norm a_{U^s(\Z_N)}^{2^s/(2^s+1)}+1$ give the bound
$$
\norm{\one_{[2L]}\cdot a}_{U^s(\Z_N)}\leq 3 \norm a_{U^s(\Z_N)}^{1/(2^s+1)}+2N^{-2^{-s}}.
$$
The asserted result follows at once from this estimate.
\end{proof}

\subsection{Relations between the norms $U^s(\Z_N)$ and  $U^s[N]$} Our next goal is to show that the
 $U^s(\Z_N)$ and  $U^s[N]$ norms (both defined in Section~\ref{SS:Gowers}) are equivalent measures of randomness. We make this  precise in Lemma~\ref{lem:NormsUs} and Proposition~\ref{prop:UsN}.
We start with two preliminary lemmas.
\begin{lemma}
\label{cl:gowersMN}
Let $s,N, N^*\in\N$ with $s\geq 2$ and $N^*\geq N$ and let  $J\subset\Z_N$ be an interval of length smaller than $N/2$. Then, for every function $a\colon [N]\to \C$   we have
$$
\norm{\one_J\cdot a}_{U^s(\Z_N)}=\Bigl( \frac {N^*} N\Bigr)^{(s+1)/2^s}\norm{\one_J \cdot a}_{U^s(\Z_{N^*})}.
$$
\end{lemma}
\begin{proof}
The proof goes by induction on $s$. The result is obvious for $s=1$. Suppose that the result holds for $s\geq 1$; we are going to show that it holds for $s+1$.
Substituting $\one_J\cdot a$ for $a$, we can (and will) assume henceforth that $a$ vanishes outside $J$.  Since  the Gowers norms
are invariant under translation, after shifting the interval $J$ to the left we can assume  that $J=[L]$ for some integer $L$ with $0<L\leq N/2$.

For convenience, we  identify $\Z_N$ and
$\Z_{N^*}$ with the intervals $I_N:=[-\lceil N/2\rceil, \lfloor N/2\rfloor) $ and  $I_N^*:=\bigl[-\lceil N^*/2\rceil, \lfloor N^*/2\rfloor\bigr)$ respectively.
For $t\in\Z_N$ we let $a_t\colon \Z_N\to \C$ be defined by
$a_t(n):=a(t+n \bmod N)$, and  for  $t\in\Z_N^*$ we let $a_t\colon \Z_N\to \C$ be defined by
$a_t(n):=a(t+n \bmod{ N^*})$.
Keeping in mind that the function $a$ vanishes outside $[L]$ we see that the following properties hold:
\begin{enumerate}
\item \label{it:a1}
If $t\in I_N^*$ and $t\notin I_N$, then the function $a\overline{a_t^*}$ is identically zero.
\item \label{it:a2}
 If $t\in I_N$ and  $|t|\geq L$, then the  functions $a\overline{a_t^*}$ and  $a\overline{a_t}$ vanish.
\item \label{it:a3}
If $|t|<L$, the functions $a\overline{a_t^*}$  and $a\overline{a_t}$ vanish outside
$[L]$ and coincide for $n\in[L]$.
\end{enumerate}

Therefore,
$$
\norm a_{U^{s+1}(\Z_N)}^{2^{s+1}}
=\E_{t\in I_N}\norm{a\, \overline{a_t}}_{U^s(\Z_N)}^{2^s}
=\frac 1{N}\sum_{|t|<L}\norm{a\, \overline{a_t}}_{U^s(\Z_N)}^{2^s}
$$
where the last equality follows from  Property~\eqref{it:a2}.
Using Property~\eqref{it:a3} and  the induction hypothesis, we see that the last quantity is equal to
$$
\frac 1{N}\sum_{|t|<L}\norm{a\, \overline{a_t^*}}_{U^s(\Z_N)}^{2^s}=\frac 1{N}\sum_{|t|<L}\Bigl(\Bigl(\frac {N^*}N\Bigr)^{(s+1)/2^s}\norm{a\, \overline{a^*_t}}_{U^s(\Z_{N^*})}\Bigr)^{2^s},
$$
which in turn,  by Properties~\eqref{it:a1} and~\eqref{it:a2},   is equal to
$$\Bigl(\frac {N^*}N\Bigr)^{s+2}\E_{t\in I_{N^*}}\norm{a\, \overline{a_t^*}}_{U^s(\Z_{N^*})}
=\Bigl(\frac {N^*}N\Bigr)^{s+2}\norm a_{U^{s+1}(\Z_{N^*})}^{2^{s+1}}.
$$
This completes the induction and the proof.
\end{proof}

\begin{lemma}
\label{L:norm-N}
If $N^*\geq N$, then for every $s\geq 2$ we have
$$
\norm{\one_{[N]}}_{U^s(\Z_{N^*})}\geq \frac N{N^*}.
$$
\end{lemma}
\begin{proof}
 By the monotonicity  property \eqref{E:UkIncreases} we have
$$
\norm{\one_{[N]}}_{U^s(\Z_{N^*})}\geq \norm{\one_{[N]}}_{U^1(\Z_{N^*})}=
\frac N{N^*}
$$
as required.\end{proof}

%

\begin{lemma}
\label{lem:NormsUs}
Let $s\geq 2$ be an integer and $\ve>0$. There exists
$\delta:=\delta(s,\ve)>0$ and $N_0:=N_0(s,\ve)>0$  such that for every integer $N\geq N_0$ and every function $a\colon [N]\to \C$ with $|a|\leq 1$ we have
\begin{gather*}
\text{if }\norm a_{U^s[N]}\leq\delta\text{ then }\norm a_{U^s(\Z_N)}\leq\ve;\\
\text{if }\norm a_{U^s(\Z_N)}\leq\delta\text{ then }\norm a_{U^s[N]}\leq\ve.
\end{gather*}
\end{lemma}

\begin{proof}
Let  $\delta:= \delta(s,\ve/9 )$, $N_0:= N_0(s,\ve/9 )$ be defined as in Lemma~\ref{lem:UsLN}. Let $N\geq N_0$ be an integer and  $a\colon [N]\to \C$ be a function with $|a|\leq 1$.

Suppose  first that $\norm a_{U^s[N]}\leq\delta$.
 By Lemma~\ref{cl:gowersMN} and the definition of the $U^s[N]$-norms we have
$$
\norm {\one_{[N]}\cdot a}_{U^s(\Z_{3N})}= \norm{\one_{[N]}}_{U^s(\Z_{3N})}\cdot \norm a_{U^s[N]}\leq \norm a_{U^s[N]}\leq \delta.
$$
We partition the interval $[N]$ into three intervals of length less than $N/2$. If $J$ is any of these intervals, by  Lemma~\ref{lem:UsLN} applied to the function $\one_{[N]}\cdot a$ and the definition of $\delta$ we have $\norm{\one_J \cdot a}_{U^s(\Z_{3N})}\leq\ve/9$.
By Lemma~\ref{cl:gowersMN} we have
\begin{equation}\label{E:UsJ}
\norm{\one_J \cdot a}_{U^s(\Z_N)}\leq 3^{(s+1)/2^s}\ve/9\leq\ve/3.
\end{equation}
Taking the sum of these estimates for the three intervals $J$ that  partition of $[N]$ we get $\norm a_{U^s(\Z_N)}\leq\ve$.

Suppose now that $\norm a_{U^s(\Z_N)}\leq\delta$. As above,
we  partition the interval $[N]$ into three intervals of length less than $N/2$. If $J$ is any of these intervals, by  Lemma~\ref{lem:UsLN} and  the definition of $\delta$ we have
$\norm{\one_J \cdot a}_{U^s(\Z_N)}\leq \ve/9$. Hence, by Lemma~\ref{cl:gowersMN}
we have $\norm{\one_J \cdot a}_{U^s(\Z_{3N})}=3^{-(s+1)/2^s}\norm{\one_J \cdot a}_{U^s(\Z_N)}\leq \ve/9$.
Lastly, note that the definition of the  $U^s[N]$-norm and Lemma~\ref{L:norm-N} give
$\norm{\one_J \cdot a}_{U^s[N]}\leq \ve/3$. Taking the sum of these estimates for the three intervals $J$ that partition  $[N]$ we deduce that $\norm a_{U^s[N]}\leq\ve$. This completes the proof.
\end{proof}

\begin{proposition}
\label{prop:UsN}
Let $s\geq 2$ and $a\colon\N\to \C$ be bounded. Then the following properties are equivalent:
\begin{enumerate}
\item \label{it:normintervals}
$\norm a_{U^s[N]}\to 0$ as $N\to+\infty$;
\item\label{it:normcyclic}
$\norm a_{U^s(\Z_N)}\to 0$ as $N\to+\infty$;
\item
\label{it:normsoussuite}
there exists $C>1$ and a sequence $(N_j)$ of integers
with $N_j<N_{j+1}\leq CN_j$ for every $j\in \N$ such that
$$
\norm a_{U^s[N_j]}\to 0\text{ as }j\to+\infty.
$$
\end{enumerate}
\end{proposition}
\begin{proof}
The equivalence between  \eqref{it:normintervals} and \eqref{it:normcyclic} is given by  Lemma~\ref{lem:NormsUs}.

The implication  \eqref{it:normintervals} $\implies$ \eqref{it:normsoussuite} is obvious.
 We show that
 \eqref{it:normsoussuite} $\implies$ \eqref{it:normintervals}.
 Let $(N_j)$ and $C$ be as in the statement. For given  $\ve>0$
let $\delta:=\delta(s,\ve/3C)$ and $N_0:=N_0(s,\ve/3C)$ be given by Lemma~\ref{lem:UsLN}.
 Let $j_0$ be such that  $\norm a_{U^s[N_j]}\leq\delta$ for $j\geq j_0$.

Let $N\geq \max\{N_{j_0},N_0\}$ be an integer. Let $j$ be the smallest integer such that $N_j\geq N$. By hypothesis, $j\geq j_0$ and $N_j\leq CN$. Let $\wt N_j:=3N_j$.
By  Lemma~\ref{cl:gowersMN} we have
$$
\norm {1_{[N_j]}a}_{U^s(\Z_{\wt N_j})}=\norm{\one_{[N_j]}}_{U^s(\Z_{\wt N_j})}\cdot\norm a_{U^s[N_j]}
\leq \norm a_{U^s[N_j]}\leq\delta.
$$
Thus, by the definition of $\delta$, we have $\norm{\one_{[N]} \cdot a}_{U^s(\Z_{\wt N_j})}\leq\ve/3C$.
Combining this and   Lemmas~\ref{cl:gowersMN} and \ref{L:norm-N} we get
$$
\norm a_{U^s[N]}=\norm{\one_{[N]}}_{U^s(\Z_{\wt N_j})}\inv\cdot  \norm{\one_{[N]}\cdot  a}_{U^s(\Z_{\wt N_j})}
\leq \frac{\wt N_j}N \cdot \norm{\one_{[N]}\cdot  a}_{U^s(\Z_{\wt N_j})}
\leq \frac{\wt N_j}N \cdot\frac{\ve}{3C}\leq\ve.
$$
Hence, $\limsup_{N\to +\infty} \norm a_{U^s[N]}\leq\ve$.  As $\ve$ is arbitrary, we get~\eqref{it:normintervals}, completing the proof.
\end{proof}

\subsection{Some estimates involving Gowers norms}
We record here  two easy  estimates that were used  in the main text.
\begin{lemma}
\label{lem:Us-intels}
 There exists a constant $C>0$ such that for every prime number   $N$, function $a\colon \Z_N\to \C$, and arithmetic progression $P$ contained in the
interval $[N]$, we have
$$
\label{eq:normU1}
 \big|\E_{n\in [N]}\one_P(n)\cdot
 a(n) \big|\leq
 C\norm{a}_{U^2(\Z_N)}.
$$
\end{lemma}
\begin{proof}
If $N$ is a prime, since the $U^2(\Z_N)$-norm of a function on $\Z_N$ is invariant under any change of variables of the form $x\mapsto kx+l \bmod N$, where $k,l\in\Z_N$ with $k\neq 0 \bmod{\Z_N}$, we  can   reduce matters to the case where
 $P=\{0,\dots,m\}$ for some  $m\in \{0,\ldots, N-1\}$, considered as a subset of $\Z_N$.
In this case, a direct computation shows that
$$
 |\widehat{\one_P}(\xi)|\leq
\frac{2}{N||\xi/N||}= \frac{2}{\min\{\xi,N-\xi\}}  \quad \text{for
}\ \xi=1,\ldots,N-1,
$$
and as a consequence
$$
\norm{\widehat{\one_P}(\xi)}_{l^{4/3}(\Z_N)}
\leq C
$$
for some universal constant $C$. Using this estimate,  Parseval's
identity,   H\"older's inequality, and identity \eqref{eq:U2Fourier},
we deduce that
$$
\big|\E_{n\in [N]}\one_P(n)\cdot a(n) \big|=\big|\sum_{\xi\in [N]}\widehat \one_P(\xi)\cdot \widehat a(\xi) \big|\leq
C\cdot\Bigl(\sum_{\xi\in[N]} |\widehat a(\xi)|^4\Bigr)^{1/4}=
C\,\norm a_{U^2(\Z_N)}.\qed
$$
\renewcommand{\qed}{}
\end{proof}
\begin{lemma}
\label{lem:U2ent}
There exists a constant $C>0$ such that for every $N\in\N$ and function $a\colon\Z_N\to \C$ we have
$$
\sup_{t\in\R} \bigl|\E_{n\in[N]} a(n)\, \e(nt)\bigr|\leq C\norm a_{U^2(\Z_N)}.
$$
\end{lemma}
\begin{proof}
Writing $\phi_t(n):=\e(nt)$,  a direct computation gives that
$\norm{\wh{\phi_t}}_{l^{4/3}(\Z_N)}\leq C$ for some universal constant $C$, and the result follows as above from Parseval's
identity,   H\"older's inequality, and identity \eqref{eq:U2Fourier}.
\end{proof}

\section{Rational elements in a  nilmanifold}
\label{ap:A}
We collect here some properties of rational elements and rational subgroups.
Additional relevant  material can be found in \cite{GT12a} and in \cite{L06}.

Let $X:=G/\Gamma$ be an $s$-step nilmanifold of dimension $m$. As
everywhere in this article we assume that $G$ is connected and
simply connected, and endowed with a Mal'cev basis. Recall that we
write $e_X$ for the image in $X$ of the unit element  $\one_G$ of
$G$.

From  Properties~\eqref{it:mal3} and \eqref{it:MalcevGamma} of Mal'cev bases stated in
Section~\ref{subsec:nilmanifolds}, we immediately deduce:
\begin{lemma}
\label{lem:finite-Gamma}
 The group $\Gamma$ is finitely generated.
\end{lemma}

\subsection{Rational elements}
We recall that an element $g\in G$ is $Q$-rational if
$g^n\in\Gamma$ for some  $n\in \N$ with $ n\leq Q$.
  We note that all quantities  introduced  below depend implicitly on the nilmanifold $X$.

\begin{lemma}[{\cite[Lemma A.11]{GT12a}}]
\label{lem:rational}
\begin{enumerate}
\item
\label{it:rat1} For every $Q\in\N$ there exists $Q'\in\N$ such that
the product of any two $Q$-rational elements is $Q'$-rational; it
follows  that the set of rational elements is a subgroup of $G$.

\item
\label{it:rat2}
For every $Q\in \N$ there exists $Q'\in\N$ such that the Mal'cev coordinates of any $Q$-rational element are rational with denominators at most $Q'$; it follows that the set of $Q$-rational elements is a discrete subset of $G$.
\item
\label{it:rat3}
Conversely, for every $Q'\in\N$ there exists $Q\in\N$ such that, if the  Mal'cev coordinates of $g\in G$  are rational with denominators at most $Q'$, then $g$ is $Q$-rational.
\end{enumerate}
\end{lemma}

\begin{corollary}
\label{cor:M-rat}
For every $Q\in\N$  there exists a finite set $\Sigma:=\Sigma(Q)$ of $Q$-rational
 elements such that  all $Q$-rational elements belong to $\Sigma(Q)\Gamma$.
\end{corollary}

\begin{proof}
Let $K$ be a compact subset of $G$ such that $G=K\Gamma$.

Let $Q\in\N$. Let $Q'$ be associated to $Q$ by Part~\eqref{it:rat1}
of Lemma~\ref{lem:rational}, and let $\Sigma'$ be the set of
$Q'$-rational elements of $K$.  By Part~\eqref{it:rat2} of
Lemma~\ref{lem:rational}, $\Sigma'$ is finite. Let $g$ be a
$Q$-rational element of $G$. There exists $\gamma\in\Gamma$ such
that $g\gamma\inv\in K$.  Since $\gamma$ is obviously $Q$-rational,
$g\gamma\inv$ is $Q'$-rational and thus it  belongs to $\Sigma'$. For
each element $h$ of $\Sigma'$ obtained this way we choose a
$Q$-rational point $g$ such that $h\in g\Gamma$. Let
$\Sigma:=\Sigma(Q)$ be the set consisting of all elements obtained
this way. Thus,  $\Sigma\Gamma$ contains all $Q$-rational
elements.  Furthermore,  $|\Sigma|\leq|\Sigma'|$ and so $\Sigma$ is
finite, completing the proof.
 \end{proof}

\subsection{Rational subgroups} We gather here some basic properties of rational
subgroups that we use in the main part of the article.

Recall that
a \emph{rational subgroup} $G'$ of $G$ is a closed, connected, and simply connected subgroup
of $G$ such that $\Gamma':=\Gamma\cap G'$ is
co-compact in $G$. In this case, $G'/\Gamma'$ is  called a
\emph{sub-nilmanifold} of $X$. It can be shown that $G'$ is a rational subgroup of $G$ if and only if
 its Lie algebra $\mathfrak g'$ admits a base
that has  rational coordinates in the Mal'cev basis of $G$.

\begin{lemma}[\mbox{\cite[Lemma A.13]{GT12a}}]
\label{lem:Ap0} If $G'$ is a rational subgroup of $G$ and $h$ is a
rational element, then $hG'h\inv$ is a rational subgroup of $G$.
\end{lemma}

\begin{proof}
The conjugacy map $h\mapsto g\inv hg$ is a polynomial map with
rational coefficients and thus the linear map $\mathrm{Ad}_h$ from
$\mathfrak g$ to itself has rational coefficients. Since $\mathfrak
g'$ has a base consisting of vectors with rational coefficients, the
same property holds for $\mathrm{Ad}_h\mathfrak g$, that is, for the
Lie algebra of $hG'h\inv$. This proves the claim.
\end{proof}

 The argument used to deduce  Lemma~\ref{lem:finite-Gamma} shows that
  the group $\Gamma\cap
(hG'h\inv)$ is finitely generated.

  We also need an  auxiliary  result.
\begin{lemma}
\label{lem:small-lemma}
Let $\Theta$ be a finitely generated nilpotent group and let $\Lambda$ be a subgroup of $\Theta$. Suppose that for every $\gamma\in \Theta$ there exists $n\in\N$ with $\gamma^n\in \Lambda$. Then $\Lambda$ has finite index in $\Theta$.
\end{lemma}
\begin{proof}
The proof goes by induction on the nilpotency degree $s$ of $\Theta$. If $s=1$, $\Theta$ is Abelian and the result is immediate. Suppose that $s>1$ and that the result holds for $(s-1)$-step nilpotent groups. By the induction hypothesis applied to the Abelian group $\Theta/\Theta_2$, the subgroup
$(\Lambda\Theta_2)/\Theta_2$ has finite index in $\Theta/\Theta_2$ and thus $\Lambda\Theta_2$ has finite index in $\Theta$. If $\gamma\in\Theta_2$, then there exists $n\in \N$ with $\gamma^n\in\Lambda$, hence $\gamma^n\in\Lambda\cap\Theta_2$.
Since $\Theta_2$ is a finitely generated $(s-1)$-step nilpotent group,  by the induction hypothesis again, $\Lambda\cap\Theta_2$ has a finite index in $\Theta_2$. Thus, $\Lambda$ has a finite index in $\Lambda\Theta_2$ which has finite index in $\Theta$.
The result follows.
\end{proof}

\begin{lemma}[\mbox{\cite[Theorem~5.29]{Co82}}]
\label{lem:finite-index}
Let $X:=G/\Gamma$ be an $s$-step
nilmanifold, $G'\subset G$ be a rational subgroup, $g\in G$ be a
rational element   and $\Lambda:=\Gamma\cap (g\inv \Gamma g)\cap G'$. Then
\begin{enumerate}
\item
\label{it:rat10bis}
$\Lambda$ is a subgroup of finite index  of
$\Gamma\cap G'$;
\item
\label{it:rat11bis}
 $\Lambda$ is a subgroup of finite index  of $(g\inv \Gamma g)\cap G'$.
\end{enumerate}
\end{lemma}
\begin{proof}
By Part~\eqref{it:rat1} of Lemma~\ref{lem:rational},  all  elements
of $g\Gamma g\inv$  are rational.
 Hence, if $\gamma\in \Gamma\cap G'$, then  there exists $n\in\N$ with
$(g\gamma g\inv)^n\in \Gamma$ and so we have $\gamma^n\in \Lambda$.
Applying Lemma~\ref{lem:finite-Gamma} to $G'$ and $\Gamma\cap G'$, we get the group $\Lambda$ is finitely generated.  By  Lemma~\ref{lem:small-lemma}, $\Lambda$ has finite index in $\Gamma\cap G'$. This proves \eqref{it:rat10bis}.
Since $g\Gamma g\inv$  is co-compact in $G$,
substituting this group for $G$ and $g\inv$ for $g$ in the preceding
statement, we get  \eqref{it:rat11bis}.
\end{proof}

\begin{lemma}
\label{lem:Gprimey} Let $g\in G$ be a rational element and  $G'$ a
rational subgroup of $G$. Then $G'g\cdot e_X:=\{hg\cdot e_X\colon
h\in G'\}$ is a closed sub-nilmanifold of $X$.
\end{lemma}
\begin{proof}
By Lemma~\ref{lem:Ap0}, $g\inv G'g$ is a rational subgroup of $G$.
Therefore, $\Gamma\cap (g\inv G'g)$ is co-compact in $g\inv G'g$ and
thus $(g\Gamma g\inv)\cap G'$ is co-compact in $G'$. Note that  $(g\Gamma
g\inv)\cap G'$ is the stabilizer $\{h\in G'\colon hg\cdot
e_X=g\cdot e_X\}$ of $g\cdot e_X$ in $G'$ and thus the orbit
$G'g\cdot e_X$ is compact and can be identified with the
nilmanifold $G'/((g\Gamma g\inv)\cap G')$.
\end{proof}

\section{Zeros of some homogeneous quadratic forms} \label{SS:AppNumberTheory}
 We prove Proposition~\ref{prop:linearfactors}. We recall  the statement for reader's convenience.
\begin{proposition*}
Let the quadratic form $p$ satisfy the hypothesis of
Theorem~\ref{th:partition-regular1}.
Then there exist
admissible integers $\ell_0,\ldots,\ell_4$ (see definition in Section~\ref{SS:parametric}),
 such that for every $k,m,n\in \Z$,  the
integers $x:=k\ell_0 (m+\ell_1n)(m+\ell_2n)$ and $y:=
k\ell_0(m+\ell_3n)(m+\ell_4n)$ satisfy the equation $p(x,y,z)=0$ for
some $z\in \Z$.
\end{proposition*}

\begin{proof}
Let
\begin{equation}\label{E:Q2'}
ax^2+by^2+cz^2+dxy+exz+fyz=0
\end{equation}
be the equation we are interested in solving. Recall  that by
assumption $a,b,c$ are non-zero integers and that  all three
integers
$$
\Delta_1:=e^2-4ac,  \quad \Delta_2:=f^2-4bc,  \quad
\Delta_3:=(e+f)^2-4c(a+b+d)
$$
are non-zero squares.

\subsection*{Step 1} We first reduce to the case where $e=f=0$.
Let
$$
p'(x,y,z):= p(2cx,2cy,z-ex-fy).
$$
Then
$$
p'(x,y,z)=c(4ac-e^2)x^2+c(4bc-f^2)y^2+cz^2+2c(2cd-ef)xy.
$$
The  coefficients of $x^2$, $y^2$, $z^2$ in the quadratic form  $p'$  are non-zero by hypothesis. The discriminants of  the three quadratic forms $p'(x,0,z)$, $p'(0,y,z)$, $p'(x,x,z)$
are equal to  $4c^2\Delta_1$, $4c^2\Delta_2$,  $4c^2\Delta_3$ respectively, and thus are non-zero squares
by hypothesis.
Suppose that  the announced result holds for the quadratic form $p'$. Then there exist admissible integers
 $\ell_0,\ldots,\ell_4$,
 such that for every $k,m,n\in \Z$,  the
integers $x':=k\ell_0 (m+\ell_1n)(m+\ell_2n)$ and $y':=
k\ell_0(m+\ell_3n)(m+\ell_4n)$ satisfy the equation $p'(x',y',z')=0$ for
some $z'\in \Z$. It follows that  $x:=2ck\ell_0 (m+\ell_1n)(m+\ell_2n)$ and $y:=
2ck\ell_0(m+\ell_3n)(m+\ell_4n)$ satisfy the equation $p(x,y,z)=0$ for
$z:=z'-ex'-fy'$.  If $c>0$ we are done, if $c<0$ we consider the solution  $-x,-y,-z$.

\subsection*{Step 2} We consider now the case where $e=f=0$. Then
$$
p(x,y,z)=ax^2+by^2+cz^2+dxy.
$$
 Our hypothesis is that $a,b,c$ are non-zero  and the integers $-ac$, $-bc$, $-c(a+b+d)$ are non-zero squares.
 Without loss, we can restrict to the case where $a> 0$ and thus $c<0$.
By taking   products  we get that $ac^2(a+b+d)$ and $bc^2(a+b+d)$ are non-zero squares, and thus
$a(a+b+d)$ and $b(a+b+d)$ are non-zero squares. We let
$$
\Delta_1':=\sqrt{b(a+b+d)}\ ; \ \Delta_2':=\sqrt{a(a+b+d)} \ ;\ \Delta_3':=\sqrt{-c(a+b+d)},
$$
and
\begin{align*}
\ell_0 &:=-c; &\ell_1 & :=-(b+\Delta_1') ;  & \ell_2 & := -(b-\Delta_1');\\
&&\ell_3 & :=-(a+d+\Delta_2')\ ; & \ell_4 & := -(a+d-\Delta_2').
\end{align*}
By direct computation, we check that for every $k,m,n\in \Z$,  the integers $x,y,z$ given by
\begin{align*}
x &:=  -kc\bigl(m^2-2bmn-b(a+d)n^2\bigr)
=k\ell_0(m+\ell_1n)(m+\ell_2n);\\
y &:=   -kc\bigl(m^2+2(a+d)mn+(ad+d^2-ab)n^2\bigr)
=-k\ell_0(m+\ell_3n)(m+\ell_4n);\\
z&:=\ \  k\Delta_3'(m^2+dmn+abn^2) .
\end{align*}
satisfy $p(x,y,z)=0$.

Since $c<0$ we have
$\ell_0>0$. Furthermore, since  $\Delta_1\neq 0$ we have $\ell_1\neq\ell_2$,  and since $\Delta_2\neq 0$ we have $\ell_3\neq\ell_4$. Lastly, we verify that
$\{\ell_1,\ell_2\}\neq \{\ell_3,\ell_4\}$. Indeed, if these pairs were identical, then
the coefficients of $mn$ in $x$ and $y$ would be the same. Hence, $-2b=2(a+d)$ and  thus $a+b+d=0$, contradicting our hypothesis. We conclude that  the integers $\ell_0, \dots,\ell_4$ are admissible. This completes the proof.
\end{proof}

\end{document}